\documentclass[a4paper,reqno]{amsart}

\usepackage{amsthm,amsfonts,amssymb,amsmath,amsxtra,mathtools}
\usepackage[all]{xy}
\SelectTips{cm}{}
\usepackage{tikz}
\usepackage{verbatim}
\usepackage{xcolor}
\usepackage[margin=1.2in]{geometry}
\usepackage{mathrsfs}

\RequirePackage{xspace}
\RequirePackage{etoolbox}
\RequirePackage{varwidth}
\RequirePackage{enumitem}
\RequirePackage{tensor}
\RequirePackage{mathtools}
\RequirePackage{longtable}
\RequirePackage{multirow}
\RequirePackage{makecell}
\RequirePackage{bigints}

\usepackage[bookmarks = true, pagebackref]{hyperref}
\hypersetup{  
	colorlinks=true,
	linkcolor=blue,
	citecolor=magenta,
	urlcolor=cyan,
}

\usepackage[alphabetic]{amsrefs}

\newcommand{\fkA}{\ensuremath{\mathfrak{A}}\xspace}
\newcommand{\fkB}{\ensuremath{\mathfrak{B}}\xspace}
\newcommand{\fkC}{\ensuremath{\mathfrak{C}}\xspace}
\newcommand{\fkD}{\ensuremath{\mathfrak{D}}\xspace}

\newcommand{\fkL}{\ensuremath{\mathfrak{L}}\xspace}
\newcommand{\fkM}{\ensuremath{\mathfrak{M}}\xspace}

\newcommand{\fkX}{\ensuremath{\mathfrak{X}}\xspace}

\newcommand{\BA}{\ensuremath{\mathbb{A}}\xspace}

\newcommand{\BD}{\ensuremath{\mathbb{D}}\xspace}
\newcommand{\BE}{\ensuremath{\mathbb{E}}\xspace}
\newcommand{\BF}{\ensuremath{\mathbb{F}}\xspace}

\newcommand{\BP}{\ensuremath{\mathbb{P}}\xspace}
\newcommand{\BQ}{\ensuremath{\mathbb{Q}}\xspace}

\newcommand{\BV}{\ensuremath{\mathbb{V}}\xspace}

\newcommand{\BX}{\ensuremath{\mathbb{X}}\xspace}

\newcommand{\BZ}{\ensuremath{\mathbb{Z}}\xspace}

\newcommand{\ba}{\ensuremath{\mathbf{a}}\xspace}
\newcommand{\bb}{\ensuremath{\mathbf{b}}\xspace}

\newcommand{\bbs}{\ensuremath{\mathbf{s}}\xspace}
\newcommand{\bt}{\ensuremath{\mathbf{t}}\xspace}
\newcommand{\bu}{\ensuremath{\mathbf{u}}\xspace}
\newcommand{\bv}{\ensuremath{\mathbf{v}}\xspace}
\newcommand{\bw}{\ensuremath{\mathbf{w}}\xspace}

\newcommand{\blambda}{\boldsymbol{\lambda}}
\newcommand{\bmu}{\boldsymbol{\mu}}

\newcommand{\CE}{\ensuremath{\mathcal{E}}\xspace}
\newcommand{\CF}{\ensuremath{\mathcal{F}}\xspace}
\newcommand{\CG}{\ensuremath{\mathcal{G}}\xspace}

\newcommand{\CL}{\ensuremath{\mathcal{L}}\xspace}

\newcommand{\CN}{\ensuremath{\mathcal{N}}\xspace}
\newcommand{\CO}{\ensuremath{\mathcal{O}}\xspace}

\newcommand{\CX}{\ensuremath{\mathcal{X}}\xspace}
\newcommand{\CY}{\ensuremath{\mathcal{Y}}\xspace}
\newcommand{\CZ}{\ensuremath{\mathcal{Z}}\xspace}

\newcommand{\ad}{\mathrm{ad}}

\DeclareMathOperator{\End}{End}

\newcommand{\Fil}{\ensuremath{\mathrm{Fil}}\xspace}

\DeclareMathOperator{\Gal}{Gal}

\DeclareMathOperator{\Gr}{Gr}

\newcommand{\Exc}{\mathrm{Exc}}
\newcommand{\NExc}{\mathrm{NExc}}

\DeclareMathOperator{\Hom}{Hom}

\newcommand{\id}{\ensuremath{\mathrm{id}}\xspace}

\DeclareMathOperator{\Lie}{Lie}

\newcommand{\loc}{\ensuremath{\mathrm{loc}}\xspace}

\newcommand{\naive}{\ensuremath{\mathrm{naive}}\xspace}

\DeclareMathOperator{\proj}{Proj}

\DeclareMathOperator{\Spec}{Spec}
\DeclareMathOperator{\Spf}{Spf}

\newcommand{\spl}{\mathrm{spl}}
\newcommand{\nspl}{\mathrm{nspl}}

\DeclareMathOperator{\tr}{tr}

\newcommand{\U}{\mathrm{U}}

\newcommand{\wit}{\widetilde}

\newcommand{\ov}{\overline}

\newcommand{\lra}{\longrightarrow}
\newcommand{\lr}{\longrightarrow}

\newenvironment{altenumerate}
   {\begin{list}
      {(\theenumi) }
      {\usecounter{enumi}
       \setlength{\labelwidth}{0pt}
       \setlength{\labelsep}{0pt}
       \setlength{\leftmargin}{0pt}
       \setlength{\itemsep}{\the\smallskipamount}
       \renewcommand{\theenumi}{\roman{enumi}}
      }}
   {\end{list}}

\newenvironment{altitemize}
   {\begin{list}
      {$\bullet$}
      {\setlength{\labelwidth}{0pt}
	   \setlength{\itemindent}{5pt}
       \setlength{\labelsep}{5pt}
       \setlength{\leftmargin}{0pt}
       \setlength{\itemsep}{\the\smallskipamount}
      }}
   {\end{list}}

\renewcommand{\to}{%
   \ifbool{@display}{\longrightarrow}{\rightarrow}%
   }
\let\shortmapsto\mapsto
\renewcommand{\mapsto}{%
   \ifbool{@display}{\longmapsto}{\shortmapsto}%
   }
\newlength{\olen}
\newlength{\ulen}
\newlength{\xlen}
\newcommand{\xra}[2][]{%
   \ifbool{@display}%
      {\settowidth{\olen}{$\overset{#2}{\longrightarrow}$}%
       \settowidth{\ulen}{$\underset{#1}{\longrightarrow}$}%
       \settowidth{\xlen}{$\xrightarrow[#1]{#2}$}%
       \ifdimgreater{\olen}{\xlen}%
          {\underset{#1}{\overset{#2}{\longrightarrow}}}%
          {\ifdimgreater{\ulen}{\xlen}%
             {\underset{#1}{\overset{#2}{\longrightarrow}}}
             {\xrightarrow[#1]{#2}}}}%
      {\xrightarrow[#1]{#2}}
   }
\makeatother
\newcommand{\xyra}[2][]{%
   \settowidth{\xlen}{$\xrightarrow[#1]{#2}$}%
   \ifbool{@display}%
      {\settowidth{\olen}{$\overset{#2}{\longrightarrow}$}%
       \settowidth{\ulen}{$\underset{#1}{\longrightarrow}$}%
       \ifdimgreater{\olen}{\xlen}%
          {\mathrel{\xymatrix@M=.12ex@C=3.2ex{\ar[r]^-{#2}_-{#1} &}}}%
          {\ifdimgreater{\ulen}{\xlen}%
             {\mathrel{\xymatrix@M=.12ex@C=3.2ex{\ar[r]^-{#2}_-{#1} &}}}
             {\mathrel{\xymatrix@M=.12ex@C=\the\xlen{\ar[r]^-{#2}_-{#1} &}}}}}%
      {\mathrel{\xymatrix@M=.12ex@C=\the\xlen{\ar[r]^-{#2}_-{#1} &}}}%
   }
\makeatletter
\newcommand{\xla}[2][]{%
   \ifbool{@display}%
      {\settowidth{\olen}{$\overset{#2}{\longleftarrow}$}%
       \settowidth{\ulen}{$\underset{#1}{\longleftarrow}$}%
       \settowidth{\xlen}{$\xleftarrow[#1]{#2}$}%
       \ifdimgreater{\olen}{\xlen}%
          {\underset{#1}{\overset{#2}{\longleftarrow}}}%
          {\ifdimgreater{\ulen}{\xlen}%
             {\underset{#1}{\overset{#2}{\longleftarrow}}}
             {\xleftarrow[#1]{#2}}}}%
      {\xleftarrow[#1]{#2}}
   }
\newcommand{\isoarrow}{%
   \ifbool{@display}{\overset{\sim}{\longrightarrow}}{\xrightarrow\sim}%
   }
\renewcommand{\lra}{%
   \ifbool{@display}{\longleftrightarrow}{\leftrightarrow}%
   }
\newcommand{\undertilde}{\raisebox{0.4ex}{\smash[t]{$\scriptstyle\sim$}}}  

\newcommand{\wt}{\wit}

\DeclareFontFamily{U}{matha}{\hyphenchar\font45}
\DeclareFontShape{U}{matha}{m}{n}{
      <5> <6> <7> <8> <9> <10> gen * matha
      <10.95> matha10 <12> <14.4> <17.28> <20.74> <24.88> matha12
      }{}
\DeclareSymbolFont{matha}{U}{matha}{m}{n}
\DeclareFontFamily{U}{mathx}{\hyphenchar\font45}
\DeclareFontShape{U}{mathx}{m}{n}{
      <5> <6> <7> <8> <9> <10>
      <10.95> <12> <14.4> <17.28> <20.74> <24.88>
      mathx10
      }{}
\DeclareSymbolFont{mathx}{U}{mathx}{m}{n}

\DeclareMathSymbol{\obot}         {2}{matha}{"6B}

\newtheorem{theorem}[subsubsection]{Theorem}
\newtheorem{proposition}[subsubsection]{Proposition}
\newtheorem{lemma}[subsubsection]{Lemma}

\newtheorem{corollary}[subsubsection]{Corollary}

\theoremstyle{definition}
\newtheorem{definition}[subsubsection]{Definition}

\newtheorem{remark}[subsubsection]{Remark}
\newtheorem{remarks}[subsubsection]{Remarks}

\numberwithin{equation}{subsection}

\usepackage{xcolor}

\usepackage{tikz-cd}
\usepackage{mathdots}

\newcommand{\colspan}{\mathrm{colspan}}
\newcommand{\Span}{\mathrm{Span}}
\newcommand{\rig}{\mathrm{rig}}
\newcommand{\obst}{\mathrm{obst}}
\newcommand{\Kra}{\mathrm{Kra}}

\title[Regular models for ramified unitary group]{Regular models of ramified unitary Shimura varieties at maximal parahoric level}

\author{Qiao He}
\address{Department of Mathematics, Columbia University, 2990 Broadway, New York, NY 10027, USA}
\email{qh2275@columbia.edu } 
\author{Yu LUO}
\address{Massachusetts Institute of Technology, Department of Mathematics, 77 Massachusetts Avenue, Cambridge, MA 02139, USA}
\email{yuluo@mit.edu}
\author{Yousheng Shi}
\address{School of Mathematical Sciences, Zhejiang University, 866 Yuhangtang Rd, Hangzhou, 310058, P.R. China}
\email{0023140@zju.edu.cn}

 \date{\today}

\begin{document}

\begin{abstract}
We use the idea of splitting models to construct and study a semi-stable model for unitary Shimura varieties of signature $(n-1,1)$ with maximal parahoric level structure at ramified primes.
In this setting, the ``naive'' splitting model defined by Pappas and Rapoport fails to be flat in a crucial way. 
We show that the ``genuine'' splitting model in this case is flat with semi-stable reduction.
In addition, we study special cycles within these splitting models.
\end{abstract}

\maketitle{}
\tableofcontents{}

\section{Introduction}
\subsection{Motivation}
In the recent paper \cite{LUO2024}, the
second author studies integral models of unitary Shimura varieties of signature $(n-1,1)$ with any parahoric level structure at ramified primes.
The integral model described in loc.cit. is canonical, meaning it can be characterized using purely group-theoretical data.
A moduli interpretation of this ``canonical'' model was proposed through a conjecture of Smithling in \cite{Smithling2015}, which builds on the work of Pappas and Rapoport in \cite{PR2009}.

Although the ``canonical'' integral model possesses many nice properties such as flatness and normality, it is not regular in many important examples.
In this paper, we contribute to the construction of regular integral models of Shimura varieties. 
Our approach is inspired by the \emph{splitting model} first introduced by Pappas and Rapoport in \cite{PR2005}.
Roughly speaking, the splitting model is a certain modification of the canonical integral model  which ``refines'' the ramification information by adding additional data. Such construction has been studied in the context of Weil restriction \cite{PR2005,Bijakowski-Hernandez-2023,shen2023fzips}, triality groups \cite{zhao-thesis}, and ramified unitary groups \cite{Kramer2003,Bijakowski-Hernandez-2023,zachos2023semistable}.
Splitting models have found applications in various settings, including \cite{Reduzzi-Xiao,Sasaki,Diamond-Kassaei,BHKRY,HLSY,Bijakowski-Hernandez-2023,bijakowski2024geometry,BZZ25}, since the first draft of this paper, the splitting model studied here has been used in \cite{LRZ25} for the arithmetic transfer conjecture.

In this paper, we use similar ideas to define and study regular integral models for unitary Shimura varieties with maximal parahoric level at a ramified prime.  We give now more details in the following.

\subsection{Canonical local model}\label{sec:intro canonical local model}
It is well-known that local properties of integral models can often be studied via local models, which have the advantage of being defined in terms of linear algebra. 
We adopt this perspective in the present work: we first construct a regular local model and then, in the final section, indicate its implications for the integral models of Rapoport--Zink spaces.

We begin by describing the setup.
Let $F/F_0$ be a ramified quadratic extension of local fields with residue field $\BF_q$, such that $\mathrm{char}(\BF_q)\neq 2$. Let $\pi\in F$ and $\pi_0\in F_0$ be the uniformizers such that $\pi^2=\pi_0$. Consider a hermitian space $(V,\varphi)$ over $F$. We are interested in the local model associated to the unitary group $G=\U(V,\varphi)$.

A \emph{vertex lattice} is a hermitian lattice $\Lambda\subset V$ such that $\pi\Lambda^{\sharp}\subseteq\Lambda\subseteq\Lambda^{\sharp}$,
where $\Lambda^{\sharp}$ is the hermitian dual of $\Lambda$.
The quotient $\Lambda^{\sharp}/\Lambda$ is always of even dimension, say $2k$, called the \emph{type} of the vertex lattice. 
In the introduction, for simplicity, we assume $n\geq 3$ and take the signature $(r,s)=(n-1,1)$\footnote{when $n=2$, the same results apply, but the reflex field of the canonical local model is $F_0$.}, it corresponds to the geometric conjugacy class of a certain minuscule cocharacter of $G$, see \cite[\S 1.3]{PR2009}.
In \S \ref{sec:moduli}, we will also consider more general signatures.

The \emph{canonical local model} $M^{\loc}_{[k]}(n-1,1)$ is a projective scheme over $O_F$, representing the moduli functor that assigns to each $O_F$-algebra $R$ the set of filtrations of $O_F\otimes_{O_{F_0}}R$-modules
\begin{equation*}
\begin{aligned}
\xymatrix{
\Lambda\otimes_{O_{F_0}}R\ar[r]&\Lambda^\sharp\otimes_{O_{F_0}}R\ar[r]^-{\pi^{-1}}&\Lambda\otimes_{O_{F_0}}R\\
\CF_\Lambda\ar@{^(->}[u]\ar[r]&\CF_{\Lambda^\sharp}\ar@{^(->}[u]\ar[r]&\CF_\Lambda\ar@{^(->}[u]
}
\end{aligned},
\end{equation*}
where $\CF_\Lambda$ and $\CF_{\Lambda^{\sharp}}$ are required to be $R$-direct summands of rank $n$, satisfying some additional conditions, see \S \ref{sec:canonical-model} for a detailed discussion.
The geometry of ramified unitary local models has been studied by various authors, for example
\cite{Pappas2000,PR2009,Arzdorf,Smithling2015,Yu2019,LUO2024}.

We say that $k$ is \emph{$\pi$-modular} when $n=2m$ and $k=m$, see Remark \ref{rem:type-name} for the terminology used in other cases. The name reflects the fact that the corresponding vertex lattice is $\pi$-modular.

\begin{theorem}[\protect{\cite{LUO2024}}]\label{thm:Luo-main}Assume that $k$ is not $\pi$-modular.
\begin{altenumerate}
\item The local model $M^{\loc}_{[k]}(n-1,1)$ is flat, normal and Cohen-Macaulay of dimension $n$.
\item The special fiber $M^{\loc}_{[k],s}(n-1,1)$ is a reduced scheme whose irreducible components are all normal and Cohen-Macaulay. The intersection of $i$ distinct irreducible components has dimension $n-i$.
\item There exists a special point $*\in M^{\loc}_{[k],s}(n-1,1)$, referred to as the \emph{worst point}. It is the only singular point of each irreducible component of the special fiber and of their intersections.
\end{altenumerate}
\end{theorem}

\subsection{Splitting model}
We now turn to the splitting model.
Let $\Pi:=\pi\otimes 1$ and $\pi:=1\otimes \pi$ be the respective elements in $O_F\otimes_{O_{F_0}}R$.
The \emph{naive splitting model} $M^{\nspl}_{[k]}(n-1,1)$ is a projective scheme over $O_F$, which represents the moduli functor which sends each $O_F$-algebra $R$ to the set of filtrations of $O_F\otimes_{O_{F_0}}R$-modules 
\begin{equation}\label{equ:splitting-model}
\begin{aligned}
\xymatrix{
\Lambda\otimes_{O_{F_0}}R\ar[r]&\Lambda^\sharp\otimes_{O_{F_0}}R\ar[r]^-{\pi^{-1}}&\Lambda\otimes_{O_{F_0}}R\\
\CF_\Lambda\ar@{^(->}[u]\ar[r]&\CF_{\Lambda^\sharp}\ar@{^(->}[u]\ar[r]&\CF_\Lambda\ar@{^(->}[u]\\
\CF^0_\Lambda\ar@{^(->}[u]\ar[r]&\CF^0_{\Lambda^\sharp}\ar@{^(->}[u]\ar[r]&\CF^0_\Lambda\ar@{^(->}[u]
}
\end{aligned},
\end{equation}
such that $(\CF_\Lambda,\CF_{\Lambda^{\sharp}})$ defines an $R$-point of the local model $M^{\loc}_{[k]}(n-1,1)$, and for each $\Box\in\{\Lambda,\Lambda^\sharp\}$, $\CF^0_{\Box}$ is an $R$-direct summand of rank $1$ of $\CF_\Box$ which satisfies the Kr\"amer condition:
\begin{equation}\label{equ:Kramer-cond}
	(\Pi+\pi)(\CF_\Box)\subseteq\CF^0_\Box,
	\quad
	(\Pi-\pi)(\CF_\Box^0)=(0).
\end{equation}
We refer to this as the \emph{naive splitting model}, as it generally fails to be flat. 

The \emph{splitting model} $M^{\spl}_{[k]}(n-1,1)$ is defined as the scheme-theoretic closure of its generic fiber $M^{\nspl}_{[k],\eta}(n-1,1)$ in $M^{\nspl}_{[k]}(n-1,1)$.
By construction, we have a natural forgetful map 
\begin{equation*}
	\tau: M^{\spl}_{[k]}(n-1,1)\longrightarrow M^{\loc}_{[k]}(n-1,1).
\end{equation*}

Suppose now that the special fiber $M^{\loc}_{[k],s}(n-1,1)$ has $c=c(n,k)$ irreducible components, see Theorem \ref{thm:canonical-LM} for the precise value of $c$ in each case.
Denote the irreducible components as $Z_i$ for $1\leq i\leq c$.  The first main result of the paper is:
\begin{theorem}[cf. \S \ref{sec:affine-charts}]\label{thm:main}
Assume that $n\geq 3$ and that $k$ is not $\pi$-modular.
\begin{altenumerate}
\item The splitting model $M^{\spl}_{[k]}(n-1,1)$ is flat of dimension $n$, with semi-stable reduction.
\item 
The special fiber of the splitting model has $c+1$ irreducible components: 
$\widetilde{Z}_i$ for $1\leq i\leq c$ and $\Exc$. 
All components are smooth of dimension $n-1$. 
\item The forgetful map $\tau:M^{\spl}_{[k]}(n-1,1)\rightarrow M^{\loc}_{[k]}(n-1,1)$ maps $\widetilde{Z}_i$ to $Z_i$ for $1\leq i\leq c$, and $\Exc$ to the worst point $*$.
\end{altenumerate}	
\end{theorem}
It is straightforward to verify that $\tau$ is an isomorphism away from the worst point. As a consequence, we obtain the following:
\begin{corollary}
When $k$ is not $\pi$-modular, the open subscheme $M^{\loc}_{[k]}(n-1,1)\setminus \{*\}$ of the canonical local model has semi-stable reduction.
\end{corollary}
This leads one to ask whether the splitting model is in fact a blow-up of the canonical model.
Indeed, we managed to characterize the splitting model in this way:
\begin{theorem}[Thm. \ref{thm:blowup}] Assume $k$ is not $\pi$-modular.
The map $\tau:M^{\spl}_{[k]}(n-1,1)\rightarrow M^{\loc}_{[k]}(n-1,1)$ is the blow-up of $M^{\loc}_{[k]}(n-1,1)$  along the worst point. Under this identification, $\Exc$ is the exceptional divisor of the blow-up, and $\widetilde{Z}_i$ is the strict transform of $Z_i$ for $1\leq i\leq c$.
\end{theorem}

\begin{remark}\label{rem:pi-modular}
When $k$ is $\pi$-modular, the canonical local model is smooth and connected. 
In contrast to other cases, it does not exhibit an analogue of the worst point. Moreover, in this case, the splitting model coincides with the canonical local model, see also Remark \ref{rmk:pi-modular-iso}.
\end{remark}

Characterizing the splitting model as a blow-up certainly suggests that it is the ``canonical'' choice in some sense, see also \cite{Gora-thesis}.
It remains an interesting open question to justify this construction from a group-theoretic perspective.

\subsection{Signature conditions}
Our definition of $M^{\spl}_{[k]}(n-1,1)$ requires that the pair $(\CF_{\Lambda},\CF_{\Lambda^\sharp})$ defines an $R$-point of the canonical local model. In practice, characterizing the points of the canonical local model is a subtle problem; we refer the reader to the introduction of \cite{LUO2024} for further discussion.
In our setting, $M^{\loc}_{[k]}(n-1,1)$ admits a moduli description via the \emph{strengthened spin condition} introduced in \cite{Smithling2015}. However, this condition is intricate and often difficult to verify in applications. We prove the following:

\begin{theorem}[Thm. \ref{thm:Kramer-vs-ss}]\label{prop:Kramer-imply-ss}
When $(r,s)=(n-1,1)$ and $k$ is not $\pi$-modular, the Kr\"amer condition implies the strengthened spin condition. 
\end{theorem}

Therefore, we can drop the strengthened spin condition when defining the splitting model, see \S \ref{sec:intro special cycles} or \S \ref{sec:RZ}.
Consequently, we obtain an alternative definition of the naive splitting model, which proves useful in practice, see \S\ref{subsec:comments} below.
\begin{remarks}
If we drop the strengthened spin condition when defining the naive splitting model in the very beginning, it is not clear that we have a natural forgetful functor $\tau:M^{\nspl}_{[k]}(n-1,1)\rightarrow M^{\loc}_{[k]}(n-1,1)$.
But this is clear within the strengthened spin condition.
Also, it is not clear to us whether Proposition \ref{prop:Kramer-imply-ss} holds for general signature $(r,s)$. 
\end{remarks}

\subsection{Special cycles on Rapoport--Zink spaces}\label{sec:intro special cycles}
Although the construction of $M^{\spl}_{[k]}(n-1,1)$ via scheme-theoretic closure obscures its moduli interpretation, we are nevertheless able to define two classes of special cycles on the associated splitting Rapoport–Zink (RZ) space and to verify certain useful properties that do not hold in the canonical RZ space.

Fix a basic framing object $\BX$ of dimension $n$ and type $2k$ over $\Spec \BF$, where $\BF_q$ is the residue field of $O_{F_0}$ and $\BF=\ov{\BF}_q$. See \S \ref{sec:RZ spaces} for the precise definition. 
Associate to such a framing object is the \emph{canonical RZ space} $\CN_{\BX}$ over $\Spf O_{\breve{F}}$. The \emph{naive splitting RZ space} $\CN_{\BX}^{\nspl}$ is defined over $\CN_{\BX}$ by classifying additional splitting structures. We then define the \emph{splitting RZ space} $\CN_{\BX}^{\spl}$ as the flat closure of $\CN_{\BX}^{\nspl}$. 
For precise definitions, we refer the reader to \S \ref{sec:RZ}. The terminology is motivated by the corresponding local model constructions.

Let $\BE$ be the unique Lubin-Tate $O_F$-module over $\BF$ and let $\CE$ be its canonical lifting to $\Spf O_{\breve{F}}$.
Define $\BV=\Hom_{O_F}(\BE,\BX)\otimes\BQ$ to be the space of special quasi-homomorphisms of the framing objects. It carries a natural hermitian form $h(\cdot,\cdot)$.
For any $x\in\BV$, we define the \emph{Kudla--Rapoport special cycles} $\CZ(x)$ and $\CY(x)$ as closed formal subschemes of $\CN_{\BX}$.
We define the \emph{Kudla--Rapoport special cycles} $\CZ(x)^{\spl}$ and $\CY(x)^{\spl}$ on $\CN^{\spl}_{\BX}$ as the pull-back along the natural projection $\CN_\BX^{\spl}\to \CN_{\BX}$. We refer the reader to Definition \ref{def:special cycle} for the precise formulation. We prove the following:
\begin{proposition}[Prop. \protect{\ref{prop:special-cartier}}]\label{prop:introduction divisor}
Let $x\in\BV$ be any nonzero vector.
If the closed formal subschemes $\CZ(x)^{\spl}\subseteq\CN^{\spl}_{\BX}$ and $\CY(x)^{\spl}\subseteq\CN^{\spl}_{\BX}$ are nonempty, then they are Cartier divisors.
\end{proposition}
A key step in the proof of Proposition \ref{prop:introduction divisor} is to relate special cycles to sections of the \emph{line bundles of modular forms} $\omega_\CZ$ (resp. $\omega_\CY$) in the square-zero extension of special cycles. 

\begin{remark}
\begin{altenumerate}
\item See also \cite[Theorem 6.1.3]{LRZ25} for an extreme example in which special cycles on $\CN_{\BX}$ fail to be Cartier divisors.
In the forthcoming work \cite{LSR}, we will show that, for most maximal parahoric level structures at ramified places, special cycles are likewise not Cartier divisors.
\item The nonemptiness of the special divisor can be explicitly checked by Bruhat--Tits stratification for reduced locus of the RZ space or the splitting RZ space studied in \cite{HLS-basic,Zachos--Zhao_regular-basic}, using arguments similar to those of \cite{HSY,HLSY}.
\end{altenumerate}
\end{remark}
By Proposition \ref{prop:introduction divisor}, when $x\neq 0$, the special divisors $\CZ(x)^{\spl}$ and $\CY(x)^{\spl}$ define classes $[\CZ(x)^{\spl}]$ and $[\CY(x)^{\spl}]$ in the Grothendieck group $K'_0(\CN_{\BX}^{\spl})$. We refer the reader to \S \ref{sec:specialcycleinsplittingmodel} for the definition of the $K$-group.
We define $[\CZ(0)^{\spl}]$ resp. $[\CY(0)^{\spl}]$ as the classes in $K_0'(\CN_{\BX}^{\spl})$ of the following two-term complexes of $\CO_{\CN_{\BX}^{\spl}}$-modules concentrated in degrees $-1$ and $0$:
\begin{equation*}
	(\cdots \lr 0\lr\omega_{\CZ}^{-1}|_{\CN_{\BX}^{\spl}}\overset{0}{\lr}\CO_{\CN_{\BX}^{\spl}}\lr 0),
	\quad\text{resp.}\quad
	(\cdots \lr 0\lr\omega_{\CY}^{-1}|_{\CN_{\BX}^{\spl}}\overset{0}{\lr}\CO_{\CN_{\BX}^{\spl}}\lr 0).
\end{equation*}
We prove the linear invariance of the intersection of special divisors, which generalizes the main result of \cite{Howard2019}.
\begin{proposition}\label{prop:introduction linear invariance}
Assume that $x_1,\cdots,x_r\in \BV$ and $y_1,\cdots y_r\in\BV$ generate the same $O_F$-submodule. Then we have
\begin{equation*}
	[\CZ(x_1)^{\spl}]\cdot\ldots\cdot[\CZ(x_r)^{\spl}]=[\CZ(y_1)^{\spl}]\cdot\ldots\cdot[\CZ(y_r)^{\spl}]\in K'_0(\CN_{\BX}^{\spl}).
\end{equation*}
Similarly, we have
\begin{equation*}
	[\CY(x_1)^{\spl}]\cdot\ldots\cdot[\CY(x_r)^{\spl}]=[\CY(y_1)^{\spl}]\cdot\ldots\cdot[\CY(y_r)^{\spl}]\in K'_0(\CN_{\BX}^{\spl}).
\end{equation*}
\end{proposition}

\subsection{Comments}\label{subsec:comments}
The main case we consider is the  \emph{strongly non-special} case, see Remark \ref{rem:type-name}.
This case exhibits the worst ramification compared to earlier works.
In this setting, the ``exceptional divisor'' of the naive splitting model has two irreducible components, one of which does not lift to the generic fiber.
An interesting direction for future research is to study the differences between these two components and to identify conditions that distinguish them. 
This may eventually lead to a moduli-theoretic characterization of the splitting model.

We also mention a recent paper by Zachos and Zhao \cite{zachos2024semistable}, where they introduce a different splitting local model $M^{\mathrm{ZZ}}_{[k]}(r,s)$ defined by adding the additional filtration $\CF^0_{\Lambda}$, (but not $\CF^0_{\Lambda^\sharp}$). 
They further construct $M_{[k]}^{\mathrm{bl}}$ by blowing up the preimage of the worst point in $M^{\mathrm{ZZ}}_{[k]}(r,s)$.
They show that $M^{\mathrm{ZZ}}_{[k]}(n-1,1)$ and $M_{[k]}^{\mathrm{bl}}(n-1,1)$
are flat, normal and Cohen-Macaulay, and that $M_{[k]}^{\mathrm{bl}}(n-1,1)$ is semi-stable. In later work, Zachos and Zhao \cite[Appendix]{Zachos--Zhao_regular-basic} prove that their $M_{[k]}^{\mathrm{bl}}(n-1,1)$ coincides with our splitting model $M^{\spl}_{[k]}(n-1,1)$.

However, since their construction introduces filtrations only on one side, the spaces
$M^{\mathrm{ZZ}}_{[k]}(n-1,1)$
and
$M^{\mathrm{ZZ}}_{[k]}(1,n-1)$
are not isomorphic.
In fact, a direct computation shows that $M^{\mathrm{ZZ}}_{[k]}(1,n-1)$ is not flat in general, due to similar reasons underlying the failure of flatness in $M^{\nspl}_{[k]}(n-1,1)$. 
It remains an interesting problem to give a more geometric explanation of the failure of flatness and to find a moduli interpretation of the flat integral model.

\subsection{The structure of the paper}
The organization of the paper is as follows.
In Section \ref{sec:moduli}, we define the splitting model for any signature and establish some geometric properties.  
For the remaining sections, we will assume the signature is $(r,s)=(n-1,1)$.
In Section \ref{sec:regularity}, we prove that the Kr\"amer condition implies the strengthened spin condition, and the splitting model is semi-stable.
In Section \ref{sec:resolution}, we prove that the splitting model is the blow-up of the canonical local model at the worst point.
In Section \ref{sec:RZ spaces}, we define and study the corresponding splitting Rapoport--Zink spaces and special cycles. We prove that special cycles are Cartier divisors and that they satisfy the linear invariance property.

\subsection{Acknowledgement}
We want to thank Michael Rapoport for his encouragement and his comments on the early drafts. We also want to thank the anonymous referees for their insights. 
Q. He and Y. Luo want to thank the Institute for Advanced Study in Mathematics and the School of Mathematical Sciences at Zhejiang University for its hospitality during the summer of 2024 when part of this work was done. Y. Shi wants to thank the Max Planck Institute for Mathematics for its hospitality during the spring of 2025 when part of the work is done. Y. Shi is supported by the start-up grant and Qizhen Grant of Zhejiang University.

\subsection{Notations}\label{sec:notation}
\begin{altitemize}
\item For any $O_F$-algebra $R$ with a structure map $s:O_F\rightarrow R$, we denote $\Pi:=\pi\otimes 1$ and $\pi:=1\otimes s(\pi)$. 
\item For an integer $0\leq \alpha\leq n$, we define $\alpha^\vee:=n+1-\alpha$.
\item Let $R$ be a DVR and let $X\rightarrow \Spec R$ be an arithmetic scheme. We denote by $X_{\eta}$ its generic fiber and $X_s$ its special fiber.
\item We consider matrices:
\begin{equation*}
	H=
    \left(\begin{array}{ccc}
        &        &   1\\
        &\iddots &   \\
    1   &        & 
    \end{array}\right),
    \quad
    J=
    \left(\begin{array}{cc}
        &   H\\
    -H&
    \end{array}\right),
    \quad
    \Upsilon=\left(\begin{array}{ccc}&&H\\&H&\\-H&&\end{array}\right).    
\end{equation*}
For any matrix $G$, we define $G^{\ad}:=HG^tH$. 
\end{altitemize}

\section{Moduli problems}\label{sec:moduli}
\subsection{Basic setup}\label{moduli_setup}
Let $O_F/O_{F_0}$ be a quadratic extension of complete discrete valuation rings with the same perfect residue field $\BF_q$ with $\mathrm{char}(\BF_q)=p\neq 2$ and uniformizers $\pi$, resp. $\pi_0$ such that $\pi^2=\pi_0$.

Consider the vector space $F^n$ with a set of $F$-basis $e_1,\cdots,e_n$ which we call the \emph{standard basis}. Let
\[
    \phi: F^n\times F^n\rightarrow F
\]
be the $F/F_0$-Hermitian form which is split with respect to the standard basis, that is,
\begin{equation}\label{moduli_setup:hermitian}
    \phi(ae_i,be_j)=\bar{a}b\delta_{ij^\vee},\quad a,b\in F,
\end{equation}
where $a\mapsto\bar{a}$ is the nontrivial element of $\Gal(F/F_0)$ and $j^\vee=n+1-j$. 
Attached to $\phi$ are the respective alternating and symmetric $F_0$-bilinear forms
\[
    F^n\times F^n\rightarrow F_0
\]
given by
\[
    \langle x,y\rangle:=\frac{1}{2}\tr_{F/F_0}(\pi^{-1}\phi(x,y))
    \quad\text{and}\quad
    (x,y):=\frac{1}{2}\tr_{F/F_0}(\phi(x,y)).
\]

For each integer $i=bn+c$ with $0\leq c<n$, define the standard $O_F$-lattice
\begin{equation}\label{moduli_setup:lattice}
    \Lambda_i:=\sum_{j=1}^c\pi^{-b-1}O_F e_j+\sum_{j=c+1}^n\pi^{-b}O_F e_j\subset F^n.
\end{equation}

For all $i$, the $\langle\,,\,\rangle$-dual of $\Lambda_i$ in $F^n$ is $\Lambda_{-i}$, by which we mean
\[
    \{x\in F^n \mid \langle \Lambda_i,x\rangle\subset O_{F_0}\}=\Lambda_{-i},
\]
and
\begin{equation}\label{moduli_setup:dual}
    \Lambda_i\times\Lambda_{-i}\xrightarrow{\langle\, , \,\rangle}  O_{F_0}
\end{equation}
is a perfect $ O_{F_0}$-bilinear pairing.
Similarly, $\Lambda_{n-i}$ is the $(\,,\,)$-dual of $\Lambda_i$ in $F^n$. The $\Lambda_i$'s form a complete, periodic, self-dual lattice chain
\[
    \cdots\subset\Lambda_{-2}\subset\Lambda_{-1}\subset\Lambda_0\subset\Lambda_1\subset\Lambda_2\subset\cdots.
\]

For an integer $0\leq k\leq \frac{n}{2}$, let $[k]=\{\pm k\}+n\BZ\subset \BZ$, we will be interested in the sub-lattice chain $\CL_{[k]}$:
\begin{equation}\label{equ:CLk lattice chain}
    \cdots\subset \Lambda_{-k}\subset\Lambda_k\subset \Lambda_{n-k}\subset\cdots.
\end{equation}

\begin{remark}\label{rem:type-name}
We classify the cases according to different choices of $k$, and assign the following names:
\begin{itemize}
\item The case $k=0$ is called the \emph{self-dual} case.
\item The case $n=2m$ and $k=m$ is called the \emph{$\pi$-modular} case.
\item The case $n=2m+1$ and $k=m$ is called the \emph{almost $\pi$-modular} case.
\item The case $n=2m$ and $k=m-1$ is called \emph{Yu's} case.
\item All other cases are called \emph{strongly non-special}.
\end{itemize}
\end{remark}

\subsection{Canonical local model}\label{sec:canonical-model}
Fix an integer $0\leq k\leq \lfloor \frac{n}{2}\rfloor$.
We recall the definition of the canonical local model with index $[k]$.
Let $r,s$ be natural numbers such that $r+s=n$.
Let $E=F_0$ when $r=s$ and $E=F$ otherwise.
\begin{definition}
(i) The \emph{naive local model $M_{[k]}^\naive(r,s)$} is a projective scheme over $\Spec  O_E$.  It represents the moduli functor that sends each $ O_E$-algebra $R$ to the set of all families
\[
   (\CF_i \subset \Lambda_i \otimes_{ O_{F_0}}R)_{i\in [k]}
\]
such that
\begin{enumerate}[label=(LM\arabic*)]
\item\label{item:LM1}
for all $i$, $\CF_i$ is an $ O_F \otimes_{ O_{F_0}} R$-submodule of $\Lambda_i \otimes_{ O_{F_0}} R$, and an $R$-direct summand of rank $n$;
\item\label{item:LM2}
for all $i < j$, the natural arrow $\lambda_{ij}:\Lambda_i \otimes_{ O_{F_0}} R \to \Lambda_j \otimes_{ O_{F_0}} R$ induced by the inclusion $\Lambda_i\hookrightarrow\Lambda_j$, maps $\CF_i$ into $\CF_j$;
\item\label{item:LM3}
for all $i$, the isomorphism $\Lambda_i \otimes_{ O_{F_0}} R \xra[\undertilde]{\pi \otimes 1} \Lambda_{i-n} \otimes_{ O_{F_0}} R$ identifies $\CF_i$ with $\CF_{i-n}$;
\item\label{item:LM4}
for all $i$, the perfect $R$-bilinear pairing
\[
   (\Lambda_i \otimes_{ O_{F_0}} R) \times (\Lambda_{n-i} \otimes_{ O_{F_0}} R)
   \xra{(\,,\,) \otimes R} R
\]
identifies $\CF_i^\perp$ with $\CF_{n-i}$ inside $\Lambda_{n-i} \otimes_{ O_{F_0}} R$, here $\CF_i^{\perp}$ is the orthogonal complement of $\CF_i$ with respect to the above perfect pairing.
\item\label{item:LM5}(Kottwitz condition) for all $i$, the element $\Pi \in  O_F \otimes_{ O_{F_0}} R$ acts on $\CF_i$ as an $R$-linear endomorphism with characteristic polynomial 
\[
   \det(T\cdot \id - \Pi \mid \CF_i) = (T+\pi)^r(T-\pi)^s \in R[T].
\]
\end{enumerate}

\noindent (ii) We have a natural embedding of the generic fiber into the naive local model:
\begin{equation}\label{eq:embedding of the naive generic fiber}
\begin{aligned}
\xymatrix{
M^{\naive}_{[k],\eta}(r,s)\ar@{^(->}[r]& M_{[k]}^{\naive}(r,s).
}
\end{aligned}
\end{equation}
We define the local model $M^{\loc}_{[k]}(r,s)$ to be the scheme-theoretic closure of this embedding.
\end{definition}

The following theorem is due to \cite{Pappas2000,Arzdorf,Smithling2015,Yu2019,LUO2024}, among others.
\begin{theorem}\label{thm:canonical-LM}
Let $(r,s)=(n-1,1)$, and assume that $k$ is not $\pi$-modular in (i)-(v) below.
\begin{altenumerate}    
\item The local model is flat, normal and Cohen-Macaulay of dimension $n$.
\item The special fiber is a reduced scheme whose irreducible components are all normal and Cohen–Macaulay of dimension $n-1$.
\item The point on the special fiber corresponding to $(\Pi\Lambda_{i,\BF_q} \subset \Lambda_{i,\BF_q})$ is referred to as the \emph{worst point}. All irreducible components are smooth except at this point.
\item The local model admits a moduli interpretation via the strengthened spin condition (see \cite[\S 1.4]{LUO2024} for more detailed discuss in each case.).
\item The geometry of the special fiber varies according to the value of $k$:
\begin{itemize}
\item\, When $k=0$, the special fiber is irreducible.
\item\, When $n=2m+1$ and $k=m$, the special fiber is smooth and irreducible,
\item\, When $n=2m$ and $k=m-1$, the special fiber has three irreducible components.
\item\, When $k$ is strongly non-special, the special fiber has two irreducible components.
\end{itemize}
\item If $k$ is $\pi$-modular, then the local model is smooth, connected, of dimension $n$, and does not contain the worst point.
\end{altenumerate}
\end{theorem}
\begin{proof}
Parts (i)-(iv) are summarized in \cite[Theorem 1.3.2]{LUO2024}. Part (v) is due to \cite{Pappas2000} when $k=0$, \cite{Arzdorf} when $n=2m+1$ and $k=m$, \cite{Yu2019} when $n=2m$ and $k=m-1$, \cite{LUO2024} in the remaining cases. See also \cite[Rem. 4.3.6]{LUO2024} for a detailed discussion. Part (vi) is due to \cite{PR2009}.
\end{proof}

\subsection{Splitting model}\label{sec:splitting-model}
We retain the assumptions in \S \ref{sec:canonical-model}.
\begin{definition}\label{defn:spl-model}
\noindent (i) We define the \emph{Kr\"amer model} $M^{\Kra}_{[k]}(r,s)$ as the moduli functor that assigns to each $ O_F$-algebra $R$ the families
\[
    (\CF_i^0\subset\CF_i\subset\Lambda_i\otimes_{ O_{F_0}}R)_{i\in [k]}
\]
such that 
\begin{enumerate}[label=(SP\arabic*)]
\item\label{item:SP1'}
The filtration $(\CF_i\subset\Lambda_i\otimes_{ O_{F_0}}R)_{i\in[k]}$ satisfies \ref{item:LM1}-\ref{item:LM4};
\item\label{item:SP2}
$\CF_i^0$ is an $O_F\otimes_{ O_{F_0}}R$-submodule of $\Lambda_i\otimes_{ O_{F_0}}R$ and an $R$-direct summand of rank $s$;
\item\label{item:SP3}
For each $i<j$, the morphism $\Lambda_i\otimes_{ O_{F_0}}R\rightarrow \Lambda_j\otimes_{ O_{F_0}}R$ maps $\CF_i^0$ into $\CF_j^0$:
\begin{equation*}
\begin{aligned}
\xymatrix{
\Lambda_i\otimes_{ O_{F_0}}R\ar[r]&\Lambda_j\otimes_{ O_{F_0}}R\\
\CF_i^0\ar@{^(->}[u]\ar[r]&\CF_j^0\ar@{^(->}[u]
}
\end{aligned}
\end{equation*}
\item\label{item:SP4}
The isomorphism $\Lambda_i\otimes_{ O_{F_0}}R\longrightarrow\Lambda_{i-n}\otimes_{ O_{F_0}}R$ induced by $\Lambda_i\overset{\pi\otimes 1}{\longrightarrow}\Lambda_{i-n}$ identifies $\CF_i^0$ with $\CF_{i-n}^0$;
\item\label{item:SP5}
\quad $(\Pi+\pi)(\CF_i)\subset\CF_i^0$;
\item\label{item:SP6}
\quad $(\Pi-\pi)(\CF_i^0)=(0)$.
\end{enumerate}

\noindent (ii)  We define the \emph{naive splitting model} $M_{[k]}^{\nspl}(r,s)$ as the closed subfunctor of $M^{\Kra}_{[k]}(r,s)$ which associates to each $ O_F$-algebra $R$ the families
\[
    (\CF_i^0\subset\CF_i\subset\Lambda_i\otimes_{ O_{F_0}}R)_{i\in [k]}\in M^{\Kra}_{[k]}(r,s)
\]
which further satisfy:
\begin{enumerate}[label=(SP\arabic*')]
\item\label{item:SP1}
The filtration $(\CF_i\subset\Lambda_i\otimes_{ O_{F_0}}R)_{i\in[k]}$ defines an $R$-point of $M^{\loc}_{[k]}(r,s)$.
\end{enumerate}
By Theorem \ref{thm:canonical-LM}(iv), this is equivalent to require that the filtration satisfies the strengthened spin condition, see \cite[\S 2]{LUO2024}.

\noindent (iii) We have a natural embedding of the generic fiber into the naive splitting model:
\begin{equation*}
\begin{aligned}
\xymatrix{
M^{\nspl}_{[k],\eta}(r,s)\ar@{^(->}[r]& M_{[k]}^{\nspl}(r,s).
}
\end{aligned}
\end{equation*}
We define the \emph{splitting model} $M^{\spl}_{[k]}(r,s)$ to be the scheme-theoretic closure of this embedding.
\end{definition}

\begin{remark}\label{rem:kramer-ss}
The space $M_{[k]}^{\mathrm{Kra}}(r,s)$ was initially studied by \cite{Kramer2003} when $k=0$ and $(r,s)=(n-1,1)$.
When $k=0$, $M_{[k]}^{\mathrm{Kra}}(r,s)$ is flat. This is proved in loc.cit. for $(r,s)=(n-1,1)$, and extended to all $(r,s)$ in \cite{bijakowski2024geometry}.

For $(r,s)=(n-1,1)$ and strongly non-special $k$, we will show that $M_{[k]}^{\nspl}(n-1,1)\simeq M_{[k]}^{\mathrm{Kra}}(n-1,1)$, see Theorem \ref{thm:Kramer-vs-ss}.
In general, there is a chain of closed embeddings
\begin{equation*}
\begin{aligned}
\xymatrix{
M^{\spl}_{[k]}(r,s)\ar@{^(->}[r]& M^{\nspl}_{[k]}(r,s)\ar@{^(->}[r]&M_{[k]}^{\mathrm{Kra}}(r,s),
}
\end{aligned}
\end{equation*}
each of which becomes an isomorphism over the generic fiber.
Therefore, $M^{\spl}_{[k]}(r,s)$ can also be defined as the scheme-theoretic closure of the open embedding:
\begin{equation*}
\begin{aligned}
\xymatrix{
M_{[k],\eta}^{\mathrm{Kra}}(r,s)\ar@{^(->}[r]& M_{[k]}^{\mathrm{Kra}}(r,s).
}
\end{aligned}
\end{equation*}
It is worth noting that for general $(r,s)$ and $k$, the existence of the forgetful functor for $M^{\Kra}_{[k]}$ as in \eqref{equ:forgetful} is unclear, though such a functor would be useful in practice.
\end{remark}

\subsection{Basic properties}\label{sec:basic-properties}
We have the forgetful functor
\begin{equation}\label{equ:forgetful}
    \tau:M_{[k]}^{\nspl}(r,s)\longrightarrow M_{[k]}^{\loc}(r,s)\otimes_{O_E}O_F,\qquad (\CF^0_i,\CF_i)\mapsto \CF_i.
\end{equation}

\begin{lemma}\label{lem:generic-fiber}
	$\tau$ induces an isomorphism over the generic fibers.	
\end{lemma}
\begin{proof}
Since $O_{F}\otimes_{ O_{F_0}}F\simeq F\times F$ is split, when $R$ is a $F$-algebra, the additional $O_F$-action induces a $\mathbb{Z}/2$-grading on $\Lambda\otimes R$ and $\CF$. This grading is the same as the additional filtration in the naive splitting model.
Moreover, the degree $1$ part of the grading can be recovered from the degree $0$ part by taking dual.
Hence both the generic fiber of the canonical local model and the naive splitting model are isomorphic to the Grassmannian $\mathrm{Gr}(r,n)$.
\end{proof}

\subsubsection{}
Let $R$ be an $O_F$-algebra and $\Lambda$ a \emph{locally free} $ O_F\otimes_{ O_{F_0}}R$-module of rank $n$. We have an exact sequence
\begin{equation*}
	0\rightarrow (\Pi-\pi)\Lambda\longrightarrow \Lambda\xrightarrow{\Pi+\pi}(\Pi+\pi)\Lambda\rightarrow 0.
\end{equation*}
This induces an isomorphism:
\begin{equation*}
	(\Pi+\pi)\Lambda\xrightarrow{\sim}\Lambda/(\Pi-\pi)\Lambda,
	\quad
	v\mapsto (\Pi+\pi)^{-1}v+(\Pi-\pi)\Lambda.
\end{equation*}
See \cite[Lemma 3.2]{Howard2019}. The same result holds if we change $\Pi\pm\pi$ to $\Pi\mp \pi$.
In particular, $(\Pi-\pi)\Lambda$ and $(\Pi+\pi)\Lambda$ are both $R$-direct summands of rank $n$. Note that these two submodules specialize to the submodule $\Pi\Lambda_{\BF_q}\subset\Lambda_{\BF_q}$ over the special fiber.

Now we come back to the perfect pairing for any $ O_F$-algebra $R$:
\begin{equation*}
    \langle\,,\,\rangle:\Lambda_{-i}\otimes_{ O_{F_0}}R\times \Lambda_i\otimes_{ O_{F_0}}R\rightarrow R.
\end{equation*}
Since
\begin{equation}\label{equ:skew-hermitian}
    \langle \Pi v,w\rangle=\langle v, -\Pi w\rangle,\quad\forall v\in\Lambda,w\in\Lambda^{\sharp}.
\end{equation}
We have 
$$\left((\Pi\pm\pi)\Lambda_{-i}\otimes_{ O_{F_0}}R\right)^\perp=(\Pi\pm\pi)\Lambda_{i}\otimes_{ O_{F_0}}R.$$ We define an induced pairing,
\begin{equation}\label{equ:new-pair}
    \{\,,\,\}:(\Pi-\pi)(\Lambda_{-i}\otimes_{ O_{F_0}}R)\times (\Pi+\pi)(\Lambda_i\otimes_{ O_{F_0}}R)\longrightarrow  R,
    \quad
    \{v,w\}\mapsto \langle v,(\Pi+\pi)^{-1}w\rangle.
\end{equation}

For any $R$-point $(\CF_i^0\subset\CF_i)_{i\in[k]}$ of $M^{\nspl}_{[k]}(r,s)$, Axiom \ref{item:SP6} implies that $\CF_{-i}^0\subset (\Pi+\pi)\Lambda_{-i}\otimes_{ O_{F_0}}R$. We define
\begin{equation}\label{equ:perp-filtration}
    \CG_i^0:=\left((\Pi+\pi)^{-1}\CF_{-i}^0\right)^{\perp}\subset \Lambda_i\otimes_{ O_{F_0}}R.
\end{equation}
The following proposition follows from the proof of \cite[Proposition 2.4]{bijakowski2024geometry}, which is also implicitly proved in \cite[\S 3]{Howard2019}.
\begin{proposition}\label{prop:duality}
\begin{altenumerate}
\item     The submodule $\CG_i^0$ is a locally free $R$-module of rank $r$ and $R$-direct summand of $\CF_i$. Moreover, we have
\begin{equation*}
    (\Pi-\pi)(\CF_i)\subseteq\CG_i^0,
    \quad
    (\Pi+\pi)(\CG_i^0)=(0).
\end{equation*}
\item Let $L_i\subset \Lambda_i\otimes_{ O_{F_0}}R/\CF_i$ be the isotropic complement of $\CG_{-i}^0\subset \CF_{-i}$ with respect to the perfect pairing
\begin{equation*}
	\langle\,,\,\rangle:\CF_{-i}\times (\Lambda_i\otimes_{ O_{F_0}}R)/\CF_i\rightarrow R.
\end{equation*}
Then $L_i$ is a direct $R$-summand of rank $s$, with 
\begin{equation*}
	(\Pi-\pi)(\Lambda_i\otimes_{ O_{F_0}}R/\CF_i)\subset L_i,
	\quad
	(\Pi+\pi)(L_i)=(0).
\end{equation*}
Furthermore, we can characterize $L_i$ using the following formula:
\begin{equation*}
	L_i=\ker\left(
	\Lambda_i\otimes_{O_{F_0}}R/\CF_i\xrightarrow{\Pi-\pi}(\Pi-\pi)\Lambda_i\otimes_{O_{F_0}}R/\CF_i^0.
	\right)
\end{equation*}
\end{altenumerate}
\end{proposition}
\begin{proof}
(i) It follows from the same proof as \cite[Proposition 2.4]{bijakowski2024geometry}, note that the rank argument follows from \eqref{equ:new-pair}.

\noindent (ii) It is clear from (i) and \eqref{equ:skew-hermitian} that we have the inclusion
\begin{equation*}
	L_i\subset\ker\left(
	\Lambda_i\otimes_{O_{F_0}}R/\CF_i\xrightarrow{\Pi-\pi}(\Pi-\pi)\Lambda_i\otimes_{O_{F_0}}R/\CF_i^0.
	\right)
\end{equation*}
By (i) we know that $L_i$ is an $R$-direct summand of rank $s$. Since $\CF_i\xrightarrow{\Pi-\pi}(\Pi-\pi)\Lambda_i/\CF_i^0$ is a surjection, the kernel is also an $R$-direct summand of rank $s$, hence we get the equality.
\end{proof}

It is clear that the morphism $\Lambda_i\otimes_{ O_{F_0}}R\rightarrow \Lambda_j\otimes_{ O_{F_0}}R$ maps $\CG_i^0$ into $\CG_j^0$ for any 
$i,j\in [k]$ such that $i<j$. 
Denote by $\overline{(-)}$ the conjugate on $R$, we have the following morphism
\begin{equation*}
M^{\nspl}_{[k]}(r,s)
    \overset{\sim}{\longrightarrow} M^{\nspl}_{[k]}(s,r),
    \quad
    (\CF_i^0\subset\CF_i)_{i\in[k]}\mapsto
    (\overline{\CG_i^0}\subset\overline{\CF_i})_{i\in [k]}.
\end{equation*}
Note that we conjugated the $R$-action on the additional filtration $\ov{\CG}_i^0$ so that the Kr\"amer conditions \ref{item:SP5} and \ref{item:SP6} are satisfied. By reversing this procedure, we see that the morphism is an isomorphism.
This isomorphism restricts to an isomorphism
\begin{equation*}
    M^{\spl}_{[k]}(r,s)\overset{\sim}{\longrightarrow}M^{\spl}_{[k]}(s,r).
\end{equation*}

\begin{remark}
Similar to \cite{bijakowski2024geometry} and \cite{BZZ25}, we can define locally closed subschemes 
\begin{equation*}
    X^i_{h,\ell}:=\{x\in M^{\nspl}_{[k]}(r,s)\mid \dim \Pi\CF_i=h,\dim\CF_i^0\cap\CG_i^0=\ell\},\quad i\in[k].
\end{equation*}
One may ask the following questions:
\begin{altenumerate}
\item What is the relation between $X^i_{h,\ell}$ for different $i$? 
\item For every $h\leq\ell$, we have the inclusion
\begin{equation*}
    \overline{X^i_{h,\ell}}\subseteq\bigsqcup_{0\leq h'\leq h\leq \ell\leq \ell'}X^i_{h',\ell'}.
\end{equation*}
It this an equality?
\item Let $Y^i_{h,\ell}:=X^i_{h,\ell}\cap M^{\spl}_{[k]}(r,s)$. Is $Y^i_{h,\ell}$ smooth? What is the dimension of $Y^i_{h,\ell}$?
\item How do these indices relate to the Kottwitz--Rapoport strata?
\end{altenumerate}
\end{remark}

\subsection{The case of signature $(n-1,1)$}\label{sec:irr-comp}
From now on, we assume $n\geq 3$ and $(r,s)=(n-1,1)$, and omit the signature when writing $M^{\loc}_{[k]}$, $M^{\nspl}_{[k]}$ and $M^{\spl}_{[k]}$. In this case, we have $E=F$, hence the canonical local model and the splitting model are defined in the same base ring.

Suppose $k$ is not $\pi$-modular, let $*\in M^{\loc}_{[k]}(\BF_q)$ be the worst point defined as
 $\{\Pi\Lambda_{i,\BF_q}\subset\Lambda_{i,\BF_q}\}_{i\in[k]}$.

\begin{proposition}\label{prop:geometric:worst}
Suppose $k$ is not $\pi$-modular, the map $\tau$ induces an isomorphism
\begin{equation*}
	\tau:M_{[k]}^{\nspl}\setminus \tau^{-1}(*)\longrightarrow M_{[k]}^\loc\setminus\{*\}.
\end{equation*}
\end{proposition}
\begin{proof}
For each $i\in [k]$, we have a sheaf of module $(\Pi+\pi)(\CF_i)$ over $M^{\nspl}_{[k]}$. 
By \ref{item:SP5}, we have $(\Pi+\pi)(\CF_i)\subseteq \CF_i^0$. 
We claim that a point $z\in M^{\loc}_{[k]}\setminus \{*\}$ corresponds to a $k(z)$-subspace of $\Lambda_i\otimes_{O_{F_0}}k(z)$ such that $(\Pi+\pi)(\CF_i)\neq (0)$: 
\begin{altenumerate}
\item If $z$ lies in the generic fiber, the subspace $(\Pi+\pi)(\CF_i)$ is locally free of rank $1$ by the Kottwitz condition.
\item If $z$ lies in the special fiber, the claim follows from the definition of the worst point.
\end{altenumerate}
The remaining proof is the same as that of \cite[Proposition 4.3]{Kramer2003}.
\end{proof}
\begin{definition}\label{def:exceptional-divisor}
    We define the \emph{naive exceptional divisor} as $\NExc:=\tau^{-1}(*)\subset M^{\nspl}_{[k]}$, and the \emph{exceptional divisor} as the scheme-theoretic intersection $\Exc:=\tau^{-1}(*)\cap M^{\spl}_{[k]}$. 
\end{definition}
\begin{remark}\label{rmk:pi-modular-iso}
When $k$ is $\pi$-modular, as the canonical local model does not contain the worst point, the same argument as in Proposition \ref{prop:geometric:worst} implies that the splitting model is isomorphic to the local model, and the exceptional divisors are empty. 
\end{remark}

\begin{proposition}\label{prop:exceptional-divisor}
Suppose $2k\notin\{0,n\}$.
Let $\lambda_k$ and $\lambda_{n-k}$ be the transition maps $\Lambda_k\rightarrow \Lambda_{n-k}$ and $\Lambda_{n-k}\rightarrow\Lambda_{n+k}$ respectively.
The naive exceptional divisor is a linked variety with equal dimension $n-1$. 
It has two irreducible components $\Exc_1$ and $\Exc_2$, defined by $\lambda_k(\CF_k^0)= (0)$ and $\lambda_{n-k}(\CF_{n-k}^0)=(0)$ respectively. Their intersection $\Exc_1\cap \Exc_2$ is irreducible of dimension $n-2$.
\end{proposition}
\begin{proof}
The transition maps induce the transition maps between the lattices which represent the worst point:
\begin{equation*}
	\lambda_k:\Pi\Lambda_{k,\BF_q}\longrightarrow \Pi\Lambda_{n-k,\BF_q},
	\quad
	\lambda_{n-k}:\Pi\Lambda_{n-k,\BF_q}\longrightarrow \Lambda_{n+k,\BF_q}.
\end{equation*}
By choosing the standard basis as in \cite[\S 3.1]{LUO2024}, the transition maps can be represented by matrices
\begin{equation}\label{equ:excep-transition}
	\lambda_k=\left(\begin{matrix}
	I_k&&\\&0&\\&&I_k
\end{matrix}\right),
\quad
\lambda_{n-k}=\left(\begin{matrix}
0&&\\&I_{n-2k}&\\&&0	
\end{matrix}\right).
\end{equation}

From the definition, the naive exceptional divisor has the moduli description which associates to each $\BF_q$-algebra $R$ the families
\begin{equation*}
	(\CF_k^0,\CF_{n-k}^0)\in \BP(\Pi\Lambda_{k,\BF_q})\times\BP(\Pi\Lambda_{n-k,\BF_q})
\end{equation*}
such that $\lambda_k(\CF_k^0)\subset \CF_{n-k}^0$ and $\lambda_{n-k}(\CF_{n-k}^0)\subset \CF_k^0$ under the identification $\Lambda_{n+k}\xrightarrow{\sim}\Lambda_{k}$ of \ref{item:LM3} and \ref{item:SP4}. All the other axioms hold automatically. The statements now is standard. For instance, the naive exceptional locus agrees with the special fiber of the local model for $\mathrm{GL}_n$ with cocharacter $\mu=(1,0^{n-1})$ and lattice chain generated by $\Lambda_{n-k}$ and $\Lambda_{n+k}$, see \cite{Gortz}.
\end{proof}

\begin{remark}
\begin{altenumerate}
\item When $I=\{0\}$, the native splitting model is identical to the splitting model, this is studied in \cite{Kramer2003}.
\item When $n=2m$ and $I=\{m\}$, the canonical local model is isomorphic to the (naive) splitting model, since it doesn't contain the worst point.
\item When $n=2m+1$ and $I=\{m\}$, the canonical local model is smooth. However, canonical local model, splitting model, and naive splitting model, are not isomorphic to each other. A direct calculation, or using Theorem \ref{thm:blowup} and \cite[\S 5.2]{Richarz-master}, one can get a moduli interpretation in this case.

\item When $n=2m$ and $I=\{m-1\}$, \cite[Proposition 9.12]{RSZ18} establishes the isomorphism:
\begin{equation*}
    M^{\loc}_{[m-1]}\simeq M^{\loc}_{[m-1,m]}.
\end{equation*}
Yu's thesis \cite{Yu2019} introduces the moduli space $M^{\spl}_{[m-1,m]}$ over $M^{\loc}_{[m-1,m]}$ parameterizing additional filtrations $(\CF_i^0\subset\CF_i)_{i\in [m-1,m]}$ that satisfy the Kr\"amer condition and, for any $i,j\in [m-1,m]$ the transition maps 
\end{altenumerate}
\begin{equation*}
    \lambda:\Lambda_i\otimes_{O_{F_0}}R\longrightarrow \Lambda_j\otimes_{O_{F_0}}R
\end{equation*}
carry $\lambda(\CF_i^0)\subseteq\CF_j^0$.
Yu proves that $M^{\spl}_{[m-1,m]}$ is flat over $O_F$.
In this setting, a similar isomorphism holds in the splitting model context:
\begin{equation}\label{equ:Yu-iso}
    M^{\spl}_{[m-1]}\simeq M^{\spl}_{[m-1,m]}.
\end{equation}
In fact, Proposition \ref{prop:geometric:worst} and its proof shows that the forgetful map
\begin{equation*}
    M^{\spl}_{[m-1,m]}\longrightarrow M^{\nspl}_{[m-1]}
\end{equation*}
is a closed embedding, as we always have $\CF^0_m=(\Pi+\pi)\CF_m$. This map induces an isomorphism over the generic fiber, and the isomorphism \eqref{equ:Yu-iso} follows by taking the flat closure.

It is noteworthy that while the additional filtration $\CF^0_m\subset\CF_m$ always exists in $M^{\nspl}_{[m-1]}$, it may not necessarily satisfy:
\begin{equation*}
        \lambda_{m-1}(\CF_{m-1}^0)\subseteq\CF^0_m,
        \quad
        \lambda_{m}(\CF^0_{m})\subseteq\CF^0_{m+1}.
\end{equation*}
Yu applies these conditions to $M^{\nspl}_{[m-1]}$, and get the closed subscheme $M^{\spl}_{[m-1]}=M^{\spl}_{[m-1,m]}$.
\end{remark}

\subsubsection{}\label{sec:irreducible components of splitting model}
Consider the base change of $\tau$ to the special fiber
\begin{equation*}
	\tau=\tau_s:M^{\nspl}_{[k],s}\longrightarrow M^{\loc}_{[k],s}.
\end{equation*}
Denote $M^{\loc}_{[k],s}=\bigcup_i Z_i$ as the union of irreducible components.
Define $\widetilde{Z_i}$ as the scheme-theoretic closure of $\tau_s^{-1}(Z_i\setminus\{*\})$ for each $i$. Then we have 
\begin{equation*}
	\tau(\widetilde{Z}_i)=Z_i \quad \text{for any } i.
\end{equation*}

\begin{proposition}
The special fiber $M^{\nspl}_{[k],s}$ of the naive splitting model is the union of $\wt{\CZ}_i$ and exceptional divisors:
\begin{equation*}
    M^{\nspl}_{[k],s}=\Exc_1\cup \Exc_2\cup \bigcup_i\wt{\CZ}_i.
\end{equation*}
\end{proposition}
\begin{proof}
It is clear from the construction that $M^{\nspl}_{[k],s}$ is the union of these closed subschemes: for any $R$-point $x\in M^{\nspl}_{[k],s}(R)$, if $\tau(x)\in M^{\loc}_{[k],s}\setminus \{*\}$, then $x\in \bigcup_i\wt{\CZ}_i$, otherwise, $x\in \NExc=\Exc_1\cup \Exc_2$.

Next, since $Z_i\setminus \{*\}$ is irreducible of dimension $n-1$ for each $i$, so is $\widetilde{Z}_i$ for each $i$. We have shown that $\Exc_i$ is irreducible of dimension $n-1$ for $i=1,2$ in Proposition \ref{prop:exceptional-divisor}. It is clear that they are irreducible components.
\end{proof}

\section{Regularity of the splitting model}\label{sec:regularity}

\subsection{Recollection of the canonical local model}\label{sec:canonical-LM}
From now on, we will always assume the signature to be $(n-1,1)$, and omit the signature when writing $M^{\loc}_{[k]}$, $M^{\Kra}_{[k]}$,  $M^{\nspl}_{[k]}$ and $M^{\spl}_{[k]}$. We assume the index $I=\{k\}$ is not $\pi$-modular.

We follow the discussion in \cite[\S 3]{LUO2024}. We will choose the \emph{standard} ordered 
$O_{F_0}$-basis as in \cite[(3.1.1)]{LUO2024}: 
\begin{align}
\begin{split}\label{equ_chart:standard-basis}
	\Lambda_{k,O_F}:  &
	\pi^{-1}e_1\otimes 1,\cdots,\pi^{-1}e_{k}\otimes 1,e_{k+1}\otimes 1,\cdots ,e_n\otimes 1;\\ 
	& e_1\otimes 1,\cdots,e_{k}\otimes 1,\pi e_{k+1}\otimes 1,\cdots,\pi e_n\otimes 1.\\
	\Lambda_{n-k,O_F}:&
	\pi^{-1}e_1\otimes 1,\cdots,\pi^{-1}e_{n-k}\otimes 1,e_{n-k+1}\otimes 1,\cdots , e_n\otimes 1;\\
    &e_1\otimes 1,\cdots,e_{n-k}\otimes 1,\pi e_{n-k+1}\otimes 1,\cdots,\pi e_n\otimes 1.
\end{split}
\end{align}
We will also choose the \emph{reordered} basis as in \cite[(3.1.3)]{LUO2024}:
\begin{align}\label{equ_chart:reordered-basis}
	\begin{split}
		\Lambda_{k,O_F}:  &
		e_{n-k+1}\otimes 1,\cdots, e_n\otimes 1,	
		\pi^{-1}e_1\otimes 1,\cdots,\pi^{-1}e_{k}\otimes 1;	
		e_{k+1}\otimes 1,\cdots, e_{n-k}\otimes 1,\\
		&   
		\pi e_{n-k+1}\otimes 1,\cdots, \pi e_n\otimes 1,
		e_1\otimes 1,\cdots,  e_{k}\otimes 1; 	
		\pi e_{k+1}\otimes 1,\cdots,\pi e_{n-k}\otimes 1.\\
		\Lambda_{n-k,O_F}:&
		e_{n-k+1}\otimes 1,\cdots, e_n\otimes 1,	
		\pi^{-1}e_1\otimes 1,\cdots,\pi^{-1}e_{k}\otimes 1;
		\pi^{-1} e_{k+1}\otimes 1,\cdots,\pi^{-1} e_{n-k}\otimes 1;\\
		&   
		\pi e_{n-k+1}\otimes 1,\cdots, \pi e_n\otimes 1,
		e_1\otimes 1,\cdots,  e_{n}\otimes 1;	
		e_{k+1}\otimes 1,\cdots, e_{n-k}\otimes 1.
	\end{split}
\end{align}
The \emph{worst point} $*\in M^{\loc}_{[k]}(\BF_q)$ is defined as 
$$
\CF_i=\Pi\Lambda_i\otimes\BF_q\subset \Lambda_i\otimes \BF_q,\quad \text{for any }i\in [k],
$$
see \cite[Lemma 3.1.1]{LUO2024}.
The naive local model embeds into $\Gr(n,\Lambda_k\otimes_{ O_{F_0}} O_F )$, and we choose an open affine chart $U_k\subset \Gr(n,\Lambda_k\otimes_{O_{F_0}}O_F)$ of the worst point.
In the chart, the $\CF_k$ and $\CF_{n-k}$ are 
represented by the $2n\times n$ matrices 
 $\left(\begin{matrix}
 X\\I_n
 \end{matrix}\right)$ and  $\left(\begin{matrix}
 Y\\I_n
 \end{matrix}\right)$
with respect to the reordered basis of $\Lambda_k$ and $\Lambda_{n-k}$ respectively. The worst point corresponds to $X=0$. We define $U^{\loc}_k$ (resp. $U^{\naive}_k$) as the scheme-theoretic intersection $U_k\cap M^{\loc}_{[k]}(n-1,1)$ (resp. $U_k\cap M^{\naive}_{[k]}(n-1,1)$).
We further give the partition of the matrix
\begin{equation*}
	X=
	\left(\begin{array}{c|c}
		X_1&X_2\\
		\hline
		X_3&X_4
	\end{array}\right)=	
	\begin{tikzpicture}[>=stealth,thick,baseline]
		\matrix [matrix of math nodes,left delimiter=(,right delimiter=)](A){ 
					A		&	B	&	L\\
					C	&	D		&	M\\
					E		&	F	&	X_4\\
				};
		\filldraw[purple] (2.0,0.5) circle (0pt) node [anchor=east]{\tiny $k$};
		\filldraw[purple] (2.0,0) circle (0pt) node [anchor=east]{\tiny $k$};
		\filldraw[purple] (2.4,-0.5) circle (0pt) node [anchor=east]{\tiny $n-2k$};
		\filldraw[purple] (-0.4,1.1) circle (0pt) node [anchor=east]{\tiny $k$};
		\filldraw[purple] (0.15,1.1) circle (0pt) node [anchor=east]{\tiny $k$};
		\filldraw[purple] (0.65,1.1) circle (0pt) node [anchor=center]{\tiny	$n-2k$};
		\draw (-1,-0.225) -- (1,-0.225);
		\draw (0.25,-0.75) -- (0.25,0.75);
	\end{tikzpicture}.
\end{equation*}

We will also denote $\CX$ and $\CY$ as the corresponding matrices with respect to the standard basis. By Axiom \ref{item:LM4}, we have:
\begin{equation}\label{equ:XYisotropic}
	\CX=\left(\begin{matrix}
	D	&	M		&	C\\
	F	&	X_4	&	E\\
	B	&	L		&	A
	\end{matrix}\right),
 \quad
 \CY=\Upsilon \CX^t\Upsilon^{-1}=\left(\begin{matrix}
        A^\ad&E^\ad&-C^\ad\\
        L^\ad&X_4^\ad&-M^\ad\\
        -B^\ad&-F^\ad&D^\ad
    \end{matrix}\right),
\end{equation}
where $\Upsilon$ defined in \S \ref{sec:notation} is the matrix representing the symmetric form. 
For later use, we denote
\begin{equation}\label{eq:basis-of-filtration}
    \left(\begin{matrix}
 \CX\\I_n
 \end{matrix}\right):=(\bv_1,\cdots,\bv_n),
 \quad
 \left(\begin{matrix}
 \CY\\I_n
 \end{matrix}\right):=(\bw_1,\cdots,\bw_n).
\end{equation}

\begin{proposition}(\cite[Theorem 8.0.1]{LUO2024})\label{thm: local model}
Assume $k$ is not $\pi$-modular, the affine ring of $U^{\loc}_k$ is isomorphic to the factor ring of $ O[X]$ modulo the ideal generated by the entries of the following matrices: 
\begin{altenumerate}
\item $\CX^2=\pi_0\id$;
\item $JX_1+X_3^tHX_3+X_1^tJ=0,-JX_2+X_3^tHX_4=0,X_2^tJ+X_4^tHX_3=0,X_4^tHX_4-\pi_0 H_{n-2k}=0,$
\item $X_1JX_1^t-\pi_0J=0,X_1JX_3^t-X_2H=0,X_3JX_1^t+HX_2^t=0,X_3JX_3^t-X_4H+HX_4^t=0$
\item $\bigwedge^2(X+\pi\id)=0$, $\bigwedge^n(X-\pi\id)=0$
\item $B=B^{\ad}, C=C^{\ad},D=-2\pi I-A^{\ad}, M=\pi E^{\ad},L=-\pi F^{\ad}, X_4=X_4^{\ad},\tr(X_4)=-(n-2k-2)\pi$.
\end{altenumerate}
Here (i)-(iii) follows from \ref{item:LM1}-\ref{item:LM4}, while (iv)-(v) follows from the strengthened spin condition \cite{LUO2024}.\qed
\end{proposition}

\begin{remark}\label{rem:LM-equ-irr}
When $k$ is strongly non-special, by Theorem \ref{thm:canonical-LM} (2), the special fiber has two irreducible components $Z_1$ and $Z_2$, they are defined by $X_1=0$ and $X_4=0$, resp. For the other cases, by \cite[Remark 4.3.6]{LUO2024}, over the special fiber, we have:
\begin{altenumerate}
\item When $k=0$, we have $X_1$ is of size $0\times 0$, hence only $X_4$ left. Hence the special fiber is irreducible.
\item When $n=2m+1$ and $k=m$, we have $X_4=0$ for free, and the special fiber is smooth and irreducible.
\item When $n=2m$ and $k=m-1$, the matrix $X_4$ is of the form 
\begin{equation*}
X_4=\left(\begin{matrix}
&*\\
*&
\end{matrix}\right).
\end{equation*}
In particular, $X_1=0$ further splits into two irreducible components, which is Yu's case \cite{Yu2019}.
\end{altenumerate}
\end{remark}

\subsection{Kr\"amer model relations}
Let $k$ be non-$\pi$-modular.
We define an open subscheme $U^{\Kra}_k\subset M^{\Kra}_{[k]}$ as the following pull-back:
\begin{equation*}
\begin{aligned}
\xymatrix{
U_k^{\Kra}\ar@{^(->}[r]\ar@{}[rd]|{\square}\ar[d]&M^{\Kra}_{[k]}\ar[d]\\
U_k\ar@{^(->}[r]&\Gr(n,\Lambda_k\otimes_{O_{F_0}}O_F).
}
\end{aligned}
\end{equation*}
The right-side map is defined by the composition:
\begin{equation*}
    M^{\Kra}_{[k]}\rightarrow M^{\naive}_{[k]}\hookrightarrow\Gr(n,\Lambda_k\otimes_{O_{F_0}}O_F).
\end{equation*}
We use an analogous pull-back diagram to define  $U^{\nspl}_k$ and $U^{\spl}_k$.
In this subsection, we write down the defining equations of $U^{\Kra}_{k}$.
We will use the standard basis of $\Lambda_k$ and $\Lambda_{n-k}$ in \eqref{equ_chart:standard-basis}.

Let us consider an $R$-point in $U_k^\Kra$ that can be represented as $(\CF_k^0\subset\CF_k;\CF_{n-k}^0\subset\CF_{n-k})$. We assume $\CF_k^0$ and $\CF_{n-k}^0$ are respectively generated by two vectors whose coordinates with respect to the standard basis are
\begin{equation*}
\ba=(\begin{matrix}a_1&\cdots & a_{2n}\end{matrix})^t\in R^{2n},\quad
\bb=(\begin{matrix}b_1&\cdots & b_{2n}\end{matrix})^t\in R^{2n}.
\end{equation*}
We can in addition assume that there exist indices $\alpha'$ and $\beta'$ in the set $\{1, ..., 2n\}$ such that $a_{\alpha'}=1$ and $b_{\beta'}=1$.
\subsubsection{}
Consider the condition: $\CF_k^0\subset\CF_k$, $\CF_{n-k}^0\subset\CF_{n-k}$.
Using \eqref{eq:basis-of-filtration}, we can find $s_i,t_j\in R$ for $i,j=1,\cdots,n$ such that 
\begin{equation}\label{equ:splitting-SP1}
\ba=\sum_{i=1}^ns_i \bv_i,
\quad
\bb=\sum_{j=1}^n t_j \bw_j.
\end{equation}
Denote $\bbs=(\begin{matrix}s_1&\cdots & s_n\end{matrix})^t$ and $\bt=(\begin{matrix}t_1&\cdots & t_n\end{matrix})^t$.
Then \eqref{equ:splitting-SP1} becomes  
\begin{equation}\label{regular_splitting:1'}
\begin{pmatrix}\CX\\ I_n\end{pmatrix} \bbs=\ba,
	\quad
\begin{pmatrix}\CY\\ I_n\end{pmatrix}\bt=\bb.
\end{equation}
One immediately deduces the following relations
\begin{equation*}
	s_1=a_{n+1},\ldots, s_n=a_{2n} \quad\text{and}\quad t_1=b_{n+1},\ldots,t_n=b_{2n}.
\end{equation*}

\subsubsection{}
Next consider \ref{item:SP6}: $(\Pi-\pi)\CF_k^0=(0),(\Pi-\pi)\CF_{n-k}^0=(0)$. 
By \ref{item:LM1}, we have $\Pi\CF_i\subset\CF_i$ ($i=k,n-k$). Since we choose the standard basis, we have
\begin{equation*}
	\Pi \begin{pmatrix}
    \CX\\
    I_n
\end{pmatrix} = \begin{pmatrix}
    \pi_0 I_n\\
    \CX 
\end{pmatrix}
\quad
\text{and}
\quad
\Pi \begin{pmatrix}
    \CY\\
    I_n
\end{pmatrix} = \begin{pmatrix}
    \pi_0 I_n\\
    \CY 
\end{pmatrix}.
\end{equation*}
Assuming \ref{item:LM1}, then \ref{item:SP6} is equivalent to
\begin{equation}\label{regular_splitting:5}
	\CX \bbs=\pi \bbs, 
	\quad
	\CY \bt=\pi \bt.
\end{equation}
Hence we conclude that $\ba=\begin{pmatrix}\pi\bbs\\  \bbs\end{pmatrix}, \bb=\begin{pmatrix} \pi\bt\\  \bt\end{pmatrix}$.  
The affine chart we are working on now can be assumed as $s_\alpha=t_\beta=1$ for some $\alpha,\beta\in\{1,\cdots,n\}$.
\subsubsection{}
Next consider \ref{item:SP5}: $(\Pi+\pi)\CF_k\subset\CF_k^0, (\Pi+\pi)\CF_{n-k}\subset\CF_{n-k}^0$.
This is equivalent to the existence of additional parameters 
$\blambda=(\begin{matrix}\lambda_1&\cdots \lambda_n\end{matrix})^t$ and $\bmu=(\begin{matrix}\mu_1&\cdots \mu_n\end{matrix})^t$ such that 
\[\begin{pmatrix}
    \pi_0 I_n\\
    \CX 
\end{pmatrix}+ \pi \begin{pmatrix}
    \CX\\
    I_n 
\end{pmatrix}= \begin{pmatrix}
    \pi \bbs\\
    \bbs
\end{pmatrix}\cdot \blambda^t,\
\begin{pmatrix}
    \pi_0 I_n\\
    \CY 
\end{pmatrix}+ \pi \begin{pmatrix}
    \CY\\
    I_n 
\end{pmatrix}= \begin{pmatrix}
    \pi \bt\\
    \bt
\end{pmatrix}\cdot \bmu^t.
\]
Or equivalently
\begin{equation}\label{regular_splitting:4}
	\CX+\pi I_n=\bbs\blambda^t, 
	\quad
	\CY+\pi I_n= \bt\bmu^t.
\end{equation}
It is clear that \ref{item:SP5} and \ref{item:SP6} imply that $\CF_i$ is $\Pi$-stable for $i=k,n-k$. Now \eqref{regular_splitting:5} and \eqref{regular_splitting:4} together imply 
\begin{equation*}
	\bbs\blambda^t \bbs=2\pi \bbs,
	\quad
	\bt\bmu^t \bt=2\pi \bt.
\end{equation*}
Recall that we have assumed $a_\alpha=1$ and $b_\beta=1$.
By comparing the $\alpha$-th row and the $\beta$-th row of $\bbs\blambda^t \bbs=2\pi \bbs$ and $\bt\bmu^t \bt=2\pi \bt$, resp., we see that the above equations are equivalent to
\begin{equation}\label{equ:whatshisname}
	\blambda^t\bbs=2\pi,
	\quad
	\bmu^t\bt=2\pi.
\end{equation}
Note that \eqref{regular_splitting:4} and \eqref{equ:whatshisname} imply \eqref{regular_splitting:5}.

Finally, denote
\begin{equation*}
	\bbs=\left(\begin{matrix}\bbs_1\\ \bbs_2 \\ \bbs_3\end{matrix}\right),
	\bt=\left(\begin{matrix}\bt_1\\ \bt_2 \\ \bt_3\end{matrix}\right),
	\blambda=\left(\begin{matrix}\blambda_1\\ \blambda_2 \\ \blambda_3\end{matrix}\right),
	\bmu=\left(\begin{matrix}\bmu_1\\ \bmu_2 \\ \bmu_3\end{matrix}\right)
\end{equation*}
where all parts are of size $k\times 1, (n-2k)\times 1$ and $k\times 1$ respectively. Then \ref{item:SP3} is equivalent to:
\begin{equation}\label{equ:SP3}
	\left(\begin{matrix}\bbs_1\\ \pi\bbs_2 \\ \bbs_3\end{matrix}\right)\in \mathrm{Span}\left(\begin{matrix}\bt_1\\ \bt_2 \\ \bt_3\end{matrix}\right),
	\quad
	\left(\begin{matrix}\pi\bt_1\\ \bt_2 \\ \pi\bt_3\end{matrix}\right)\in \mathrm{Span} \left(\begin{matrix}\bbs_1\\ \bbs_2 \\ \bbs_3\end{matrix}\right).
\end{equation}

\subsection{Kr\"amer condition implies the strengthened spin condition}\label{sec:kramer-imply-ss}
In this subsection, we will show that  closed embedding $M_{[k]}^{\nspl}\hookrightarrow M_{[k]}^{\mathrm{Kra}}$ is an isomorphism.
Recall that in Definition \ref{defn:spl-model}, we only impose Axioms \ref{item:LM1}-\ref{item:LM4} to $\CF_k$ and $\CF_{n-k}$ in $M_{[k]}^{\mathrm{Kra}}$. By Proposition \ref{thm: local model}, we only have the following relations for $\CX$
\begin{altenumerate}
\item $\CX^2=\pi_0\id$;
\item $-JX_1+X_3^tHX_3+X_1^tJ=0,-JX_2+X_3^tHX_4=0,X_2^tJ+X_4^tHX_3=0,X_4^tHX_4-\pi_0 H_{n-2k}=0,$
\item $X_1JX_1^t-\pi_0J=0,X_1JX_3^t-X_2H=0,X_3JX_1^t+HX_2^t=0,X_3JX_3^t-X_4H+HX_4^t=0$.
\end{altenumerate}
Our goal is to show that the $\CX$ in $U_{[k]}^{\mathrm{Kra}}$ will satisfies the relations (iv)-(v) of Proposition \ref{thm: local model}.

Using the reordered basis \eqref{equ_chart:reordered-basis} and \eqref{regular_splitting:4}, we have
\begin{equation}\label{equ:reorder-basis}
		X_1=\begin{pmatrix}
		\bbs_3\\ \bbs_1
	\end{pmatrix}
    \begin{pmatrix}
		\blambda_3^t&\blambda_1^t
	\end{pmatrix}-\pi I_{2k},
	\quad
		X_2=\begin{pmatrix}\bbs_3\\ \bbs_1\end{pmatrix}\blambda_2^t,
	\quad
	X_3=\bbs_2\begin{pmatrix}\blambda_3^t&\blambda_1^t\end{pmatrix},
	\quad
	X_4=\bbs_2\blambda_2^t-\pi I_{n-2k}.
\end{equation}
Denote
\begin{equation}\label{eq:nabla definition}
	\bbs^{\nabla}:=\Upsilon\bbs
	\quad
	\blambda^{\nabla}:=\Upsilon\blambda
	\quad
	\bt^{\nabla}:=\Upsilon^t\bt
	\quad
	\bmu^{\nabla}:=\Upsilon^t\bmu.
\end{equation}
Apply \eqref{regular_splitting:4} to \eqref{equ:XYisotropic}, we get 
\begin{equation*}
	\bbs\blambda^t=(\Upsilon^t\bmu)(\Upsilon^t\bt)^t,
	\quad
	\bt\bmu^t=(\Upsilon\blambda)(\Upsilon\bbs)^t.
\end{equation*}
Since $s_\alpha=1$ and $t_\beta=1$, this implies
\begin{equation}\label{equ:lambda-mu}
	\blambda=\mu^{\nabla}_{\alpha}(\Upsilon^t\bt)=\mu^{\nabla}_{\alpha}\begin{pmatrix}-H\bt_3\\ H\bt_2\\ H\bt_1\end{pmatrix},
	\quad
\bmu=\lambda^{\nabla}_\beta(\Upsilon\bbs)=\lambda^{\nabla}_\beta\begin{pmatrix}H\bbs_3\\ H\bbs_2\\ -H\bbs_1\end{pmatrix}.
\end{equation}
Moreover as $\Upsilon\blambda=\mu_{\alpha}^{\nabla}\bt$, by comparing the $\beta's$ row, we get $\lambda_{\beta}^{\nabla}=\mu_{\alpha}^{\nabla}$.

\begin{lemma}\label{lem:vs-1}
	In $U^{\Kra}_k$, we have $JX_1+X_1^tJ=-2\pi I$, or equivalently
\begin{equation*}
	B=B^{\ad}, C=C^{\ad},D=-2\pi I-A^{\ad}.
\end{equation*}
\end{lemma}
\begin{proof}
By \eqref{regular_splitting:4}, we have
\begin{equation*}
	X_1+\pi I_{2k}=\begin{pmatrix}
		\bbs_1\blambda_3^t	&	\bbs_3\blambda_1^t\\
		\bbs_1\blambda_3^t	&	\bbs_1\blambda_1^t
	\end{pmatrix}=
	\begin{pmatrix}
		\bbs_3\\ \bbs_1
	\end{pmatrix}
	\begin{pmatrix}
		\blambda_3^t&\blambda_1^t
	\end{pmatrix}.
\end{equation*}	
Therefore it suffices to verify that:
\begin{equation*}
		\begin{pmatrix}
		H\bbs_1\\ -H\bbs_3
	\end{pmatrix}
	\begin{pmatrix}
		\blambda_3^t&\blambda_1^t
	\end{pmatrix}+
		\begin{pmatrix}
		\bbs_3\\ \bbs_1
	\end{pmatrix}
	\begin{pmatrix}
		-\blambda_1^tH&\blambda_3^tH
	\end{pmatrix}=\mathbf{0}.
\end{equation*}
Recall that $t_{\beta}=1$, by \eqref{equ:SP3}, we have:
\begin{equation}\label{equ:vs-1}
	\bbs_i=\left\{\begin{array}{ll}
		s_\beta\bt_i	& \text{ if } \beta\in[1,k]\cup [n-k+1,n];\\
		\pi s_\beta\bt_i	&  \text{ if } \beta\in [k+1,n-k].
	\end{array}\right.
\end{equation}
The desired relations now follow from \eqref{equ:vs-1} together with $\blambda_1=-\mu_{\alpha}^{\nabla}H\bt_3$ and $\blambda_3=\mu_\alpha^{\nabla}H\bt_1$ from \eqref{equ:lambda-mu}.
\end{proof}

\begin{lemma}\label{lem:vs-2}
In $U^{\Kra}_k$, the following relations are equivalent:
\begin{equation*}
	\blambda^t\bbs=2\pi,
	\quad
	\bmu^t\bt=2\pi,
	\quad
	\tr(X_4)=-(n-2k-2)\pi,
	\quad
	\blambda_2^t\bbs_2=2\pi,
	\quad
	\bmu_2^t\bt_2=2\pi.
\end{equation*}
\end{lemma}
\begin{proof}
	We have
	\begin{equation*}
		\blambda^t\bbs=\tr(\bbs\blambda^t)=\tr(\CX+\pi I_n)=\tr(D)+\tr(A)+\tr(X_4)+n\pi.
	\end{equation*}
	Since $\tr(A^{\ad})=\tr(A)$ and $D+A^{\ad}=-2\pi I_{k}$, we have
	\begin{equation*}
		\tr(A)+\tr(D)+n\pi=(n-2k)\pi.
	\end{equation*}
	Therefore,
	\begin{equation*}
		\blambda^t\bbs=\tr(X_4)+(n-2k)\pi.
	\end{equation*}
	We see that $\blambda^t\bbs=2\pi$ is equivalent to $\tr(X_4)=-(n-2k-2)\pi$. The same argument applies to $\bmu^t\bt=2\pi$.
	Now
	\begin{equation*}
		\blambda_2^t\bbs_2=\tr(\bbs_2\blambda_2^t)=\tr(X_4+\pi I_{n-2k}).
	\end{equation*}
	We get the last two equivalences.
\end{proof}

Finally, we have
\begin{lemma}\label{lem:vs-3}
	In $U^{\Kra}_k$, the relations
\begin{equation}\label{equ:vs-3}
	-JX_2+X_3^tHX_4=0,
	\quad 
	X_2^tJ+X_4^t H X_3=0,
\end{equation}
are equivalent to the relations
\begin{equation*}
	L=-\pi F^{\ad},\quad M=\pi E^{\ad}.
\end{equation*}
\end{lemma}
\begin{proof}
	Relation \eqref{equ:vs-3} is equivalent to:
\begin{equation*}
X_2=J^tX_3^tHX_4=J^tX_3^tX_4^{\ad}H=
\begin{pmatrix}-H\blambda_1\\H\blambda_3\end{pmatrix}
\begin{pmatrix}\bbs_2^t\blambda_2\bbs_2^tH-\pi \bbs_2^tH\end{pmatrix}.
\end{equation*}
Since $\bbs_2^t\blambda_2=\blambda^t_2\bbs_2=2\pi$,
this is equivalent to
\begin{equation*}
X_2=\pi\begin{pmatrix}-H\blambda_1\bbs_2^tH\\H\blambda_3\bbs_2^tH\end{pmatrix}=\begin{pmatrix}-\pi F^{\ad}\\ \pi E^{\ad}\end{pmatrix}.
\end{equation*}
Now the assertion follows from the fact that $\left(\begin{matrix}L\\M\end{matrix}\right)=X_2$.
\end{proof}
\begin{theorem}\label{thm:Kramer-vs-ss}
For $k$  not $\pi$-modular, we have $M^{\nspl}_{[k]}=M^{\Kra}_{[k]}$, i.e., in this case, the Kr\"amer condition implies the strengthened spin condition.
\end{theorem}
\begin{proof}
We want to show relation (iv) and (v) in Proposition \ref{thm: local model} follows from the Kr\"amer condition. We first verify this in the chart $U^{\Kra}_k$.
The relations (iv) follows when we replace $\CX$ by $\bbs\blambda^t-\pi I_n$ and $\CY$ by $\bt\bmu^t-\pi I_n$. The relations (v) follows from Lemma \ref{lem:vs-1}, \ref{lem:vs-2} and \ref{lem:vs-3}.

Next, we extend the isomorphism to the whole space. Let $M^{\wedge}_{[k]}$ be the closed subscheme of the naive local model $M^{\naive}_{[k]}$ defined in \S \ref{sec:canonical-model}, such that the filtration $(\CF_i\subset\Lambda_i\otimes_{O_{F_0}}R)_{i\in[k]}$ further satisfies the \emph{wedge condition}:
$$
\bigwedge^2(\Pi-\pi\mid \CF_i)=0,\quad \bigwedge^n(\Pi+\pi\mid \CF_i)=0.
$$
By \cite[Prop. 3.3]{HLS-basic}, we have natural closed immersion $M^{\loc}_{[k]}\subseteq M^{\wedge}_{[k]}$ which is defined by nilpotent elements. It is clear that the forgetable morphism $M^{\Kra}_{[k]}\to M^{\naive}_{[k]}$ factor through $M^\wedge_{[k]}$, hence we have the following commutative diagram:
\begin{equation*}
\begin{aligned}
\xymatrix{
M^{\nspl}_{[k]}\ar@{^(->}[r]\ar[d]&M^{\Kra}_{[k]}\ar[dl]\\
M^{\wedge}_{[k]}&
}
\end{aligned}
\end{equation*}
Moreover, this morphism is equivariant with respect to the action of the loop group. 

Let $z\in M^{\Kra}_{[k]}\setminus U_k^{\Kra}$ be a point in the special fiber with image $\tau(z)\in M_{[k]}^{\wedge}$.
Let $U^{\wedge}\subset M^{\wedge}_{[k]}$ be the open chart of $M^{\wedge}_{[k]}$ contains the worst point defined in the same way as $U^{\loc}$ in \S \ref{sec:canonical-LM}. The subschemes $U^{\loc}\subseteq U^{\wedge}$ share the same underlying topological space. In particular, the reduced locus of the special fiber $U^{\wedge}_s$ can be written as the union of Schubert cells, all of them will specialize to the worst point.

Now let $y\in U^{\wedge}_s$ be any point that lying in the same Schubert cell as $\tau(z)$. By transitivity of the loop group action on each Schubert cell, there exists an element of the loop group carrying $\tau(z)$ to $y$. Using this group action, we may move the point into the affine chart. Consequently, the isomorphism on the affine chart extends, under the group action, to the entire integral model.
\end{proof}

\begin{remark}
As suggested by the anonymous referee, using the moduli interpenetration, the forgettable functor $M^{\Kra}_{[k]}\to M_{[k]}^{\naive}$ factor through $M^{\mathrm{ZZ}}_{[k]}$ (see \S \ref{subsec:comments} for notation). By \cite{zachos2024semistable}, the scheme $M^{\mathrm{ZZ}}_{[k]}$ is flat, hence the morphism factors through $M^{\loc}_{[k]}$. This provides an alternative proof of Theorem \ref{thm:Kramer-vs-ss}.
\end{remark}

\subsection{Further simplification}\label{sec:simp}
We simplify the results obtained from \ref{item:LM1}-\ref{item:LM4} and \ref{item:SP1'}-\ref{item:SP6}.

\begin{lemma}\label{lem:vs-4}
	In $U^{\nspl}_k$, the following relations hold:
	\begin{equation*}
		X_1JX_1^t-\pi_0J=0,
		\quad
		X_1JX_3^t-X_2H=0,
		\quad
		X_3JX_1^t+HX_2^t=0,
		\quad
		X_3JX_3^t=0.
	\end{equation*}
\end{lemma}
\begin{proof}
We plug in \eqref{equ:reorder-basis}, the equalities follow from Lemma \ref{lem:vs-1} and Lemma \ref{lem:vs-3}.
\end{proof}

\begin{lemma}\label{lem:vs-5}
	In $U^{\nspl}_k$, $X_4=X_4^{\ad}$ and $\tr(\blambda_2^t\bbs_2)=2\pi$ implies $X_4^tHX_4=\pi_0 H$.
\end{lemma}
\begin{proof}
	Since $X_4=X_4^{\ad}$, the condition  $X_4^tHX_4=\pi_0 H$ is equivalent to $X_4^2=\pi_0I_{n-2k}$, by assumption. Now we have
\begin{equation*}
	X_4^2=(\bbs_2\blambda^2-\pi I)^2=\bbs_2\blambda_2^t\bbs_2\blambda_2-2\pi\bbs_2\blambda_2^t+\pi_0I=\pi_0 I_{n-2k}.\qedhere
\end{equation*}
\end{proof}
We conclude the following:

\begin{theorem}\label{regularity_relations:conclusion}
When $k$ is not $\pi$-modular, the naive splitting model is a closed subscheme of 
\begin{equation*}
	\left[\BP^{n-1}\times \BA^n\times\Gr(n,\Lambda_k\otimes_{ O_{F_0}} O_F)\right]\times
 \left[\BP^{n-1}\times \BA^n\times\Gr(n,\Lambda_{n-k}\otimes_{ O_{F_0}} O_F)\right].
\end{equation*}	
The open subscheme $U^{\nspl}_k\subset M^{\nspl}_k$ with coordinates $\bbs,\bt$ of $\BP^{n-1}$, and coordinates $\blambda,\bmu$ of $\BA^n$, is cut out by the following relations:
\begin{altenumerate}
\item $A=-\frac{1}{2}F^{\ad}E-\pi\id, B=-\frac{1}{2}F^{\ad}F,C=\frac{1}{2}E^{\ad}E,D=\frac{1}{2}E^{\ad}F-\pi\id$;
\item $M=\pi E^{\ad},L=-\pi F^{\ad}$;
\item $X_4=X_4^{\ad}$;
\item $\blambda_2^t\bbs_2=\bmu^t_2\bt_2=\pi$;
\item $\CX=\begin{pmatrix}
	D	&	M		&	C\\
	F	&	X_4	&	E\\
	B	&	L		&	A
	\end{pmatrix}=\bbs\blambda^t-\pi\id, 
	\quad
	\CY=\begin{pmatrix}
        A^\ad&E^\ad&-C^\ad\\
        L^\ad&X_4^\ad&-M^\ad\\
        -B^\ad&-F^\ad&D^\ad
    \end{pmatrix}=\bt\bmu^t-\pi\id$; 
\item $\begin{pmatrix}\bbs_1\\ \pi\bbs_2 \\ \bbs_3\end{pmatrix}\in \mathrm{Span}\begin{pmatrix}\bt_1\\ \bt_2 \\ \bt_3\end{pmatrix}$,
\quad
$\begin{pmatrix}\pi\bt_1\\ \bt_2 \\ \pi\bt_3\end{pmatrix}\in \mathrm{Span} \begin{pmatrix}\bbs_1\\ \bbs_2 \\ \bbs_3\end{pmatrix}.$\qed
\end{altenumerate}
\end{theorem}

\begin{remark}
By relation (v), the variables $\blambda$ and $\bmu$ are actually redundant.
\end{remark}

\begin{proof}
The splitting model relations \eqref{regular_splitting:4}-\eqref{equ:SP3} are exactly equations (iv)-(vi). By Theorem \ref{thm:Kramer-vs-ss}, we are left to show that the relations in \ref{item:LM1}-\ref{item:LM4} are equivalent to (i)-(iii).
	
Since $JX_1+X_1^t J=-2\pi I$, the relation $-JX_1+X_3^tHX_3+X_1^tJ=0$ implies that
\begin{equation*}
	X_1=J^{-1}X_3^tHX_3-\pi I_{2k}.
\end{equation*}
This is exactly (i). The relations (ii) and (iii) hold by the previous discussion.

Conversely, suppose we have (i)-(iii). It is clear that (i) implies
\begin{equation*}
	B=B^{\ad}, C=C^{\ad},D=-2\pi I-A^{\ad}.
\end{equation*}
This then implies $-JX_1+X_3^tHX_3+X_1^tJ=0$ by Lemma \ref{lem:vs-1} and (i). 
Next, (ii) implies $-JX_2+X_3^tHX_4=0$ and $X_2^tJ+X_4^tHX_3=0$ by Lemma \ref{lem:vs-3}.
Moreover, (iii) and (iv) imply $X_4^tHX_4-\pi_0 H_{n-2k}=0$ by Lemma \ref{lem:vs-5}.
The relations $X_1JX_1^t-\pi_0J=0,X_1JX_3^t-X_2H=0,X_3JX_1^t+HX_2^t=0$ follow from (v) by Lemma \ref{lem:vs-4}.
Finally, (iii) implies that $X_3JX_3^t-X_4H+HX_4^t=0$ since $X_3JX_3^t=0$ by Lemma \ref{lem:vs-4}.
\end{proof}

\subsubsection{}\label{sec:affine-charts}
In the remaining part of the section, we will prove Theorem \ref{thm:main}.
By Proposition \ref{prop:geometric:worst}, it suffices to prove the theorem in every affine chart $U_{\alpha\beta}\subset U^{\nspl}_k$ of the naive splitting model corresponding to $s_\alpha=t_\beta=1$.
By symmetry, it falls into four different classes:
\begin{enumerate}[label=(\roman*)]
\item $\alpha\in [k+1,n-k], \beta\in [1,k]\cup [n-k+1,n]$.
\item $\alpha,\beta\in [1,k]\cup[n-k+1,n]$.
\item $\alpha,\beta\in [k+1,n-k]$.
\item$\alpha\in [1,k]\cup [n-k+1,n],\beta\in [k+1,n-k]$.
\end{enumerate}
The notation $[a,b]$ represents a closed interval. 
Recall that we have define $Z_i\subset M^{\loc}_{[k],s}$ and $\widetilde{Z}_i\subset M^{\nspl}_{[k],s}$ in \S \ref{sec:irr-comp}. 
For each $i$, denote
\begin{equation*}
    \widetilde{Z}_{i,\alpha\beta}:=\widetilde{Z}_i\cap U_{\alpha\beta},
    \quad
    \Exc_{i,\alpha\beta}:=\Exc_i\cap U_{\alpha\beta}.
\end{equation*}

In each case, we will introduce a closed subscheme $U'_{\alpha\beta}\subset U_{\alpha\beta}$ and demonstrate it is flat, thus proving that it equals $U_{\alpha\beta}\cap M^{\spl}_{[k],s}$. It can be defined as the flat closure of the generic fiber in $U_{\alpha\beta}$ in each case, but the important part is determining its equations and prove the regularity. See \eqref{eq:chart (iv)}, \eqref{caseii-u'}, and Proposition \ref{prop: chart ii} for the defining equation of $U'_{\alpha\beta}$ in each case.

We will provide a detailed proof for case (i) as it is most interesting. The remaining cases will follow from similar computations and proofs.

\subsection{Affine charts of class (i)}\label{subsec: chart of type iv}
We consider two cases:
\begin{enumerate}
    \item When $k+1\leq \alpha\leq n-k$ and $1\leq \beta\leq k$.
    \item When $k+1\leq \alpha\leq n-k$ and $n-k+1\leq \beta\leq n$.
\end{enumerate}
The second case follows from the same calculation as the first and we will not repeat here.

We assume that $k$ is not $\pi$-modular.
In the affine chart $U_{\alpha\beta}$ with $s_\alpha=1,t_\beta=1$ for some $k+1\leq \alpha\leq n-k$ and $1\leq \beta \leq k$, we have:

\begin{proposition}\label{prop:chart-iv-relations}
The following relations among $\bbs,\blambda,\bt,\bmu$:
	\begin{itemize}
	\item $\bbs_1= s_\beta \bt_1$, $\bt_2=t_\alpha \bbs_2$, $\bbs_3= s_\beta \bt_3$, and $t_\alpha s_\beta=\pi$;
	\item $\blambda_1=-\mu_{\alpha^\vee}H \bt_3$, $\blambda_2=\mu_{\alpha^\vee} t_\alpha H \bbs_2$, $\blambda_3=\mu_{\alpha^\vee}H \bt_1$;
	\item $\bmu_1=s_\beta \mu_{\alpha^\vee} H\bt_3$, $\bmu_2=\mu_{\alpha^\vee}  H\bbs_2$, $\bmu_3=-s_\beta \mu_{\alpha^\vee} H\bt_1$.
	\end{itemize}
\end{proposition}
\begin{proof}
Since $t_\beta=1$, we have
\begin{equation*}
	\left(\begin{matrix}
		\bbs_1\\ \pi\bbs_2 \\ \bbs_3
	\end{matrix}\right)\in\colspan
	\left(\begin{matrix}
		\bt_1\\ \bt_2 \\ \bt_3
	\end{matrix}\right)
	\Rightarrow
	\left(\begin{matrix}
		\bbs_1\\ \pi\bbs_2 \\ \bbs_3
	\end{matrix}\right)=
	s_\beta\left(\begin{matrix}
		\bt_1\\ \bt_2 \\ \bt_3
	\end{matrix}\right).
\end{equation*}
Hence we have $\bbs_1=s_\beta \bt_1$ and $\bbs_3=s_\beta \bt_3$.
Similarly, since $s_\alpha=1$, we have
\begin{equation*}
	\left(\begin{matrix}
		\pi\bt_1\\ \bt_2 \\ \pi\bt_3
	\end{matrix}\right)\in\colspan
	\left(\begin{matrix}
		\bbs_1\\ \bbs_2 \\ \bbs_3
	\end{matrix}\right)
	\Rightarrow
	\left(\begin{matrix}
		\pi\bt_1\\ \bt_2 \\ \pi\bt_3
	\end{matrix}\right)=t_\alpha
	\left(\begin{matrix}
		\bbs_1\\ \bbs_2 \\ \bbs_3
	\end{matrix}\right).
\end{equation*}
Hence we have $\bt_2=t_\alpha \bbs_2$, and $t_\alpha s_\beta=\pi t_\beta=\pi$. 
To summarize, we have
\begin{equation*}
	\bbs_1=s_\beta \bt_1,\ \bt_2=t_\alpha \bbs_2,\ \bbs_3=s_\beta \bt_3.
\end{equation*}
We have $\lambda^{\nabla}_{\beta}=\lambda_{\beta^\vee}$ and $\mu^\nabla_{\alpha}=\mu_{\alpha^\vee}$. Applying this to \eqref{equ:lambda-mu}, we get the remaining relations.
\end{proof}

\subsubsection{}
Now we use the relation $C=\frac{1}{2} E^\ad E$. Since $C=\bbs_1 \blambda_3^t$ and $E^\ad=\bt_1 \bmu_2^t$, we have
\[2 \bbs_1 \blambda_3^t=\bt_1 \bmu_2^t H \bmu_2 \bt_1^t H.\]
Combining the above with $\blambda_3=\mu_{\alpha^\vee}H \bt_1$ and $\bmu_2=\mu_{\alpha^\vee}  H\bbs_2$, we get
\[2\mu_{\alpha^\vee} \bbs_1 \bt_1^t H=(\mu_{\alpha^\vee})^2 (\bbs_2^t H\bbs_2) \bt_1 \bt_1^t H.\]
Canceling $H$ on both sides of the equation, we get
\[2 \mu_{\alpha^\vee} \bbs_1 \bt_1^t =\mu_{\alpha^\vee}^2(\bbs_2^t H\bbs_2) \bt_1 \bt_1^t .\]
Since $t_\beta=1$, by comparing the $(\beta,\beta)$-th entry, we get
\begin{equation}\label{eq: 4.7.4}
    2 \mu_{\alpha^\vee} s_\beta=\mu_{\alpha^\vee}^2 \Bigl(\sum_{i=k+1}^{n-k}s_{i}s_{i^\vee}\Bigr).
\end{equation}
By the relation $\tr(\bt_2 \bmu_2^t)=2\pi$, we get
\begin{equation*}
    \mu_{\alpha^\vee} t_\alpha\Bigl(\sum_{i=k+1}^{n-k}s_{i}s_{i^\vee}\Bigr)=2\pi.
\end{equation*}
\begin{proposition}\label{prop: chart iv}
	The affine chart $U_{\alpha\beta}$ is isomorphic to the spectrum of the ring 
\begin{equation*}
	\frac{ O_F[\bt_1,\bbs_2,\bt_3,t_\alpha,s_\beta, \mu_{\alpha^\vee}]}{
		t_\beta-1,s_\alpha-1,t_\alpha s_\beta-\pi,\mu_{\alpha^\vee} t_\alpha\left(\sum_{i=k+1}^{n-k}s_{i}s_{i^\vee}\right)-2\pi, 2 \mu_{\alpha^\vee} s_\beta-\mu_{\alpha^\vee}^2  \left(\sum_{i=k+1}^{n-k}s_{i}s_{i^\vee}\right)}.
\end{equation*}
\end{proposition}
\begin{proof}
$\CX+\pi\id=\bbs\blambda^t$ implies that 
\begin{equation*}
    \CY+\pi\id=\left(\begin{matrix}
	(\bbs_3\blambda_3^t)^{\ad}&(\bbs_2\blambda_3^t)^{\ad}&-(\bbs_1\blambda_3^t)^{\ad}\\
	(\bbs_3\blambda_2^t)^{\ad}&(\bbs_2\blambda_2^t)^{\ad}&-(\bbs_1\blambda_2^t)^{\ad}\\
	-(\bbs_3\blambda_1^t)^{\ad}&-(\bbs_2\blambda_1^t)^{\ad}&(\bbs_1\blambda_1^t)^{\ad}
\end{matrix}\right)=H\left(\begin{matrix}\blambda_3\\ \blambda_2\\ -\blambda_1\end{matrix}\right)\left(\begin{matrix}\bbs_3^t&\bbs_2^t&-\bbs_1^t\end{matrix}\right)H.
\end{equation*}
On the other hand, we have
\begin{equation*}
	\bt\bmu^t=\left(\begin{matrix}\bt_1\\ \bt_2 \\ \bt_3\end{matrix}\right)\left(\begin{matrix} \bmu_1^t & \bmu_2^t & \bmu_3^t\end{matrix}\right)=\mu_{\alpha^\vee}\left(\begin{matrix}
	\bt_1\\ t_\alpha\bbs_2\\ \bt_3\end{matrix}\right)\left(\begin{matrix}s_\beta\bt_3^t&\bbs_2^t&-s_\beta \bt_1^t\end{matrix}\right)H.
\end{equation*}
Therefore Proposition \ref{regularity_relations:conclusion}(iv) holds.
Next we have
\begin{equation*}
A+\pi\id=\bbs_3\blambda_3^t=s_\beta\mu_{\alpha^\vee}\bt_3\bt_1^t H,
\quad\text{and}\quad
-\frac{1}{2}F^{\ad}E=\frac{1}{2}\mu_{\alpha^\vee}^2\bt_3(\bbs_2^tH\bbs_2)\bt_1^tH.
\end{equation*}
Using $\mu_{\alpha^\vee}^2(\bbs_2^tH\bbs_2)=2\mu_{\alpha^\vee}\bbs_{\beta}$, we get $A=-\frac{1}{2}F^{\ad}E-\pi\id$.
Similarly, we get all relations in (i). 
Next, we have
\begin{equation*}
L=\bbs_3\blambda_2^t=s_\beta\mu_{\alpha^\vee}t_\alpha \bt_3\bbs_2^tH=\pi\mu_{\alpha^\vee}\bt_3\bbs_2^tH
\end{equation*}
and
\begin{equation*}
-\pi F^{\ad}=-\pi H\blambda_1\bbs_2^tH=\pi\mu_{\alpha^\vee}\bt_3\bbs_2^tH.    
\end{equation*}
Therefore, we have $L=-\pi F^{\ad}$. Similarly, $M=\pi E^{\ad}$ holds. Finally, we have
\begin{equation*}
	X_4=\bbs_2\blambda_2^t=\mu_{\alpha^\vee}t_\alpha\bbs_2\bbs_2^tH=X_4^{\ad}.
\end{equation*}
This verifies Theorem \ref{regularity_relations:conclusion} (ii)(iii), part (iv) follows from the (3.6.5).
\end{proof}

We will assume $k$ is strongly non-special, and discuss the remaining case at the end of the section.
\begin{lemma}\label{lem:chart-iv-irr-comp}
Assume $k$ is strongly non-special, the special fiber of $U_{\alpha \beta}$ has four irreducible components defined by the following equations respectively. 
\begin{itemize}
        \item $\Exc_{1,\alpha\beta}$: $\mu_{\alpha^\vee}=0$, $s_\beta=0$.
        \item $\Exc_{2,\alpha\beta}$: $\mu_{\alpha^\vee}=0$, $t_\alpha=0$.
        \item $\widetilde{Z}_{1,\alpha\beta}$: $s_\beta=0$, $\sum_{i=k+1}^{n-k}s_{i}s_{i^\vee}=0$.
        \item $\widetilde{Z}_{2,\alpha\beta}$: $t_\alpha=0$, $2 s_\beta-\left(\sum_{i=k+1}^{n-k}s_{i}s_{i^\vee}\right)=0$.
    \end{itemize}
   They are all smooth of dimension $n-1$.
\end{lemma}
\begin{proof}
On the special fiber, $\pi=0$, and $U_{\alpha\beta,s}$ is isomorphic to the factor ring $\BF_q[\bt_1,\bbs_2,\bt_3,t_{\alpha},s_{\beta},\mu_{\alpha^\vee}]$ modulo the following relations
\begin{equation*}
    t_\alpha s_\beta=0,
    \quad
    \mu_{\alpha^\vee} t_\alpha\Bigl(\sum_{i=k+1}^{n-k}s_{i}s_{i^\vee}\Bigr)=0,
    \quad
    \mu_{\alpha^\vee} \left(2s_\beta-\Bigl(\sum_{i=k+1}^{n-k}s_{i}s_{i^\vee}\Bigr)\right)=0.
\end{equation*}

On an irreducible component, the left hand side of each equation has at least one vanishing factor. 
By considering all possible combinations, we obtain the four irreducible components stated.

We now analyze these components individually:
First, consider the locus defined by $\mu_{\alpha^\vee}=0$.
This corresponds to the exceptional locus as $\blambda=\mathbf{0}$ and $\bmu=\mathbf{0}$ by Proposition \ref{prop: chart iv}, hence $\CX=0$ and $\CY=0$.

Let $E_{\alpha \beta}$ be the closed subscheme of $U_{\alpha \beta}$ defined by $\mu_{\alpha^\vee}=0$ and $s_\beta=0$. 
Then
\begin{align*}
E_{\alpha \beta}={}&\Spec \BF_q[\bt_1,\bbs_2,\bt_3,t_\alpha]/(t_\beta-1,s_\alpha-1)   \\
\cong{}&\left\{\begin{array}{ll}
 \mathbb{A}_{\BF_q}^{n-1}&\alpha\neq \alpha^\vee;\\
 V(2s_{k+1}s_{n-k}+\cdots,2s_{(n-1)/2}s_{(n+3)/2}+1)\subseteq \mathbb{A}_{\mathbb{F}_q}^n&\alpha=\alpha^\vee.
    \end{array}\right.
\end{align*}
Since $s_\beta=0$, we have $\bbs_1=\bbs_3=0$ by Proposition \ref{prop: chart iv}. 
Hence $\lambda_{k}(\CF_{k}^0)=(0)$ and $E_{\alpha\beta}\subset \Exc_{1,\alpha\beta}$ by Proposition \ref{prop:exceptional-divisor}. Similar the component defined by $\mu_{\alpha^\vee}=0$ and $t_\alpha=0$ is a subscheme of $\Exc_{2,\alpha\beta}$.

Next, consider the subscheme of $U_{\alpha\beta,s}$ defined by 
\begin{equation*}
	s_\beta=0, \sum_{i=k+1}^{n-k}s_{i}s_{i^\vee}=0.
\end{equation*}
It is isomorphic to
\begin{equation*}
    \Spec \BF_q[\bt_1,\bbs_2,\bt_3,t_\alpha, \mu_{\alpha^\vee}]/(t_\beta-1,s_\alpha-1,\sum_{i=k+1}^{n-k}s_{i}s_{i^\vee})\simeq \BA^{n-1}_{\BF_q}.
\end{equation*}
By Proposition \ref{prop:exceptional-divisor}, $s_\beta=0$ implies $\bbs_1=\bbs_3=(0)$, hence $A=B=C=D=0$, in other words $X_1=0$. Hence the image of this component under $\tau$ lies in $Z_2$, see Remark \ref{rem:LM-equ-irr}. 
This component is not supported in the exceptional divisor, therefore it corresponds to $\widetilde{Z}_{1,\alpha\beta}$. 
Similarly, we obtain the corresponding equations for $\widetilde{Z}_{2,\alpha\beta}$.
\end{proof}

Notice that $U_{\alpha\beta}$ is not flat: the component in the integral ring defined by $\mu_{\alpha^\vee}=0$ cannot be lifted to the generic fiber since the special fiber is reduced and integrally we have $2\pi=\mu_{\alpha^\vee}t_\alpha\left(\sum_{i=k+1}^{n-k}s_i s_{i^\vee}\right)=0$.
Consider the closed subscheme $U'_{\alpha\beta}\subset U_{\alpha\beta}$ defined by 
\begin{equation}\label{equ:chart-i-spl}
	2s_\beta-\mu_{\alpha^\vee}  \Bigl(\sum_{i=k+1}^{n-k}s_{i}s_{i^\vee}\Bigr).
\end{equation}
It is isomorphic to
\begin{align}\label{eq: U' chart i}
	\nonumber U_{\alpha\beta}'={}&\Spec\frac{ O_F[\bt_1,\bbs_2,\bt_3,t_\alpha,s_\beta, \mu_{\alpha^\vee}]}{
		t_\beta-1,s_\alpha-1,t_\alpha s_\beta-\pi,\mu_{\alpha^\vee} t_\alpha\left(\sum_{i=k+1}^{n-k}s_{i}s_{i^\vee}\right)-2\pi, 2s_\beta-\mu_{\alpha^\vee}  \left(\sum_{i=k+1}^{n-k}s_{i}s_{i^\vee}\right)},\\ 
		\simeq{}&\Spec\frac{ O_F[\bt_1,\bbs_2,\bt_3,t_\alpha,s_\beta, \mu_{\alpha^\vee}]}{
		t_\beta-1,s_\alpha-1,t_\alpha s_\beta-\pi, 2s_\beta-\mu_{\alpha^\vee}  \left(\sum_{i=k+1}^{n-k}s_{i}s_{i^\vee}\right)}.
\end{align}
The special fiber $U'_{\alpha\beta,s}$ has three irreducible components, defined by
\begin{itemize}
\item $\sum_{i=k+1}^{n-k}s_{i}s_{i^\vee}=0$,
\item $t_\alpha=0$,
\item $\mu_{\alpha^\vee}=0$,
\end{itemize}
which corresponds to $\widetilde{Z}_{1,\alpha\beta}$, $\widetilde{Z}_{2,\alpha\beta}$, and $\Exc_{1,\alpha\beta}$ respectively.
Therefore, it is the union  
\begin{equation*}
	U'_{\alpha\beta,s}=\widetilde{Z}_{1,\alpha\beta}\cup \widetilde{Z}_{2,\alpha\beta}\cup \Exc_{1,\alpha\beta}.
\end{equation*}
By substituting $s_\beta$ (as $2$ is invertible), we have the isomorphism
\begin{equation}\label{eq:chart (iv)}	U'_{\alpha\beta}\simeq\Spec\frac{ O_F[\bt_1,\bbs_2,\bt_3,t_\alpha,\mu_{\alpha^\vee}]}{t_\beta-1,s_\alpha-1,\mu_{\alpha^\vee}t_\alpha\left(\sum_{i=k+1}^{n-k}s_i s_{i^\vee}\right)-2\pi}.
\end{equation}
It is Zariski locally isomorphic to an affine ring of the form
\begin{equation*}
    \Spec\frac{O_F[x_1,x_2,x_3,y_1,\cdots,y_\ell]}{x_1x_2x_3-2\pi}.
\end{equation*}
Therefore it is flat over $O_F$ with semi-stable reduction, and  $U'_{\alpha\beta}=U_{\alpha\beta}\cap M^{\spl}_{[k]}\subset M^{\nspl}_{[k]}$.
To summarize, we have
\begin{theorem}\label{thm:chart-iv}
With notation as above,
\begin{altenumerate}
\item the affine chart $U_{\alpha\beta}$ is not flat over $ O_F$. Moreover, 
\begin{equation*}
    U'_{\alpha\beta}=U_{\alpha\beta}\cap M^{\spl}_{[k]}
\end{equation*}
is flat over $O_F$ with semi-stable reduction.
\item Over the special fiber, we have
\begin{align*}
    U_{\alpha\beta}={}&\widetilde{Z}_{1,\alpha\beta}\cup\widetilde{Z}_{2,\alpha\beta}\cup \Exc_{1,\alpha\beta}\cup \Exc_{2,\alpha\beta};\\
    U_{\alpha\beta,s}'={}&\widetilde{Z}_{1,\alpha\beta}\cup\widetilde{Z}_{2,\alpha\beta}\cup \Exc_{1,\alpha\beta}.
\end{align*}
All four irreducible components are isomorphic to the affine space $\BA^{n-1}_{\BF_q}$.\qed
\end{altenumerate}
\end{theorem}

\begin{remark}\label{rmk:chart-i-remain}
We have considered affine charts of class (i) when the index $k$ is strongly non-special.
It remains to consider the case when $n=2m+1$ and $k=m$ and the case when $n=2m$ and $k=m-1$.

Case $1$: $n=2m+1$ and $k=m$.
In this case, $k+1=n-k=m+1=\alpha$. 
According to Proposition \ref{prop: chart i }, the summation $\sum_{i=k+1}^{n-k}s_is_{i^\vee}=s_{m+1}^2$, which equals $1$ since $s_\alpha-1=0$.
Consequently, the affine ring becomes:
\begin{equation*}
	U_{\alpha\beta}=\Spec\frac{ O_F[\bt_1,\bbs_2,\bt_3,t_\alpha,s_\beta, \mu_{\alpha^\vee}]}{
		t_\beta-1,s_\alpha-1,t_\alpha s_\beta-\pi,\mu_{\alpha^\vee} t_\alpha-2\pi, 2 \mu_{\alpha^\vee} s_\beta-\mu_{\alpha^\vee}^2 }.
\end{equation*}
Since $\sum_{i=k+1}^{n-k}s_is_{i^\vee}=1$, we only have irreducible components $\Exc_{1,\alpha\beta}$, $\Exc_{2,\alpha\beta}$ and $\widetilde{Z}_{2,\alpha\beta}$ in Lemma \ref{lem:chart-iv-irr-comp}. 
The closed subscheme defined in \eqref{equ:chart-i-spl} remains valid, and the subsequent arguments still hold.
Therefore, we obtain a similar statement as Theorem \ref{thm:chart-iv} except we only have $\widetilde{Z}_{2,\alpha\beta}$. 
This aligns with Smithling's discussion \cite[Proposition 3.20]{Smithling2015}, see also \cite[Remark 4.3.6]{LUO2024}.

Case $2$: $n=2m$ and $k=m-1$. 
Here, $k+1=m$ and $n-k=m+1$, and $\alpha=m$ or $m+1$. 
In Proposition \ref{prop: chart i }, the summation $\sum_{i=k+1}^{n-k}s_is_{i^\vee}=2s_{m}s_{m+1}$. 
Since $s_\alpha=1$, this term remains irreducible, and the discussions for the strongly non-special cases still apply.
However, in the affine charts where $s_m=1$ and $s_{m+1}=1$, the irreducible component $\widetilde{Z}_{2,\alpha\beta}$ correspond to $s_{m+1}=0$ and $s_m=0$ respectively.
These are different irreducible components in the local model and splitting model.
Only the affine charts with $\alpha\in[1,k]\cup [n-k+1,n]$
includes both components.
This relates to the fact that in this case, $Z_2$ further decompose into two irreducible components, see \cite[Remark 4.3.6]{LUO2024}.    
\end{remark}

\subsection{Affine charts of class (ii)}\label{subsec: affine charts of class (i)}
We consider the case when $1\leq \alpha,\beta\leq k$, the other cases in this class follow from the same computation. 
At this point, we do not impose any restrictions on the value of $k$.

\begin{proposition}\label{prop:chart-i-relations}
	We have the following relations among $\bbs,\blambda,\bt,\bmu$:
\begin{itemize}
\item $\bt_1=t_\alpha \bbs_1$, $\bt_2=\pi t_\alpha \bbs_2$, $\bt_3=t_\alpha \bbs_3$. In particular $t_\alpha s_\beta=1$;
\item $\blambda_1=-\lambda_{\beta^\vee}t_{\alpha}H\bbs_3, \blambda_2=\pi\lambda_{\beta^\vee}t_{\alpha}H\bbs_2,
\blambda_3=\lambda_{\beta^\vee}t_\alpha H\bbs_1$. In particular, $\lambda_{\alpha^\vee}=\lambda_{\beta^\vee}t_\alpha$.
\item $\bmu_1=\lambda_{\beta^\vee}H\bbs_3,
\bmu_2=\lambda_{\beta^\vee}H\bbs_2,
\bmu_3=-\lambda_{\beta^\vee}H\bbs_1$.\qed
\end{itemize}
\end{proposition}

Apply $C=-\frac{1}{2}E^{\ad}E$ and $\blambda_2^t\bbs_2=2\pi$ to Proposition \ref{prop:chart-i-relations}, we get
\begin{equation}\label{eq: 4.4.1}
    2\lambda_{\alpha^\vee}=	 \lambda_{\alpha^\vee}^2\Bigl(\sum_{i=k+1}^{n-k} s_{i}s_{i^\vee}\Bigr)
    \quad\text{and}\quad \pi\lambda_{\alpha^\vee}\Bigl(\sum_{i=k+1}^{n-k} s_{i}s_{i^\vee}\Bigr)=2\pi.
\end{equation}

By the same computation as in the proof of Proposition \ref{prop: chart iv}, one deduces:
\begin{proposition}\label{prop: chart i }
	The affine chart $U_{\alpha\beta}$ is isomorphic to the spectrum of the ring 
	\begin{equation*}
		U_{\alpha\beta}\simeq\Spec \frac{ O_F[s_1,\cdots,s_n,\lambda_{\alpha^\vee},s_\beta^{-1}]}
		{s_\alpha-1,\lambda_{\alpha^\vee}^2\left(\sum_{i=k+1}^{n-k} s_{i}s_{i^\vee}\right)-2\lambda_{\alpha^\vee},\pi\lambda_{\alpha^\vee}\left(\sum_{i=k+1}^{n-k} s_{i}s_{i^\vee}\right)-2\pi}.\eqno\qed
	\end{equation*}
\end{proposition}

We consider a case where $k$ is not in Yu's case, i.e., not where $n=2m$, $k=m-1$.
Recall we have the forgetful functor
\begin{equation*}
	\tau:M_{[k]}^{\nspl}\longrightarrow M_{[k]}^{\loc}.
\end{equation*}
Over the special fiber, we have $\CY=\bt\bmu^t$. 
By Proposition \ref{prop:chart-i-relations} we have $\bt_2=\mathbf{0}$ which implies $\NExc_{\alpha\beta}=\Exc_{2,\alpha\beta}$ by Proposition \ref{prop:exceptional-divisor}.
We also get $X_4=\bt_2\bmu_2^t=0$. Hence, by Remark \ref{rem:LM-equ-irr}, we have $\tau(U_{\alpha\beta})\subset Z_2$, see Remark \ref{rem:LM-equ-irr}.

Observe that $\lambda_{\alpha^\vee}$ and $\lambda_{\alpha^\vee}\left(\sum_{i=k+1}^{n-k} s_{i}s_{i^\vee}\right)-2$ cannot both be zero on a geometric point of $U_{\alpha \beta}$, as this would imply $-2=0$.
Hence we have the decomposition 
\[U_{\alpha\beta}=U'_{\alpha \beta}\sqcup F_{\alpha \beta},\] 
where $F_{\alpha \beta}$ is the open subscheme where $\lambda_{\alpha^\vee}(\sum_{i=k+1}^{n-k} s_{i}s_{i^\vee})-2$ is a unit and $U'_{\alpha \beta}$ is the open subscheme where $\lambda_{\alpha^\vee}$ is a unit. 
By \eqref{eq: 4.4.1}, we have $\lambda_{\alpha^\vee}=0$ on $F_{\alpha \beta}$. Hence 
\[F_{\alpha \beta}=
\Spec \frac{ O_F[s_1,\cdots,s_n,\lambda_{\alpha^\vee},s_\beta^{-1}]}
	{s_\alpha-1,\lambda_{\alpha^\vee},-2\pi}\cong \mathbb{A}_{\BF_q}^{n-1}. 
\]
Since $F_{\alpha \beta}$ is supported on the special fiber of $U_{\alpha \beta}$, we see that $U_{\alpha \beta}$ is not flat over $\mathcal{O}_F$. 

Since $\lambda_{\alpha^\vee}$ is a unit on $U'_{\alpha \beta}$, we have $ U'_{\alpha\beta}=\Spec R_{\alpha \beta}$ where
 \[ R_{\alpha \beta}= \frac{ O_F[s_1,\cdots,s_n,\lambda_{n+1-\alpha},s_\beta^{-1}]}
		{s_\alpha-1,\lambda_{\alpha^\vee}\left(\sum_{i=k+1}^{n-k} s_{i}s_{i^\vee}\right)-2}.\]
Therefore, $U_{\alpha\beta}'$ and $F_{\alpha\beta}$ are the closed subschemes of $U_{\alpha\beta}$ defined by the following two ideals respectively:
\begin{equation*}
	I_{\alpha\beta,1}=\left(\lambda_{\alpha^\vee}\Bigl(\sum_{i=k+1}^{n-k}s_is_{i^\vee}\Bigr)-2\right),
	\quad
	I_{\alpha\beta,2}=\left(\lambda_{\alpha^\vee},\pi\right).
\end{equation*}

Over the special fiber, $\tau(U_{\alpha\beta,s}')\subset M^{\loc}_{[k],s}$ does not pass through the worst point: indeed,
$\lambda_{\alpha^\vee}$ is invertible which implies $\CX\neq 0$ since $s_\alpha=1$. Therefore, $U_{\alpha\beta,s}'=\widetilde{Z}_{2,\alpha\beta}$.
On the other hand, the image $\tau(F_{\alpha\beta,s})$ is the worst point as $\blambda=0$ on $\tau(F_{\alpha\beta})$. Therefore, $F_{\alpha\beta}=F_{\alpha\beta,s}=\Exc_{2,\alpha\beta}$.

\begin{theorem}\label{thm:chart-i}
With notation as above and $k$ is not in Yu's case.
\begin{altenumerate}
\item The affine chart $U_{\alpha\beta}$ is a disjoint union
 \[U_{\alpha\beta}=F_{\alpha \beta} \sqcup U'_{\alpha \beta},\]
 where $F_{\alpha \beta}\cong \mathbb{A}_k^{n-1}$ and 
\begin{equation}\label{caseii-u'}
U'_{\alpha\beta}\simeq\Spec \frac{ O_F[s_1,\cdots,s_n,\lambda_{\alpha^\vee},s_\beta^{-1}]}
		{s_\alpha-1,\lambda_{\alpha^\vee}\left(\sum_{i=k+1}^{n-k} s_{i}s_{i^\vee}\right)-2}.
\end{equation}
In particular $U_{\alpha\beta}$ is not flat over $\mathcal{O}_F$. Moreover 
\begin{equation*}
	U_{\alpha\beta}'=U_{\alpha\beta}\cap M^{\spl}_{[k]}
\end{equation*}
is flat and smooth. 
\item Over the special fiber, we have identifications 
\begin{equation*}
	F_{\alpha\beta,s}=\Exc_{2,\alpha\beta},
	\quad
	U'_{\alpha\beta,s}=\widetilde{Z}_{2,\alpha\beta}.
\end{equation*}
\end{altenumerate}

\end{theorem}
\begin{proof}
It is clear that $U'_{\alpha\beta}$ is flat, hence $U_{\alpha\beta}'=U_{\alpha\beta}\cap M^{\spl}_{[k]}$. The smoothness follows from the computation of the Jacobian matrix.
\end{proof}
\begin{remark}
In Yu's case, i.e., when $n=2m$ and $k=m-1$, we have $\sum_{i=k+1}^{n-k}s_is_{i^\vee}=2s_ms_{m+1}$. Consequently, $\widetilde{Z}_{2,\alpha\beta}$ further decompose into irreducible components defined by $s_m=0$ and $s_{m+1}=0$, respectively. 
In this case, $U'_{\alpha\beta}$ remains flat, but the special fiber now consists of two smooth irreducible components. This corresponds to the two irreducible components of $Z_2$ mentioned in Remark \ref{rmk:chart-i-remain}.
\end{remark}

\subsection{Affine charts of class (iii)}\label{subsec: chart of type ii}
In this case, we have $s_\alpha=1,t_\beta=1$ for some $k+1\leq \alpha,\beta \leq n-k$. At this point, we do not impose any restrictions on $k$.

\begin{proposition}\label{prop:chart-ii-relations}
	We have the following relations among $\bbs,\blambda,\bt,\bmu$:
	\begin{itemize}
	\item $\bbs_1=\pi s_\beta \bt_1$, $\bbs_2=s_\beta \bt_2$, $\bbs_3=\pi s_\beta \bt_3$, in particular $t_\alpha s_\beta=1$;
	\item $\blambda_1=-\lambda_{\beta^\vee}H\bt_3, \blambda_2=\lambda_{\beta^\vee}H\bt_2,
	\blambda_3=\lambda_{\beta^\vee}H\bt_1.$.
	\item $\bmu_1=\lambda_{\beta^\vee}H\bbs_3,
	\bmu_2=\lambda_{\beta^\vee}H\bbs_2,
	\bmu_3=-\lambda_{\beta^\vee}H\bbs_1$, in particular, $\mu_{\beta^\vee}=\lambda_{\beta^\vee}s_\beta$.
	\end{itemize}
\end{proposition}

\subsubsection{}\label{subsubsec: 4.5.5}
Applying Proposition \ref{prop:chart-ii-relations} to $\bmu_2^t\bt_2=2\pi$, we get
\begin{equation*}
	2\pi=\sum_{i=k+1}^{n-k}t_i\mu_{i^\vee}=\mu_{\beta^\vee}\Bigl(\sum_{i=k+1}^{n-k}t_i t_{i^\vee}\Bigr).
\end{equation*}

By the same computation as in the proof of Proposition \ref{prop: chart iv}, one deduces:
\begin{proposition}\label{prop: chart ii}
	The affine chart $U_{\alpha\beta}$ is isomorphic to the spectrum of the ring 
\begin{equation*}
		U_{\alpha\beta}\simeq\Spec\frac{ O_F[t_1,\cdots,t_n,\mu_{\beta^\vee},t_\alpha^{-1}]}
		{t_\beta-1,\mu_{\beta^\vee}\left(\sum_{i=k+1}^{n-k}t_{i}t_{i^\vee}\right)-2\pi}.\eqno\qed
\end{equation*}
\end{proposition}
Assume $k$ is not almost $\pi$-modular, i.e., not where $n=2m+1$ and $k=m$.
Denote $U_{\alpha\beta}=\Spec R$. The special fiber $\Spec R_s$ consists of two irreducible components defined by 
\begin{equation*}
E_{\alpha\beta}:\mu_{\beta^\vee}=0, \quad \widetilde{Z}_{1,\alpha\beta}':\sum_{i=k+1}^{n-k}t_{i}t_{i^\vee}=0.
\end{equation*}
It is straightforward to see that $E_{\alpha\beta}$ and $\widetilde{Z}_{1,\alpha\beta}'$ are smooth over $k$.

By Proposition \ref{prop:chart-ii-relations} we have $\bbs_1=\mathbf{0}, \bbs_3=\mathbf{0}$ on the special fiber, hence $\NExc_{\alpha\beta}\subset \Exc_{1,\alpha\beta}$ by Proposition \ref{prop:exceptional-divisor}.
Recall that $\CX=\bbs\blambda^t$, hence we have $A=B=C=D=0$.
Therefore, we have $\tau(U_{\alpha\beta})\subset Z_1$, see Remark \ref{rem:LM-equ-irr}.
Moreover, since $\mu_{\beta^\vee}=0$ implies $\CX=0$, the image $\tau(E_{\alpha\beta})$ is the worst point. Hence
\begin{equation*}
	\widetilde{Z}_{1,\alpha\beta}'=\widetilde{Z}_{1,\alpha\beta},
	\quad
	E_{\alpha\beta}=\Exc_{1,\alpha\beta}.
\end{equation*}

\begin{theorem}
When $k$ is not almost $\pi$-modular.
\begin{altenumerate}
\item The affine chart $U'_{\alpha\beta}:=U_{\alpha\beta}$ is flat and regular, with semi-stable reduction.
\item The special fiber equals
\begin{equation*}
    U'_{\alpha\beta,s}=\widetilde{Z}_{1,\alpha\beta}\cup \Exc_{1,\alpha\beta}.\eqno\qed
\end{equation*}
\end{altenumerate}
\end{theorem}
\begin{remark}
When $k$ is almost $\pi$-modular, we have $\sum t_i t_{i^\vee}=1$, hence we only have $E_{\alpha\beta}$ in this chart.
\end{remark}

\subsection{Affine charts of class (iv)}
\label{subsec: chart of type iii}
We consider the case when $1\leq \alpha\leq k$ and $k+1\leq \beta \leq n-k$, the other cases in this class follow from the same computation.
In this case, as $t_\beta=1$, we have
\begin{equation*}
\left(\begin{matrix}
	\bbs_1\\ \pi\bbs_2 \\ \bbs_3
\end{matrix}\right)\in\colspan
\left(\begin{matrix}
	\bt_1\\ \bt_2 \\\bt_3
\end{matrix}\right)
\text{ implies }	
\left(\begin{matrix}
	\bbs_1\\ \pi\bbs_2 \\ \bbs_3
\end{matrix}\right)=\pi s_\beta
\left(\begin{matrix}
	\bt_1\\ \bt_2 \\\bt_3
\end{matrix}\right).
\end{equation*}
Similarly, since $s_\alpha=1$, we have
\begin{equation*}
\left(\begin{matrix}
\pi\bt_1\\ \bt_2 \\ \pi\bt_3	
\end{matrix}\right)\in\colspan
\left(\begin{matrix}
\bbs_1\\ \bbs_2 \\ \bbs_3	
\end{matrix}\right)
\text{ implies }
\left(\begin{matrix}
\pi\bt_1\\ \bt_2 \\ \pi\bt_3	
\end{matrix}\right)=\pi t_\alpha
\left(\begin{matrix}
\bbs_1\\ \bbs_2 \\ \bbs_3	
\end{matrix}\right).
\end{equation*}
In particular, we get $\pi t_\alpha s_\beta=t_\beta=1$, hence $\pi$ is invertible in class (iii). Therefore, $U_{\alpha\beta}$ is contained in the generic fiber, hence is smooth, hence regular. In this case we define $U'_{\alpha\beta}:=U_{\alpha\beta}$.

Combining all the discussions, we finish the proof of Theorem \ref{thm:main}.\qed

\section{Resolution of singularity}\label{sec:resolution}
In this section we assume $k$ is not $\pi$-modular.
The main result of this section is Theorem \ref{thm:blowup} below. After proving the theorem, we will identify the exceptional divisor $\Exc_1$ with a blow-up of $\mathbb{P}^{n-1}_{\BF_q}$ in Proposition \ref{prop:exc 1 shape}.
\begin{theorem}\label{thm:blowup}
For $k$ not $\pi$-modular, the morphism 
\begin{equation*}
	\tau:M^\spl_{[k]} \longrightarrow M^\loc_{[k]}
\end{equation*}
is the blow-up of $M^\loc_{[k]}$ along the worst point.  
\end{theorem}

\subsection{Defining equations of the blow-up}
We define $M^{\mathrm{bl}}_{[k]}$ to be the blow up of $M^{\loc}_{[k]}$ over the worst point $*$. 
By Theorem \ref{thm:main} and the universal property of the blow up, we have a unique natural morphism $M^{\spl}\rightarrow M^{\mathrm{bl}}$ defined over $M^{\loc}_{[k]}$.
By Proposition \ref{prop:geometric:worst}, $\tau$ is an isomorphism outside the worst point $*$. 
Therefore, it suffices to consider an affine neighborhood of $*$ as in Proposition \ref{thm: local model}. We denote $U^{\mathrm{bl}}$ to be the blow-up of $U^{\loc}$ over the worst point.

Denote by $R=R^{\loc}$ the affine ring of $U^{\loc}$, which is the quotient algebra of $ O_F[X]$ modulo the relations in Proposition \ref{thm: local model}. The worst point $*$ is defined by the ideal $I=(X,\pi)\subset R$. Then the blow-up of $U^{\mathrm{bl}}$ along $*$ is 
\begin{equation*}
    U^{\mathrm{bl}}=\mathrm{Proj}_{R}\, R^{\mathrm{bl}}, \text{ where } R^{\mathrm{bl}}=\bigoplus_{d\geq 0} I^d\text{ with }I^0=R.
\end{equation*}

We define $\mathfrak{X}$ be another $n\times n$ matrix of 
independent variables:
\begin{equation}\label{equ:fkXY}
\mathfrak{X}=\left(\begin{matrix}
	\mathfrak{D}	&	\mathfrak{M}		&	\mathfrak{C}\\
	\mathfrak{F}	&	\mathfrak{X}_4	&	\mathfrak{E}\\
	\mathfrak{B}	&	\mathfrak{L}		&	\mathfrak{A}
	\end{matrix}\right),
 \quad
\mathfrak{Y}=\Upsilon\mathfrak{X}^t\Upsilon^t=\left(\begin{matrix}
    \mathfrak{A}^\ad & \mathfrak{E}^\ad & -\mathfrak{C}^\ad\\
        \mathfrak{L}^\ad & \mathfrak{X}_4^\ad & -\mathfrak{M}^\ad\\
        -\mathfrak{B}^\ad & -\mathfrak{F}^\ad & \mathfrak{D}^\ad
    \end{matrix}\right).
\end{equation}
We think of the variables in $\mathfrak{X}$ as variables in $I^1\subset R^{\mathrm{bl}}$.
Parallel to \S \ref{sec:regularity}, we denote
\begin{equation*}
    \fkX_1=\left(\begin{matrix}\fkA&\fkB\\\fkC&\fkD\end{matrix}\right),
    \fkX_2=\left(\begin{matrix}\fkL\\\fkM\end{matrix}\right),
    \fkX_3=\left(\begin{matrix}\fkA&\fkB\end{matrix}\right).
\end{equation*}
Define $\tilde{R}=R[\mathfrak{X},\varpi]/\mathfrak{I}$ where $\mathfrak{I}$ is the ideal of $R[\mathfrak{X},\varpi]$ (homogeneous with respect to $\mathfrak{X}$ and $\varpi$) generated by 
\begin{equation}\label{eq:blow up relations}
\begin{aligned}
&\CX_{ij}\varpi- \mathfrak{X}_{ij} \pi,\ \mathfrak{X}_{ij} \CX_{k \ell}-\mathfrak{X}_{k \ell} \CX_{ij}\\ 
        & \mathfrak{X}_1^2+\mathfrak{X}_2 \mathfrak{X}_3-\varpi^2 I,\ 
        \mathfrak{X}_1 \mathfrak{X}_2+\mathfrak{X}_2 \mathfrak{X}_4,\ 
        \mathfrak{X}_3 \mathfrak{X}_1+\mathfrak{X}_4 \mathfrak{X}_3,\ 
        \mathfrak{X}_4^2+\mathfrak{X}_3 \mathfrak{X}_2-\varpi^2 I, \\  
        & \wedge^2(\mathfrak{X}+\varpi\id),\ \wedge^n(\mathfrak{X}-\varpi\id), \\ 
        & \mathfrak{B}-\mathfrak{B}^{ad},\ \mathfrak{C}-\mathfrak{C}^{ad},\ \mathfrak{D}+2\varpi I+\mathfrak{A}^{ad},\ \mathfrak{M}-\pi \mathfrak{E}^{ad},\ \mathfrak{L}+\pi \mathfrak{F}^{ad},\ \mathfrak{X}_4-\mathfrak{X}_4^{ad},\\ 
        & \mathfrak{A}+\frac{1}{2}F^{\ad}\mathfrak{E}+\varpi\id,\ \mathfrak{B}+\frac{1}{2}F^{\ad}\mathfrak{F},\ \mathfrak{C}-\frac{1}{2}E^{\ad}\mathfrak{E},\ \mathfrak{D}-\frac{1}{2}E^{\ad}\mathfrak{F}+\varpi\id\\ 
        &  \tr(\mathfrak{X}_4)=-(n-2k-2)\varpi. 
\end{aligned}
\end{equation}

Let $\wt{U}:=\mathrm{Proj}\, \tilde{R}$. By definition $U^{\mathrm{bl}}$ is a closed subscheme of $\wt{U}$. 
Suppose $1\leq \gamma,\eta\leq n$. Define $V_{\gamma \eta}$ to be the open subscheme of $\tilde{U}$ where $\mathfrak{X}_{\gamma \eta}+ \varpi \delta_{\gamma \eta}\neq 0$ where $\delta_{\gamma\eta}$ is the Kronecker delta.
\begin{lemma}\label{lem:cover-bl}
$\wt{U}=\proj R[\mathfrak{X},\varpi]/\mathfrak{I}$ can be covered by $\{V_{\gamma \eta}\mid 1\leq \gamma, \eta \leq n\}$.
\end{lemma}
\begin{proof}
    By the relation $\tr(\mathfrak{X}_4)=-(n-2k-2)\varpi$, the lemma follows from the same argument as that of \cite[Lemma A.1]{Shi2022}.
\end{proof}

On $V_{\gamma \eta}$ we can assume $\mathfrak{X}_{\gamma \eta}+ \varpi \delta_{\gamma \eta}=1$. 
Define $\bbs=(s_1,\ldots,s_n)^t$ to be the $\eta$-th column of $\mathfrak{X}+\varpi I$, i.e.,
\begin{equation*}
    s_i=\mathfrak{X}_{i \eta}+ \varpi\delta_{i \eta}.
\end{equation*}
Then $s_\gamma=1$. 
Also define
\begin{equation*}
  u_j=\mathfrak{X}_{\gamma j}+\varpi \delta_{\gamma j},\ \text{and } \lambda_j=X_{\gamma j}+\pi \delta_{\gamma j}.  
\end{equation*}
In particular $u_\eta=1$.

\begin{lemma}\label{lem:X s u}
We have
\begin{equation}\label{eq:X s u}
    \mathfrak{X}+\varpi I=\bbs \mathbf{u}^t,
\end{equation}
and
\begin{equation}\label{eq:X s lambda}
\CX+\pi I=\bbs \blambda^t.
\end{equation}
\end{lemma}
\begin{proof}
By the first line of relations in \ref{eq:blow up relations}, we have
\begin{equation}\label{eq:mathfrak X and X}
    (\CX_{ij}+\pi \delta_{ij})(\mathfrak{X}_{k \ell}+\varpi\delta_{k\ell})=(\CX_{k \ell}+\pi\delta_{k\ell})(\mathfrak{X}_{ij}+\varpi \delta_{ij}).
\end{equation}
Relation \eqref{eq:X s u} follows directly from the wedge condition for $\mathfrak{X}$ as
\[(\mathfrak{X}_{ij}+\varpi\delta_{ij})(\mathfrak{X}_{\gamma \eta}+\varpi\delta_{\gamma \eta})-(\mathfrak{X}_{i\eta}+\varpi\delta_{i\eta})(\mathfrak{X}_{\gamma j}+\varpi\delta_{\gamma j})=0,\]
and $\mathfrak{X}_{\gamma \eta}+\varpi\delta_{\gamma \eta}=1$ on $V_{\gamma \eta}$. Now by \eqref{eq:mathfrak X and X}, we have
\[\CX_{ij}+\pi \delta_{ij}=(\CX_{\gamma \eta}+\pi\delta_{\gamma\eta}) (\mathfrak{X}_{ij}+\varpi\delta_{ij}).\]
Combining with \eqref{eq:X s u}, we have
\[\CX+\pi I=(\CX_{\gamma \eta}+\pi\delta_{\gamma\eta}) \bbs \bu^t.\]
So in order to show \eqref{eq:X s lambda}, it suffices to show that 
\[(\CX_{\gamma \eta}+\pi\delta_{\gamma\eta}) \bu=\blambda,\]
or equivalently
\[(\CX_{\gamma \eta}+\pi\delta_{\gamma\eta})(\mathfrak{X}_{\gamma j}+\varpi \delta_{\gamma j})=\CX_{\gamma j}+\pi \delta_{\gamma j}.\]
But this again follows from \eqref{eq:mathfrak X and X}.
\end{proof}

Similarly, by the relation
\[\mathfrak{Y}+\varpi I=\Upsilon(\mathfrak{X}+\varpi I)^t \Upsilon^t, \]
we know that if $\mathfrak{X}_{\gamma \eta}+ \varpi \delta_{\gamma \eta}=1$, then
\[\mathfrak{Y}_{\eta^\vee \gamma^\vee}+\varpi \delta_{\eta^\vee\gamma^\vee}=\pm 1.\]
Define
\begin{equation*}
    t_i=\begin{cases}
        \mathfrak{Y}_{i \gamma^\vee}+ \varpi\delta_{i \gamma^\vee} &\text{if } \mathfrak{Y}_{\eta^\vee \gamma^\vee}+\varpi \delta_{\eta^\vee\gamma^\vee}=1,\\
        -(\mathfrak{Y}_{i \gamma^\vee}+ \varpi\delta_{i \gamma^\vee}) &\text{if } \mathfrak{Y}_{\eta^\vee \gamma^\vee}+\varpi \delta_{\eta^\vee\gamma^\vee}=-1.
    \end{cases}
\end{equation*}
Also define
\begin{equation*}
  v_j=\mathfrak{Y}_{\eta^\vee j}+\varpi \delta_{\eta^\vee j},\ \text{and } \mu_j=X_{\eta^\vee j}+\pi \delta_{\eta^\vee j}.  
\end{equation*}
Then $t_{\eta^\vee}=1$ and $v_{\gamma^\vee}=\pm 1$. By the same proof as that of Lemma \ref{lem:X s u}, we have
\begin{equation}\label{eq:Y t v}
    \mathfrak{Y}+\varpi I=\bt \mathbf{v}^t,
\end{equation}
and
\begin{equation}\label{eq:Y t mu}
    \CY+\pi I=\bt \bmu^t.
\end{equation}

\subsection{Proof of the main theorem}
In this section, we will prove Theorem \ref{thm:blowup}.

Recall from Theorem \ref{regularity_relations:conclusion} that $U^{\spl}$ is the open subscheme $\tau^{-1}(U^{\loc})$ of $M^\spl_{[k]}$, and we have a natural morphism
\begin{equation}\label{eq:blow up morphism f}
   f:U^{\spl}\to U^{\mathrm{bl}}\hookrightarrow \wt{U}, 
\end{equation}
where $U^{\spl}\to U^{\mathrm{bl}}$ is the unique morphism obtained from the universal property of the blow-up $U^{\mathrm{bl}}$. Recall \S \ref{sec:affine-charts} that $U^{\spl}$ is covered by $\{U'_{\alpha\beta}\}$ and 
Lemma \ref{lem:cover-bl}, that $\wt{U}$ is covered by $\{V_{\gamma\eta}\}$. 
The main result of this section is the following.

\begin{theorem}\label{thm:bl-charts}
\begin{altenumerate}
\item For charts of cases (i)-(iii), the morphism $f$ induces isomorphisms
\begin{equation*}
    f_{\gamma\eta}:U'_{\gamma\eta^\vee}\rightarrow V_{\gamma\eta}
\end{equation*}
for all indices $\gamma$ and $\eta$.
\item The charts of case (iv) lie over the generic fiber. In particular, $f_{\gamma\eta}:U'_{\gamma\eta^\vee}\rightarrow V_{\gamma\eta}$ is an isomorphism.
\end{altenumerate}
\end{theorem}

The proof of Theorem \ref{thm:bl-charts} will base on an explicit computation on local charts, which is the main content of this subsection. Let us assume this for now.

\begin{proof}[Proof of Theorem \ref{thm:blowup}]
By the definition of $\wt{U}$ in \eqref{eq:blow up relations} and \cite[Theorem 8.0.1]{LUO2024}, $f$ is an isomorphism over the generic fiber $\Spec F$. 
By Lemma \ref{lem:cover-bl} and Theorem \ref{thm:bl-charts}, the composition $f$ is \'etale. We have the following diagram of proper maps:
\begin{equation*}
\begin{tikzcd}
    M^{\spl}\arrow[d]\arrow[r]&M^{\mathrm{bl}}\arrow[dl]\\
    M^{\loc}&
\end{tikzcd}.
\end{equation*}
Therefore, the base change $U^{\spl}\rightarrow U^{\mathrm{bl}}$ is also proper, so is the composition $f$. 
By \cite[03WS]{stackproject}, $f$ is \'etale and therefore locally quasi-finite. By \cite[02LS]{stackproject}, it is finite. 
Since it is an isomorphism on the generic fiber, the finite \'etale morphism $f$ is an isomorphism. It follows that the morphism $U^{\spl}\rightarrow U^{\mathrm{bl}}$ is an isomorphism. By Proposition \ref{prop:geometric:worst}, the morphism $M^{\spl}\rightarrow M^{\mathrm{bl}}$ is an isomorphism.
\end{proof}

Now we come to the proof of Theorem \ref{thm:bl-charts}. 
We will separate them into the following four cases:
\begin{enumerate}[label=(\roman*)]
\item $\eta\in [1,k]\cup [n-k+1,n],\gamma\in [k+1,n-k]$. 
\item $\gamma,\eta\in [1,k]\cup[n-k+1,n]$. 
\item $\gamma,\eta\in [k+1,n-k]$. 
\item$\gamma\in [1,k]\cup [n-k+1,n],\eta\in [k+1,n-k]$.
\end{enumerate}
We will construct the isomorphism $f_{\gamma\eta}$  defined over $U^{\loc}$ directly in cases (i)-(iii).



As in \eqref{eq:nabla definition}, denote
\begin{equation}\label{eq:nabla definition section iv}
	\bbs^{\nabla}:=\Upsilon\bbs,
	\quad
	\blambda^{\nabla}:=\Upsilon\blambda,
	\quad
	\bt^{\nabla}:=\Upsilon^t\bt,
	\quad
	\bmu^{\nabla}:=\Upsilon^t\bmu.
\end{equation}
Apply relations \eqref{eq:X s lambda} and \eqref{eq:Y t mu} to \eqref{equ:XYisotropic} (which still holds as it is a relation on $U^\loc$), we get 
\begin{equation}\label{eq:s lambda t t mu t}
	\bbs\blambda^t=(\Upsilon^t\bmu)(\Upsilon^t\bt)^t,
	\quad
	\bt\bmu^t=(\Upsilon\blambda)(\Upsilon\bbs)^t.
\end{equation}
Since $s_\gamma=1$ and $t_{\eta^\vee}=1$, this implies
\begin{equation}\label{eq:lambda-mu section 4}
	\blambda=\mu^{\nabla}_{\gamma}(\Upsilon^t\bt)=\mu^{\nabla}_{\gamma}\left(\begin{matrix}-H\bt_3\\ H\bt_2\\ H\bt_1\end{matrix}\right),
	\quad
	\bmu=\lambda^{\nabla}_{\eta^\vee}(\Upsilon\bbs)=\lambda^{\nabla}_{\eta^\vee}\left(\begin{matrix}H\bbs_3\\ H\bbs_2\\ -H\bbs_1\end{matrix}\right).
\end{equation}
Moreover as $\Upsilon\blambda=\mu_{\gamma}^{\nabla}\bt$, by comparing the $\eta^\vee$'s row, we get $\lambda_{\eta^\vee}^{\nabla}=\mu_{\gamma}^{\nabla}$.

Similarly denote
\begin{equation*}
	\bu^{\nabla}:=\Upsilon\bu,
	\quad	
	\bv^{\nabla}:=\Upsilon^t\bv.
\end{equation*}
Apply relations \eqref{eq:X s u} and \eqref{eq:Y t v} to the relation
$\mathfrak{Y}=\Upsilon\mathfrak{X}^t\Upsilon^t$, we get 
\begin{equation*}
	\bbs\bu^t=(\Upsilon^t\bv)(\Upsilon^t\bt)^t,
	\quad
	\bt\bv^t=(\Upsilon\bu)(\Upsilon\bbs)^t.
\end{equation*}
Since $s_\gamma=1$ and $t_{\eta^\vee}=1$, this implies
\begin{equation}\label{eq:u t v s}
	\bu=v^{\nabla}_{\gamma}(\Upsilon^t\bt)=v^{\nabla}_{\gamma}\left(\begin{matrix}-H\bt_3\\ H\bt_2\\ H\bt_1\end{matrix}\right),
	\quad
	\bv=u^{\nabla}_{\eta^\vee}(\Upsilon\bbs)=u^{\nabla}_{\eta^\vee}\left(\begin{matrix}H\bbs_3\\ H\bbs_2\\ -H\bbs_1\end{matrix}\right).
\end{equation}
Moreover as $\Upsilon\bu=v_{\gamma}^{\nabla}\bt$, by comparing the $\eta^\vee$'s row, we get $v_{\gamma}^{\nabla}=u_{\eta^\vee}^{\nabla}=\pm 1$.

From \eqref{eq:X s u}, we get
\[\mathfrak{X}_4+\varpi I=\bbs_2 \bu^t_2. \]
Comparing this with the last equation of \eqref{eq:blow up relations}, we get
\begin{equation}\label{eq:varpi}
    \varpi=\frac{1}{2} \bu^t_2 \bbs_2
\end{equation}

\subsubsection{Case (i)}
We focus on $V_{\gamma \eta}$ when $n-k+1\leq \eta\leq n$ and $k+1\leq \gamma \leq n-k$. The case $1\leq \eta \leq k$ and $k+1\leq \gamma \leq n-k$ can be dealt with in the same way.

\noindent (1) By \eqref{eq:lambda-mu section 4}, we have
\begin{equation*}
    \blambda_1=-\mu_{\gamma^\vee}H \bt_3,\quad
    \blambda_2=\mu_{\gamma^\vee}H \bt_2,\quad
    \blambda_3=\mu_{\gamma^\vee}H \bt_1.
\end{equation*}
and by \eqref{eq:u t v s}, we have
\begin{equation}\label{eq: u and t}
    \bu_1=-v_{\gamma^\vee}H \bt_3,\quad \bu_2=v_{\gamma^\vee}H \bt_2,\quad \bu_3=v_{\gamma^\vee}H \bt_1.
\end{equation}
In particular
\[\lambda_{\eta}=\mu_{\gamma^\vee} t_{\eta^\vee}=\mu_{\gamma^\vee},\quad 
u_{\eta}=v_{\gamma^\vee} t_{\eta^\vee}=v_{\gamma^\vee}.\]

\noindent (2)
Next, we want to express $\bmu$ in terms of $\bbs$ and $\blambda$. 
From Theorem \ref{regularity_relations:conclusion}(v), we have $D=\bbs_1\blambda_1^t-\pi\id$, $A^\ad=\bt_1\bmu_1^t-\pi\id$.
By one of the strengthened spin relations $D=-2\pi I-A^\ad$, we have $\bt_1\bmu_1^t=-\bbs_1\blambda_1^t$, by comparing the $\eta^\vee$'s rows, we obtain 
\begin{equation*}
  \bmu_1=-s_{\eta^\vee} \blambda_1
\end{equation*}
From \eqref{eq:lambda-mu section 4}, we have
\begin{equation*}
     \bmu_2=\lambda_\eta H\bbs_2.
\end{equation*}
Moreover, from Theorem \ref{regularity_relations:conclusion}(v), we have $C=\bbs_1\blambda_3^t$, $C^\ad=-\bt_1\bmu_3^t$. Apply one of the strengthened spin conditions $C=C^\ad$, we get $\bt_1\bmu_3^t=-\bbs_1\blambda_3$. By comparing the $\eta^\vee$'s rows, we deduce that
\begin{equation*}
    \bmu_3=-s_{\eta^\vee}\blambda_3.
\end{equation*}
We apply the same arguments to $\mathfrak{X}$, we can get:
\begin{equation}\label{eq:v and u}
    \bv_1=-s_{\eta^\vee} \bu_1,\quad  \bv_2=u_{\eta}H\bbs_2=v_{\gamma^\vee} H\bbs_2,\quad \bv_3=-s_{\eta^\vee} \bu_3.
\end{equation}

\noindent(3)
From equations \eqref{eq:X s u}, \eqref{eq:Y t v}, and expression \eqref{equ:fkXY}, we have $\mathfrak{C}=\bbs_1 \bu_3^t=-\bt_1 \bv_3^t$. Applying \eqref{eq:v and u}, we obtain
$\bbs_1 \bu_3^t=s_{\eta^\vee} \bt_1 \bu_3^t$.
Comparing the $\eta$-th row of this equation yields
\begin{equation*}
    \bbs_1= s_{\eta^\vee} \bt_1.
\end{equation*}
Similarly, using the same equations and expression, we have $\mathfrak{A}+\varpi I=\bbs_3 \bu_3^t=-\bt_3 \bv_3^t$. Applying \eqref{eq:v and u}, we get $\bbs_3 \bu_3^t=s_{\eta^\vee} \bt_3 \bu_3^t$.
Comparing the $\eta$-th row gives:
\begin{equation*}
    \bbs_3= s_{\eta^\vee} \bt_3.
\end{equation*}
By the equation \eqref{eq:Y t v} we have $\mathfrak{E}=\bt_2 \bv_2^t$. Applying \eqref{eq:v and u}, we get $\mathfrak{E}=v_{\gamma^\vee} \bt_2 \bbs_2^t H=\bt_2 \bbs_2^t H$.
By one of the strengthened spin relations, we have $\mathfrak{E}=\mathfrak{E}^\ad$. Hence, we get $\bt_2 \bbs_2^t=\bbs_2 \bt_2^t$.
Comparing the $\gamma$-th row, we have
\begin{equation*}
   \bt_2=t_\gamma \bbs_2. 
\end{equation*}
In summary we have
\begin{equation*}
\bbs_1= s_{\eta^\vee} \bt_1,\quad
\bt_2=t_\gamma \bbs_2,\quad
\bbs_3= s_{\eta^\vee} \bt_3.
\end{equation*}

\noindent(4)
From Theorem \ref{regularity_relations:conclusion}(v), we have $C=\bbs_1 \blambda_3^t$, $E^\ad=\bt_1 \bmu_2^t$ and $\mathfrak{E}=(\bt_1 \bv_2^t)^\ad$.
Now we use the relation $\mathfrak{C}=\frac{1}{2} E^\ad \mathfrak{E}$ in \eqref{eq:blow up relations}, we get:
\[2 \bbs_1 \blambda_3^t=\bt_1 \bmu_2^t H \bv_2 \bt_1^t H.\]
Combing the above with $\blambda_3=\mu_{\gamma^\vee}H \bt_1$ and $\bv_2=v_{\gamma^\vee}  H\bbs_2$, we get
\[2v_{\gamma^\vee} \bbs_1 \bt_1^t H=v_{\gamma^\vee}\mu_{\gamma^\vee} (\bbs_2^t H\bbs_2) \bt_1 \bt_1^t H.\]
Canceling $H$ and $v_{\gamma^\vee}$ ($v_{\gamma^\vee}=\pm 1$) on both sides of the equation, we get
\[2  \bbs_1 \bt_1^t =\mu_{\gamma^\vee}(\bbs_2^t H\bbs_2) \bt_1 \bt_1^t .\]
Now by comparing the $(\eta^\vee,\eta^\vee)$-th entry, we get
\begin{equation*}
    2  s_{\eta^\vee}=\mu_{\gamma^\vee} \Bigl(\sum_{i=k+1}^{n-k}s_{i}s_{i^\vee}\Bigr).
\end{equation*}
By the relation $\tr(\bt_2 \bmu_2^t)=2\pi$, we get
\begin{equation}\label{equ:bl-chart1}
    \mu_{\gamma^\vee} t_\gamma\Bigl(\sum_{i=k+1}^{n-k}s_{i}s_{i^\vee}\Bigr)=2\pi.
\end{equation}

\noindent(5)
From \eqref{eq:varpi}, we get
\[\varpi=\frac{1}{2} \bu^t_2 \bbs_2=\frac{1}{2} v_{\gamma^\vee} t_\gamma \bbs_2^t H \bbs_2. \]

\noindent(6)
To summarize we have express all the variables in terms of $\bt_1$, $\bbs_2$, $\bt_3$, and $s_{\eta^\vee}$, $t_\gamma$, $\mu_{\gamma^\vee}$, with the relation \eqref{equ:bl-chart1}. 
Using the same argument as in \S \ref{sec:simp}, we can see that all the relations in \eqref{eq:blow up relations} are satisfied.
Comparing with the affine chart $U'_{\alpha \beta}$ in \eqref{eq: U' chart i}, we see that $V_{\gamma \eta}\cong U'_{\alpha \beta}$ for $\gamma=\alpha$ and $\eta=\beta^\vee$.

\subsubsection{Case (ii)}
We focus on $V_{\gamma \eta}$ when $1\leq \gamma \leq k$ and $n-k+1\leq \eta\leq n$. The other cases can be argued exactly the same way. 

\noindent (1)
First we want to express $\blambda$ in terms of $\bbs$ and $\lambda_{\gamma^\vee}$. 
By the relation $\mathcal{X}+\pi I=\bbs \blambda^t$, and comparing terms in the following strengthened spin relations:
\begin{equation*}
    D=-2\pi\id-A^{\ad},\quad M=\pi E^{\ad},\quad C=C^{\ad},
\end{equation*}
we get
\begin{equation*}
    \blambda_1=-\lambda_{\gamma^\vee}H\bbs_3,
    \quad
    \blambda_2=\pi\lambda_{\gamma^\vee}H\bbs_2,
    \quad
    \blambda_3=\lambda_{\gamma^\vee}H\bbs_1.
\end{equation*}

\noindent(2)
Next we express $\mathbf{u}$ in terms of $\bbs$.
Using the relations $\mathfrak{D}=-2\varpi I-\mathfrak{A}^\ad$, $\mathfrak{M}=\pi \mathfrak{E}^\ad$, $\mathfrak{C}=\mathfrak{C}^\ad$, and $\mathfrak{X}+\varpi I=\bbs \mathbf{u}^t$, argue exactly the same way as above, we know that
\begin{equation*}
    \mathbf{u}_1=-u_{\gamma^\vee} H \bbs_3,\quad \mathbf{u}_2=\pi u_{\gamma^\vee} H \bbs_2, \quad \mathbf{u}_3=u_{\gamma^\vee} H \bbs_1.
\end{equation*}
Since $u_\eta=1$, we have $u_{\gamma^\vee} s_{\eta^\vee}=u_\eta=1$.
In particular $u_{\gamma^\vee}$ and $s_{\eta^\vee}$ are units and $u_{\gamma^\vee}=s_{\eta^\vee}^{-1}$. So we get
\begin{equation}\label{eq:chart i u in terms of s}
    \mathbf{u}_1=-s_{\eta^\vee}^{-1} H \bbs_3,\quad \mathbf{u}_2=\pi s_{\eta^\vee}^{-1} H \bbs_2, \quad \mathbf{u}_3=s_{\eta^\vee}^{-1} H \bbs_1.
\end{equation}

\noindent(3)
Now we express $\bmu$ and $\mathbf{v}$ in terms of $\bbs$ and $\blambda$. 
By \eqref{eq:lambda-mu section 4}, we have
\begin{equation*}
 \bmu_1=\lambda_\eta H\bbs_3,\quad
 \bmu_2=\lambda_\eta H\bbs_2,\quad
 \bmu_3=-\lambda_\eta H\bbs_1.
\end{equation*}
By \eqref{eq:u t v s}, we get $v_\gamma^\nabla=u_{\eta^\vee}^\nabla=u_\eta=1$ and
\begin{equation*}
\mathbf{v}_1=u_\eta H\bbs_3=H\bbs_3,\quad \mathbf{v}_2=u_\eta H\bbs_2=H \bbs_2,\quad
\mathbf{v}_3=-u_\eta H\bbs_1=-H\bbs_1.
\end{equation*}
Or equivalently $\mathbf{v}=\Upsilon \bbs$.

\noindent(4)
We express $\bt$ in terms of $\bbs$.
By \eqref{eq:u t v s}, we get $\bu=v_\gamma^\nabla \Upsilon^t \bt=\Upsilon^t \bt$ or equivalently $\bt=\Upsilon \bu$.
Combining with \eqref{eq:chart i u in terms of s}, we get
\[ \bt_1=s_{\eta^\vee}^{-1} \bbs_1,
	\quad
	\bt_2=\pi s_{\eta^\vee}^{-1} \bbs_2,
	\quad
	\bt_3=s_{\eta^\vee}^{-1} \bbs_3.\]
In particular $t_\gamma= s_{\eta^\vee}^{-1} s_\gamma=s_{\eta^\vee}^{-1}$ or equivalently  $t_\gamma s_{\eta^\vee}=1$.

\noindent(5)
From \eqref{eq:varpi}, we get
\[\varpi=\frac{1}{2} \bu^t_2 \bbs_2=\frac{1}{2} \pi s_{\eta^\vee}^{-1}\cdot  \bbs_2^t  H \bbs_2. \]

\noindent(6)
To summarize we have express all the variables in terms of $\bbs$ and $\lambda_{\gamma^\vee}$, and $s_{\eta^\vee}$ is a unit.
Now consider the relation $\mathfrak{C}=\frac{1}{2}E^{\ad}\mathfrak{E}$, using the facts that $\mathfrak{C}=\bbs_1 \mathbf{u}_2^t$, $\mathfrak{E}=\bbs_2 \mathbf{u}_3^t$ and $E=\bbs_2 \blambda_3^t$, we have
\[2 \bbs_1 \mathbf{u}_2^t=(\bbs_2 \blambda_3^t)^\ad \cdot \bbs_2 \mathbf{u}_3^t.\]
By comparing the $(\gamma,k+1-\gamma)$-th entry, we have
\[2 u_{\gamma^\vee}=u_{\gamma^\vee} \lambda_{\gamma^\vee} \Bigl(\sum_{i=k+1}^{n-k} s_{i}s_{i^\vee}\Bigr). \]
Canceling $u_{\gamma^\vee}$ (recall $u_{\gamma^\vee} s_{\eta^\vee}=1$), we have
\begin{equation*}
    2=\lambda_{\gamma^\vee} \Bigl(\sum_{i=k+1}^{n-k} s_{i}s_{i^\vee}\Bigr).
\end{equation*}
Using the same argument as in \S \ref{sec:simp}, we can see that all the relations in \eqref{eq:blow up relations} are satisfied.
Comparing with the affine chart $U_{\alpha \beta}$ in Theorem \ref{thm:chart-i}, we see that $V_{\gamma \eta}\cong U'_{\alpha \beta}$ for $\gamma=\alpha$ and $\eta=\beta^\vee$. 

\subsubsection{Case (iii)}
We focus on $V_{\gamma \eta}$ when $k+1\leq \gamma,\eta \leq n-k$. 

\noindent(1)
We express $\blambda$ in terms of $\bt$ and $\bmu$. By \eqref{eq:lambda-mu section 4}, we have
\begin{equation*}
\blambda_1=-\mu_{\gamma^\vee} H \bt_3,
\quad
\blambda_2=\mu_{\gamma^\vee} H\bt_2,
\quad
 \blambda_3=\mu_{\gamma^\vee} H\bt_1.
\end{equation*}

\noindent(2)
We express $\mathbf{u}$ in terms of $t$, by \eqref{eq:u t v s} we get
\begin{equation*}
\mathbf{u}_1=-v_{\gamma^\vee} H \bt_3,
\quad
\mathbf{u}_2=v_{\gamma^\vee} H\bt_2,
\quad
\mathbf{u}_3=v_{\gamma^\vee} H\bt_1.
\end{equation*}
In this case $v_{\gamma^\vee}=u_{\eta^\vee}^\nabla=u_\eta=1$, to summarize we have
\begin{equation*}
    \mathbf{u}=\Upsilon^t \bt.
\end{equation*}

\noindent(3)
We express $\bmu$ in terms of $\bt$ and $\mu_\eta$. Recall that $t_{\eta^\vee}=1$. By the relation $\mathcal{Y}+\pi I=\bt \bmu^t$, and comparing terms in the following strengthened spin relations:
\begin{equation*}
    L=-\pi F^{\ad}
    \quad
    X_4=X_4^{\ad}
    \quad
    M=\pi E^{\ad},
\end{equation*}
we get
\begin{equation*}
\bmu_1= \pi\mu_\eta H\bt_3,
\quad
\bmu_2=\mu_\eta H \bt_2,
\quad
 \bmu_3=-\pi \mu_\eta H \bt_1.
\end{equation*}

\noindent(4)
Next we want to express $\mathbf{v}$ in terms of $\bt$.
Using the relations $\mathfrak{Y}+\varpi I=\bt \mathbf{v}^t$, $\mathfrak{L}=-\pi \mathfrak{F}^\ad$, $\mathfrak{X}_4=\mathfrak{X}_4^\ad$ and $\mathfrak{M}=\pi \mathfrak{E}^\ad$, arguing exactly the same way as above, we have 
\begin{equation*}
\bv_1=\pi u_{\eta} H \bt_3,\quad
\bv_2= u_{\eta}H \bt_2, \quad
\bv_3=-\pi u_{\eta}H \bt_1.
\end{equation*}
Since $v_{\gamma^\vee}=1$, we have $u_\eta t_{\gamma}=v_{\gamma^\vee}=1$.
In particular $u_\eta$ and $t_{\gamma}$ are units. So we have
\begin{equation}\label{eq:chart (ii) v in terms of t}
    \bv_1=\pi t_\gamma^{-1} H \bt_3,\quad
    \bv_2= t_\gamma^{-1}H \bt_2, \quad
    \bv_3=-\pi t_\gamma^{-1}H \bt_1.
\end{equation}

\noindent(5)
We express $\bbs$ in terms of $\bt$.
By \eqref{eq:u t v s}, we get $\mathbf{v}=u_{\eta^\vee}^\nabla \Upsilon \bbs=\Upsilon \bbs$ or equivalently $\bbs=\Upsilon^t \mathbf{v}$.
Combining with \eqref{eq:chart (ii) v in terms of t}, we get
\[ \bbs_1=\pi t_{\gamma}^{-1} \bt_1,
	\quad
	\bbs_2=t_{\gamma}^{-1} \bt_2,
	\quad
	\bbs_3=\pi t_{\gamma}^{-1} \bt_3.\]
In particular $s_{\eta^\vee}=t_\gamma^{-1} t_{\eta^\vee}=t_\gamma^{-1}$ or equivalently $s_{\eta^\vee}t_\gamma=1$. So 
\[ \bbs_1=\pi s_{\eta^\vee} \bt_1,
	\quad
	\bbs_2=s_{\eta^\vee} \bt_2,
	\quad
	\bbs_3=\pi s_{\eta^\vee} \bt_3.\]
    
\noindent(6)
From \eqref{eq:varpi}, we get
\[\varpi=\frac{1}{2} \bu^t_2 \bbs_2=\frac{1}{2} t_\gamma^{-1} \bt_2^t H\bt_2.  \]

\noindent(7)
To summarize we have express all the variables in terms of $\bt$ and $u_\eta$, and $t_\gamma$ is a unit.
Consider the relation $\tr(X_4)=-(n-2k-2)\pi$, this is equivalent to $\tr(\bt_2\bmu_2^t-\pi\id)=-(n-2k-2)\pi$, hence $\tr(\bt_2\bmu_2^t)=2\pi$.
Using the relation $\bmu_2=\mu_\eta H \bt_2$, we have
\begin{equation*}
	\mu_{\eta}\Bigl(\sum_{i=k+1}^{n-k}t_{i}t_{i^\vee}\Bigr)=2\pi.
\end{equation*}
Using the same argument as in \S \ref{sec:simp}, we can see that all the relations in \eqref{eq:blow up relations} are satisfied.
Comparing with Proposition \ref{prop: chart ii}, we see that $V_{\gamma \eta}\cong U'_{\alpha \beta}$ for  
$\gamma=\alpha$ and $\eta=\beta^\vee$.

\subsubsection{Case (iv)}
We focus on $V_{\gamma \eta}$ when $1\leq \gamma\leq k$ and $k+1\leq \eta \leq n-k$. The case $n-k+1\leq \gamma\leq n$ and $k+1\leq \eta \leq n-k$ can be dealt with in the same way. 
By the relation $\mathfrak{M}=\pi \mathfrak{E}^\ad$,  we know that
\[\mathbf{u}_2=\pi u_{\gamma^\vee} H \bbs_2. \]
Since $u_\eta=1$, we have 
\[\pi u_{\gamma^\vee} s_{\eta^\vee}=u_\eta=1.\]
This shows that $\pi$ is a unit and $V_{\gamma \eta}$ is supported on the generic fiber.

\subsubsection{}
We can now conclude the proof of Theorem \ref{thm:bl-charts}. In cases (i)-(iii), we have constructed an isomorphism $f_{\gamma \eta}:U_{\gamma \eta^\vee}'\rightarrow V_{\gamma \eta}$. By \eqref{regular_splitting:4} and \eqref{eq:X s lambda}, these morphisms preserve the matrix $\CX$, hence are defined over $U^\loc$. Moreover, over $U^{\mathrm{spl}}$, we can define 
\[\bu:=v^{\nabla}_{\gamma}\left(\begin{matrix}-H\bt_3\\ H\bt_2\\ H\bt_1\end{matrix}\right),\]
as in \eqref{eq:u t v s} and
obtain global sections $\varpi :=\frac{1}{2} \bu_2^t \bbs_2$ and $\mathfrak X_{ij}$ ($1\leq i,j \leq n$) of the line bundle $\mathcal O (\Exc)$ over $U^\spl$  where
\[\mathfrak{X}:=-\varpi I+\bbs \mathbf{u}^t. \]
Via the morphisms $f_{\gamma\eta}$, these sections agree with the corresponding sections of the invertible sheaf $(I\cdot \mathcal O_{\wt{U}})^{-1}$ over $\widetilde{U}$.
This shows that the morphisms $f_{\gamma \eta} $ are indeed induced by the morphism $f$ in \eqref{eq:blow up morphism f}  (see for example the proof of \cite[Proposition 7.14]{Hartshorne}). This finishes the proof.\qed

\subsection{Exceptional divisor}
To the end of the section, we describe the exceptional divisor of the blow up.
Consider the $\BF_q$-vector space $M_n=\Lambda_{n-k}\otimes \BF_q$. Recall that we have the following standard $ O_F$-basis of $M_n$,
\[\pi^{-1}e_1\otimes 1,\cdots,\pi^{-1}e_{n-k}\otimes 1,e_{n-k+1}\otimes 1,\cdots, e_n\otimes 1.\]
Recall we have maps $\lambda_k: \Lambda_k\rightarrow \Lambda_{n-k}$ and $\lambda_{n-k}: \Lambda_{n-k}\rightarrow \Lambda_{n+k}$.
Define
\[M_{n-2k}=(\mathrm{coker} \lambda_k)_{\BF_q}= \Span_{\BF_q} \{\pi^{-1}e_{k+1}\otimes 1,\cdots,\pi^{-1}e_{n-k}\otimes 1 \}\cong (\mathrm{Im} \lambda_{n-k})_{\BF_q},\]
and
\[M_{2k}=(\mathrm{im} \lambda_k)_{\BF_q}=(\mathrm{ker} \lambda_{n-k})_{\BF_q}=\Span_{\BF_q} \{\pi^{-1}e_1\otimes 1,\cdots,\pi^{-1}e_k\otimes 1,  e_{n-k+1}\otimes 1,\cdots, e_n\otimes 1\}.\]
Then we have an exact sequence
\begin{equation}\label{eq:M_n exact sequence}
0\rightarrow M_{2k} \rightarrow M_n\xrightarrow{\rho} M_{n-2k}\rightarrow 0.
\end{equation}

Denote by $\mathbb{P}^{n-1}=\proj\BF_q[T_1,\ldots,T_n]$ the projective space of lines of the $M_n$, $\mathbb{P}^{2k-1}$ the projective space of $M_{2k}$, and $\mathbb{P}^{n-2k-1}$ the projective space of $M_{n-2k}$.
Here we use the homogeneous coordinate ring $\BF_q[T_1,\ldots,T_n]$ with respect to the standard basis.
From the exact sequence \eqref{eq:M_n exact sequence}, we get a monomorphism $\mathbb{P}^{2k-1}\hookrightarrow \mathbb{P}^{n-1}$. 
\begin{lemma}
    The subvariety of $\mathbb{P}^{n-1}\times \mathbb{P}^{n-2k-1}$ characterized by
    \[Z:=\{ (L_1,L_2)\in \mathbb{P}^{n-1}\times \mathbb{P}^{n-2k-1}\mid \rho(L_1)\subset L_2  \}\]
    is the blow-up of $\mathbb{P}^{n-1}$ along $\mathbb{P}^{2k-1}$.
\end{lemma}
\begin{proof}
Write homogeneous coordinates on $\mathbb{P}(M_{n})$ as $[x_1:\cdots:x_{2k},y_1:\cdots:y_{n-2k}]$, where the $y_i$'s are coordinates on the quotient $M_{n-2k}$. Then $P(M_{2k})=V(y_1,\cdots,y_{n-2k})\subset \mathbb{P}(M_n)$. Now let $[z_1:\cdots:z_{n-2k}]$ be coordinates on $\mathbb{P}(M_{n-2k})$. The condition $\rho(L_1)\subset L_2$ is exactly (A local chart computation will show that this defined a reduced closed subscheme)
$$
y_iz_j-y_jz_i=0\quad\text{for all }i,j.
$$

But the blow-up of $\mathbb{P}(M_n)$ along $\mathbb{P}(M_{2k})=V(y_1,\cdots,y_{n-2k})$ is the closure of the graph of the rational map
$$
\mathbb{P}(M_n)\dashrightarrow\mathbb{P}(M_{n-2k}),\quad [x_1:\cdots:x_{2k},y_1:\cdots:y_{n-2k}]\mapsto [y_1:\cdots:y_{n-2k}].
$$
Inside $\mathbb{P}(M_n)\times\mathbb{P}(M_{n-2k})$, this graph closure is also cut out by the equations $y_iz_j-y_jz_i=0$ for all $i,n$. Hence $Z$ and the blow-up are the same closed subscheme.
\end{proof}
We denote the blow-up of $\mathbb{P}^{n-1}$ along $\mathbb{P}^{2k-1}$ by $\mathrm{Bl}_{\mathbb{P}^{2k-1}}(\mathbb{P}^{n-1})$.

\begin{proposition}\label{prop:exc 1 shape}
    The exceptional divisor $\Exc_1$ of the blow-up $M_{[k]}^{\spl}\rightarrow M_{[k]}^\loc$ is isomorphic to $\mathrm{Bl}_{\mathbb{P}^{2k-1}}(\mathbb{P}^{n-1})$.
\end{proposition}
\begin{proof}
We have seen that there exists an open cover
\begin{equation*}
\left\{E_{\alpha \beta}\mid \alpha \in [k+1,n-k], \beta\in [1,k]\cup[n-k+1,n] \right\}
\end{equation*}
of $E$. We will show that on each $E_{\alpha \beta}$ there is a natural embedding $E_{\alpha \beta}\rightarrow \mathrm{Bl}_{\mathbb{P}^{2k-1}}(\mathbb{P}^{n-1})$.

First we assume that $k+1\leq \beta \leq n-k$. This corresponds to charts of type (ii) as in \S\ref{subsec: chart of type ii}. In this case we know that
\[E_{\alpha \beta}=\Spec \BF_q[t_1,\ldots,t_n]/(t_\beta-1)\cong \mathbb{A}_{\BF_q}^{n-1} ,\]
and both $\bbs_2$ and $\bt_2$ are not zero. In particular $\bt$ determines an element $L_1\in \mathbb{P}^{n-1}$ and $\bbs_2$ determines an element $L_2\in \mathbb{P}^{n-2k-1}$. The relation $\bbs_2=s_\beta \bt_2$ in Proposition \ref{prop:chart-ii-relations} shows that $\rho(L_1)=L_2$. Hence we know that $E_{\alpha \beta}$ is a subvariety of $\mathrm{Bl}_{\mathbb{P}^{2k-1}}(\mathbb{P}^{n-1})$.

Secondly we assume that $1\leq \beta \leq k$. As before the case $n-k+1\leq \beta\leq n$ is exactly the same, so is omitted. This corresponds to charts of type (iv) as in \S\ref{subsec: chart of type iv}. In this case we know that
\[E_{\alpha \beta}=\Spec \BF_q[t_1,\cdots,t_k,s_{k+1},\cdots, s_{n-k},t_{n-k+1},\cdots,t_n,t_\alpha]/(t_\beta-1,s_\alpha-1)\cong \mathbb{A}_{\BF_q}^{n-1} ,\]
and $\bbs_2$ is not zero. In particular $\bt$ determines an element $L_1\in \mathbb{P}^{n-1}$ and $\bbs_2$ determines an element $L_2\in \mathbb{P}^{n-2k-1}$. The relation $\bt_2=t_\alpha \bbs_2$ in Proposition \ref{prop:chart-iv-relations} suggests that $\rho(L_1)\subset L_2$. Hence we know that $E_{\alpha \beta}$ is a subvariety of $\mathrm{Bl}_{\mathbb{P}^{2k-1}}(\mathbb{P}^{n-1})$.

Obviously the images of the above embeddings cover $\mathrm{Bl}_{\mathbb{P}^{2k-1}}(\mathbb{P}^{n-1})$ and these embeddings agree on overlaps. This finishes the proof of the proposition.
\end{proof}

\section{Applications to Rapoport--Zink spaces}\label{sec:RZ spaces}
In this section, we apply our previous results to Rapoport--Zink (RZ) spaces, and conclude interesting results on line bundle of modular forms and special cycles over them.

\subsection{Review of strict $O_{F_0}$-modules}\label{sec:strict-module}
In this subsection, we briefly review the theory of $O_{F_0}$-strict modules. 
For more details, see \cite{Mihatsch_2022,kudla2023padicuniformizationunitaryshimura,LMZ}.
Assume $p\neq 2$ throughout.
Recall that $F_0/\BQ_p$ is an extension of $p$-adic fields with a fixed uniformizer $\pi_0\in O_{F_0}$. Assume the residue field of $O_{F_0}$ is a finite field of order $q$.
Let $S$ be an $\Spf O_{F_0}$-scheme.
A \emph{strict $O_{F_0}$-module} $X$ over $S$ is a pair $(X,\iota)$ where $X$ is a $p$-divisible group over $S$ and $\iota: O_{F_0}\rightarrow \End(X)$ an action such that $O_{F_0}$ acts on $\Lie(X)$ via the structure morphism $O_{F_0}\rightarrow \CO_S$. 
A strict $O_{F_0}$-module is called \emph{formal} if the underlying $p$-divisible group is formal. 
By Ahsendorf-Cheng-Zink \cite{Ahsendorf-Cheng-Zink}, there is an equivalence of categories between the category of strict formal $O_{F_0}$-modules over $S$ and the category of nilpotent $O_{F_0}$-displays over $S$.
To any strict formal $O_{F_0}$-module, there is also an associated crystal $\BD_X$ on the category of $O_{F_0}$-pd-thickenings. 
We define the de Rham homology $D(X):=\BD_X(S)$.
There is a short exact sequence of $\CO_S$-modules:
\begin{equation*}
	0\rightarrow \Fil(X)\rightarrow D(X)\rightarrow \Lie(X)\rightarrow 0;
\end{equation*}
where $\Fil(X)\subset D(X)$ is the Hodge filtration. The (relative) Grothendieck--Messing theory states that the deformations of $X$ along $O_{F_0}$-pd-thickenings are in bijection with liftings of the Hodge filtration.

We will restrict to the case when $X=(X,\iota)$ is \emph{biformal}, see \cite[Definition 11.9]{Mihatsch_2022} for the definition.
For a biformal strict $O_{F_0}$-module $X$, we can define the \emph{(relative) dual} $X^\vee$ of $X$, and hence the \emph{(relative) polarization} and the \emph{(relative) height}.
It follows from the definition that there is a perfect pairing
\begin{equation*}
    D(X)\times D(X^\vee)\rightarrow \CO_S
\end{equation*}
such that $\Fil(X)\subset D(X)$ and $\Fil(X^\vee)$ are orthogonal complements to each other.

When $S=\Spf R$ is perfect, a nilpotent $O_{F_0}$-display is equivalent to a relative Dieudonne module $M(X)$ over $W_{O_{F_0}}(R)$ with the action of a $\sigma$-linear operator $F$ and a $\sigma^{-1}$-linear operator $V$, such that $FV=VF=\pi_0\id$.

\subsection{RZ spaces}\label{sec:RZ}
For the remaining part of the paper, assume $F/F_0$ is a ramified quadratic extension of $p$-adic fields. Fix uniformizers $\pi\in F$ and $\pi_0\in F_0$ such that $\pi^2=\pi_0$.
\begin{definition}
\noindent(i) Let $S$ be a $\Spf O_{\breve{F}}$-scheme. A \emph{hermitian $O_{F}$-module of type $2k$ and dimension $n$} is a triple $(X,\iota_X,\lambda_X)$ consisting of a strict biformal $O_{F_0}$-module $X$ of height $2n$ and dimension $n$ over $S$, a homomorphism
$$
\iota_X \colon O_F \to \End_S (X),
$$
and a \emph{relative} polarization
$$
\lambda_X \colon X \to X^\vee,
$$
subject to the following constraints:
\begin{itemize}
\item 
the Rosati involution on $\End_S^\circ(X)$ attached to $\lambda_X$ induces the nontrivial Galois automorphism on $O_F$; and
\item 
$\ker \lambda_X \subset X[\iota_X(\pi)]$ has height $q^{2k}$.
\end{itemize}

\noindent(ii) A hermitian $O_F$-module is \emph{admissible} if it satisfies the following condition: denote by $D(X)$ and $D(X^\vee)$ the respective de Rham homology of $X$ and $X^\vee$. Since $\ker \lambda_X$ is contained in $X[\iota(\pi)]$ and of rank $q^{2k}$, there is a unique (necessarily $O_F$-linear) isogeny $\lambda^\vee$ such that the composite
$$
X\overset{\lambda}{\longrightarrow} X^\vee\overset{\lambda^\vee}{\longrightarrow} X
$$
is $\iota(\pi)$, and the induced diagram
\begin{equation*}
			D(X)\overset{\lambda_*}{\longrightarrow} D(X^\vee)\overset{\lambda^\vee_*}{\longrightarrow} D(X)
		\end{equation*}
then extends periodically to a polarized chain of $O_F\otimes_{O_{F_0}}\CO_S$-modules $\CL_{[k]}$, comp.  \eqref{equ:CLk lattice chain}. By \cite[Th. 3.16]{RZ}, \'etale-locally on $S$ there exists an isomorphism of polarized chains
\begin{equation*}\label{equ:ss-trivilization}
[\cdots\xrightarrow{\lambda_*^\vee}D(X)\xrightarrow{\lambda_*}D(X^\vee)\xrightarrow{\lambda_*^\vee}\cdots]\xrightarrow{\sim}\Lambda_{[k]}\otimes_{O_{F_0}}\CO_S,
\end{equation*}
which in particular gives an isomorphism of $O_F\otimes_{O_{F_0}}\CO_S$-modules
\begin{equation}\label{eq:lie-alg-trivilization}
	D(X)\overset{\sim}{\longrightarrow}\Lambda_{-k}\otimes_{O_{F_0}}\CO_S.
\end{equation}
The admissible condition we impose is that upon identifying $\Fil(X)$ with a submodule of $\Lambda_{-k}\otimes_{O_{F_0}}\CO_S$ via \eqref{eq:lie-alg-trivilization}, it defines a point in the local model $M^{\loc}_{[k]}$.

\noindent(iii) A \emph{splitting structure} on a hermitian $O_{F}$-module  $(X,\iota_X,\lambda_X)$ of dimension $n$ and type $2k$ is a pair
of two locally $\CO_S$-direct summands of rank one,
\begin{equation*}
	\Fil^0(X)\subseteq\Fil(X)\subseteq D(X),\quad \Fil^0(X^\vee)\subseteq\Fil(X^\vee)\subseteq D(X^\vee),
\end{equation*}
subject to the following constraints:
\begin{itemize}
\item  the morphisms between Hodge filtrations induced by the polarization carry the one additional filtration into the other:
\begin{equation*}
	\lambda_*(\Fil^0(X))\subseteq \Fil^0(X^\vee),\quad \lambda^\vee_*(\Fil^0(X^\vee))\subseteq \Fil^0(X).
\end{equation*}
\item 
(\emph{Kr\"amer condition}) if $n > 1$ is even and $2k=n$, then the condition states that $\Fil^0(X)=(\iota(\pi)-\pi)\Fil(X)$.
For the other cases, it requires
\begin{align*}
	(\iota(\pi)-\pi)\Fil(X)\subseteq \Fil^0(X), & \quad (\iota(\pi)+\pi)\Fil^0(X)=(0);\\
	(\iota(\pi)-\pi)\Fil(X^\vee)\subseteq \Fil^0(X^\vee), &\quad (\iota(\pi)+\pi)\Fil^0(X^\vee)=(0).
\end{align*}
\end{itemize}
\end{definition}

\begin{remark}
\begin{altenumerate}
\item The admissible condition can be replaced by the strengthened spin condition, see \cite{LUO2024}.
\item Let $(X,i,\lambda,\rho)$ be any hermitian $O_F$-module. The action $i: O_F\rightarrow \End(X)$ induces an action $i^\vee: O_F\rightarrow \End(X^\vee)$.
However throughout the paper unless otherwise stated, in order to make $X^\vee$ have the same signature as $X$ and have a natural definition of special cycles $\CY(x)$ (see below), we use the conjugate action $\ov{i}^\vee$ of $i^\vee$.
See also \cite[Definition 3.21(v)]{RZ}. 
\end{altenumerate}
\end{remark}

\begin{definition}\label{def:RZ spaces}
\noindent(i) A hermitian $O_F$-module $(\BX,i_{\BX},\lambda_{\BX})$ over $\Spec\BF$ is called the \emph{framing object} if the relative Dieudonn\'e module is isoclinic. See \cite[Proposition 5.1.3]{LRZ25} for a full classification of the framing objects.

\noindent(ii)
Let $(\BX,i_{\BX},\lambda_{\BX})$ be a framing object of type $2k$ and dimension $n$. We define the \emph{canonical RZ space} $\CN_{\BX}$ over $\Spf O_{\breve F}$  as the formal scheme parametrizing isomorphism classes of quadruples $(X, \iota_X, \lambda_X, \rho_X)$, where:
\begin{itemize}
\item $(X,\iota_X,\lambda_X)$ is an admissible hermitian $O_{F}$-module of type $2k$;
\item  $\rho_X$, called the {\it framing} with framing object $(X, \iota_X, \lambda_X, \rho_X)$, is an $O_F$-linear quasi-isogeny of height $0$
\begin{equation*}
	\rho_X: X\times_S \ov{S}\longrightarrow \BX\times_{\BF}\ov{S} ,
\end{equation*}
such that $\rho^*(\lambda_{\BX}\times_{\BF}\ov{S})=\lambda_X\times_S \ov{S}$. 
\end{itemize}

\noindent(iii)
Fix a framing object $(\BX,i_{\BX},\lambda_{\BX})$. We define the \emph{naive splitting RZ space} $\CN_{\BX}^{\nspl}$ over $\Spf O_{\breve F}$ parametrizing the collection of data 
\begin{equation*}
	(X, \iota, \lambda, \Fil^0(X),\Fil^0(X^\vee); \rho),
\end{equation*}
where $(X,\iota,\lambda,\Fil^0(X),\Fil^0(X^\vee))$ is a hermitian $O_F$-module of type $2k$ with splitting structure, and where $\rho$ is a {\it framing} with the fixed framing object.

\noindent(iv)
We define the \emph{splitting RZ space} $\CN_{\BX}^{\spl}$ over $\Spf O_{\breve F}$ as the flat closure of $\CN_{\BX}^{\nspl}$, i.e. closed formal subscheme defined by the ideal sheaf $\CO_X[\pi^\infty]\subseteq \CO_X$ for $X=\CN_{\BX}^{\spl}$ ($\pi$-power torsion elements of the structure sheaf).
\end{definition}

By Theorem \ref{thm:main} and Theorem \ref{prop:Kramer-imply-ss}, combine with the local model diagram, we immediately have
\begin{proposition}
\begin{altenumerate}
\item There is natural chain of maps:
\begin{equation}\label{equ:chain of RZ spaces}
\begin{aligned}
\xymatrix{
\CN_{\BX}^{\spl}\ar@{^(->}[r]&\CN_{\BX}^{\nspl}\ar[r]&\CN_{\BX},
}
\end{aligned}
\end{equation}
where the morphism on the left is a closed immersion.
\item The splitting RZ space $\CN_{\BX}^{\spl}$ is a flat formal scheme with semi-stable reduction.
\end{altenumerate}
\end{proposition}
\begin{proof}
Part (i) is straightforward, part (ii) follows from the local model diagram for $\CN_{\BX}^{\spl}$, which is obtained from the local model diagram for $\CN_{\BX}^{\nspl}$ via the modification \cite[Prop. 2.3]{Pappas2000} and the uniqueness of the flat closure.
\end{proof}

\subsection{The line bundle of modular forms}\label{vec}
In this subsection, we define two natural line bundles over the naive splitting RZ space.

Let $\BE$ be the Lubin-Tate $O_F$-module over $\Spec \BF$, which is the unique strict $O_F$-module over $\Spec \BF$.
There is a unique lifting (the canonical lifting) $\CE$ of Lubin-Tate $O_F$-module $\BE$ over $\Spf O_{\breve{F}}$, equipped with its framing $\rho_{\CE}:\CE\times_{\Spf O_{\breve{F}}}\Spec \BF\overset{\sim}{\to} \BE$ and its principal polarization $\lambda_{\CE}$ lifting $\rho_{\CE}^*(\lambda_{\BE})$.

Let $0\leq k\leq n/2$ be \emph{any} integer.
Let $S$ be a  $\Spf O_{\breve{F}}$-scheme over $\CN_{\BX}^{\nspl}$ and $X\in \CN^{\nspl}_{\BX}(S)$ be a $S$-point. In the remaining part of the paper, we can assume $S=\CN^{\spl}_{\BX}$ defined by the closed embedding $\CN^{\spl}_{\BX}\hookrightarrow\CN^{\nspl}_{\BX}$.
Recall that the de Rham homology  $D(X)$ and $D(X^\vee)$ defined in \S \ref{sec:strict-module} are vector bundles over $S$ of rank $2n$. These fit into short exact sequences
\begin{equation*}
    0\rightarrow \Fil(X)\longrightarrow D(X)\longrightarrow \Lie(X)\rightarrow 0,
    \quad
    0\rightarrow \Fil(X^\vee)\longrightarrow D(X^\vee)\longrightarrow \Lie(X^\vee)\rightarrow 0.
\end{equation*}
Similarly, the de Rham homology of $\CE=\CE\times_{\Spf O_{\breve{F}}}S$ determines a short exact sequence
\begin{equation*}
    0\lr \Fil(\CE)\lr D(\CE)\lr \Lie(\CE)\lr 0
\end{equation*}
of vector bundles on $S$.

From the moduli description, the actions $\iota_\CE: O_F\rightarrow \End(\CE)$ and $\iota: O_F\rightarrow \End(X)$ and $\iota^\vee: O_F\rightarrow \End(X^\vee)$ induce actions of $ O_F$ on all of these vector bundles over $S$, and the above short exact sequences are $ O_{F}$-linear.

By definition, we have a perfect alternating pairing
\begin{equation*}
    \langle\cdot , \cdot \rangle:D(X)\times D(X^\vee)\lr \CO_{S},
\end{equation*}
which is compatible with the action $\iota: O_F\rightarrow \End_{ \CO_{S}}(D(X))$ and $\ov{\iota}^\vee:O_F\rightarrow \End_{\CO_{S}}(D(X^\vee))$, in the sense that
\begin{equation*}
    \langle \alpha\cdot x,y\rangle=\langle x,\ov{\alpha}\cdot y\rangle
\end{equation*}
for all $\alpha\in O_F$ and all local sections $x$ (resp. $y$) of $D(X)$ (resp. $D(X^\vee)$).

The local direct summands  $\Fil(X)\subset D(X)$ and $\Fil(X^\vee)\subset D(X^\vee)$ are perpendicular, hence there are two induced perfect pairings
\begin{equation}\label{vec:pairings}
    \langle\cdot,\cdot\rangle:\Fil(X)\times \Lie(X^\vee)\lr \CO_{S},
    \text{ and }
    \langle\cdot,\cdot\rangle:\Fil(X^\vee)\times \Lie(X)\lr \CO_{S}.
\end{equation}

By virtue of the moduli problem defining $\CN^{\nspl}_{\BX}$, we have distinguished local direct summands $\Fil^0(X)\subset\Fil(X)$ and $\Fil^0(X^\vee)\subset\Fil(X^\vee)$ of rank $1$ on $S$. By taking dual via \eqref{vec:pairings}, we have local direct summands  $F_{X^\vee}\subset\Lie(X^\vee)$ and  $F_X\subset\Lie(X)$  of rank $n-1$ respectively.
Let 
\begin{equation}\label{equ:L_X-in-Lie}
	L_X:=\ker\left(
	\Lie(X)\xrightarrow{\Pi-\pi}(\Pi-\pi)D(X)/\Fil^0(X).
	\right)
\end{equation}
By Proposition \ref{prop:duality}, it is a direct summand of $\Lie_X$ of rank one. It is stable under the $O_F$-action. Moreover, we have $(\Pi-\pi)\Lie(X)\subset L_X$ and $(\Pi+\pi)L_X=(0)$.
Replacing $X$ with $X^\vee$, we can similarly construct a line bundle $L_{X^\vee}\subset\Lie(X^\vee)$.

\begin{definition}\label{def:line-bundle-modular-form}
The line bundle of modular forms $\omega_{\CZ}$ on $\CN_\BX^{\nspl}$ is the invertible sheaf of $\CO_{\CN_{\BX}^{\nspl}}$-modules with inverse
    \begin{equation*}
        \omega_{\CZ}^{-1}=\underline{\Hom}(\Fil(\CE),L_X).
    \end{equation*}
Similarly, one can also define
\begin{equation*}
        \omega_{\CY}^{-1}=\underline{\Hom}(\Fil(\CE),L_{X^\vee}).
    \end{equation*}
For any $S\to \CN_{\BX}^{\nspl}$, we define the line bundle of modular forms $\omega_{\CZ,S}$ (resp, $\omega_{\CY,S}$) over $S$ by pull-back.
\end{definition}

\begin{remark}
We define the line bundles of modular forms following the construction in \cite{Howard2019}, which is different from that in \cite{BHKRY}, cf. \cite[Remark 3.5]{Howard2019}.
The terminology is motivated by their global counterpart on the unitary Shimura variety, see \cite[Remark 2.4.1]{BHKRY} for further explanation.
\end{remark}

\begin{remark}\label{rmk:no extend from rigid}
Let $\CN_\BX^{\mathrm{rig}}$ be the rigid generic fiber of the RZ space. We have the Grothendieck--Messing (or Rapoport--Zink) crystalline morphism in \cite{RZ} (which extends to the relative in \cite{Fargues08}):
\begin{equation}\label{equ:period map}
    \pi:\CN_\BX^{\rig}\lr N^{\rig}
\end{equation}
where $N^{\rig}$ is the rigid analytic flag variety parameterizing all codimension one subspaces of $(\Pi+\pi)V$. The tautological line bundle on $N^{\rig}$ then pull-back to a line bundle $\omega^\rig$ over $\CN_\BX^{\rig}=\CN_\BX^{\nspl,\rig}$. This is the rigid generic fiber of line bundle of modular forms:
$$
\omega^{\rig}=\omega_\CZ^\rig=\omega_\CY^\rig.
$$
It is the $p$-adic analog of the line bundle of modular forms over the complex hermitian symmetric domain of $\U(n-1,1)$. 

\end{remark}

\subsection{Special cycles on splitting RZ spaces}\label{sec:specialcycleinsplittingmodel}
Given a framing object $(\BX,i_{\BX},\lambda_{\BX})$, we define the space of \emph{special quasi-homomorphisms}
\begin{equation}\label{eq:special-hom}
\BV=\BV(\BX):=\Hom^\circ_{O_F}(\BE,\BX).
\end{equation}
Then $\BV$ is an $n$-dimensional $F$-vector space. It carries a natural nondegenerate $F/F_0$-hermitian form $h$: for $x,y\in \BV$, the composite
\begin{equation*}
	\BE\overset{y}{\lr}\BX
	\overset{\lambda_{\BX}}{\lr}\BX^\vee\overset{x^\vee}{\lr}\BE^\vee\overset{\lambda^{-1}_{\BE}}{\lr}\BE
\end{equation*}
lies in $\End^\circ_{O_F}(\BE)$ and, hence, identifies with an element $h(x,y)\in F$ via the isomorphism
\begin{equation*}
	\iota_{\BE}:F\overset{\sim}{\longrightarrow}\End^\circ_{O_F}(\BE).
\end{equation*}

\begin{definition}\label{def:special cycle}
\begin{altenumerate}
\item For $x\in \mathbb{V}$, define the ($\CZ$-Kudla--Rapoport) special cycle $\CZ(x)$ as the closed subfunctor of $\CN^{\nspl}_{\BX}$, such that for any $\Spf O_{\breve{F}}$-scheme $S$, $\CZ(x)(S)$ is the set of isomorphism classes $(X,i,\lambda,\rho,\CF_X,\CF_{X^\vee})$
such that the quasi-homomorphism 
$$
\rho^{-1}\circ x\circ \rho_0:\CE\times_{\Spf O_{\breve{F}}} \ov{S} \longrightarrow X\times_{S} \ov{S}
$$
extends to a homomorphism $\CE\times_{\Spf O_{\breve{F}}}S\rightarrow X$.
For any $\Spf O_{\breve{F}}$-scheme $S$ over $\CN_{\BX}^{\nspl}$, we define the special cycle $\CZ(x)_S$ on $S$ as the pull-back of $\CZ(x)$ on $\CN^{\nspl}_{\BX}$.

\item For $x\in \mathbb{V}$, define the ($\CY$-Kudla--Rapoport) special cycle $\CY(x)$ as the closed subfunctor of $\CN_{\BX}^{\nspl}$, such that for any $\Spf O_{\breve{F}}$-scheme $S$, $\CY(x)(S)$ is the set of isomorphism classes $(X,i,\lambda,\rho,\CF_X,\CF_{X^\vee})$
such that the quasi-homomorphism 
$$
\rho^\vee\circ(\lambda_{\BX}\times \ov{S})\circ x\circ \rho_0:\CE\times_{\Spf O_{\breve{F}}} \ov{S} \longrightarrow X^\vee\times_{S} \ov{S}
$$
extends to a homomorphism $\CE\times_{\Spf O_{\breve{F}}}S\rightarrow X^\vee$. 
For any $\Spf O_{\breve{F}}$-scheme $S$ over $\CN_{\BX}^{\nspl}$, we define the special cycle $\CY(x)_S$ on $S$ as the pull-back of $\CY(x)$ on $\CN_{\BX}^{\nspl}$.
\end{altenumerate}
\end{definition}

We generalize Howard's work \cite{Howard2019} to our case. 
Let $Z\subset S$ be any closed formal subscheme, and denote by $I_Z\subset\CO_{S}$ its ideal sheaf. The square $I_Z^2$ is the ideal sheaf of a larger closed formal subscheme
\begin{equation*}
    Z\subset\wt{Z}\subset S
\end{equation*}
called the \emph{first-order infinitesimal neighborhood} of $Z$ in $S$.

Let $S$ be a $\Spf O_F$-scheme over $\CN_{\BX}^\nspl$ and let $x\in \BV$. We consider the first-order infinitesimal neighborhood
\begin{equation*}
\begin{aligned}
\xymatrix{
\CZ(x)_S\ar@{^(->}[r]\ar[d]&\wt{\CZ}(x)_S\ar@{^(->}[r]\ar[d]&S\ar[d]\\
\CZ(x)\ar@{^(->}[r]&\wt{\CZ}(x)\ar@{^(->}[r]&\CN_{\BX}^{\nspl}
}
\end{aligned}
\end{equation*}
of the special cycles and their pull-backs to $S$. 

By the definition of $\CZ(x)$, the restriction of $\CE$ and $X$ over $S$ to $\CZ(x)_S$ has a distinguished morphism of strict $O_{F_0}$-modules
\begin{equation*}
    \CE|_{\CZ(x)_S}\overset{x}{\longrightarrow} X|_{\CZ(x)_S},
\end{equation*}
which further induces an $O_F$-linear morphism of vector bundles on $\CZ(x)_S$ 
\begin{equation}\label{def:special-cycle-morphism}
    D(\CE)|_{\CZ(x)_S}\overset{x}{\longrightarrow} D(X)|_{\CZ(x)_S}
\end{equation}
that preserves the Hodge filtrations.
By (relative) Grothendieck--Messing theory,  we have a canonical extension to $\wt{\CZ}(x)_S$
\begin{equation*}
    D(\CE)|_{\wt{\CZ}(x)_S}\overset{\wt{D(x)}}{\longrightarrow} D(X)|_{\wt{\CZ}(x)_S},
\end{equation*}
which does not preserve the Hodge filtrations anymore. In particular, along the commutative diagram it induces a nontrivial composite morphism 
\begin{equation}\label{de:GM-Fil-to-Lie}
   \wt{x}: \Fil(\CE)|_{\wt{\CZ}(x)_S}\longrightarrow  D(\CE)|_{\wt{\CZ}(x)_S} \overset{\wt{D(x)}}{\longrightarrow}  D(X)|_{\wt{\CZ}(x)_S}\longrightarrow \Lie(X)|_{\wt{\CZ}(x)_S},
\end{equation}
where the first arrow and last arrow come from the Hodge exact sequences for $\CE $ and $X$ respectively.

\begin{proposition} \cite[Proposition 4.1]{Howard2019}\label{prop:factor-thgourh}
    The morphism \eqref{de:GM-Fil-to-Lie} factors through $L_X|_{\wt{\CZ}(x)_S}$. In particular, 
    it can be viewed as a morphism of line bundles
    \begin{equation}\label{de:GM-Fil-to-Lie-II}
        \Fil(\CE)|_{\wt{\CZ}(x)_S}\overset{\wt{x}}{\longrightarrow} L_X|_{\wt{\CZ}(x)_S}.
    \end{equation}
    The Kudla--Rapoport divisor $\CZ(x)_S$ coincides with the zero locus of $\wt{x}$. \qed
\end{proposition}

\begin{definition}\label{def:special-divisor-obstrucntion}
    The section
    \begin{equation*}
        \obst(x)\in H^0(\wt{\CZ}(x)_S,\omega_{\CZ}^{-1}|_{\wt{\CZ}(x)_S})
    \end{equation*}
    determined by \eqref{de:GM-Fil-to-Lie-II} is called the \emph{obstruction to deforming }$x$.
    As we explained, $\CZ(x)_S$ is the closed subscheme of $\wt{\CZ}(x)_S$ defined by $\obst(x)=0$.
\end{definition}

\begin{proposition}\label{prop:special-cartier}
For any nonzero $x\in \BV$, if the closed formal subscheme $\CZ(x)\subset \CN_{\BX}^{\nspl}$ (resp. $\CY(x)\subset \CN_{\BX}^{\nspl}$) is nonempty, it is a Cartier divisor; that is to say, it is defined locally by a single nonzero equation.
\end{proposition}
\begin{proof}
    Let $R$ be the local ring of $\CN^{\nspl}_{\BX}$ at a point $z\in \CZ(x)$, and let $I\supset I^2$ be the ideals of $R$ corresponding to $\CZ(x)\subset\wt{\CZ}(x)$. 
    The line bundle $\omega_{\CZ}$ is free on $S=\Spf(R)$ and the obstruction to deforming $x$ becomes an $R$-module generator
    \begin{equation*}
        \obst(x)\in I/I^2.
    \end{equation*}
    It follows from Nakayama's lemma that $I\subset R$ is a principal ideal, and it only remains to show that $I\neq 0$. Recall from Remark \ref{rmk:no extend from rigid} that we have the Grothendieck--Messing (or Rapoport--Zink) crystalline morphism 
    \begin{equation*}
        \pi:\CN^{\nspl,\rig}_{\BX}\longrightarrow N^{\rig},
    \end{equation*}
    It follows from \cite[Proposition 5.17]{RZ} that $\pi$ is \'etale.
    
    Suppose $I=0$, this implies that we may find an open subset $U\subseteq \CN^{\nspl}_{\BX}$ 
    such that $\CZ(x)|_U=U$, and $U\subseteq \CN^{\nspl}_{\BX}$ determines an admissible open subset
    \begin{equation*}
        U^{\rig} \subseteq \CN^{\nspl,\rig}_{\BX}.
    \end{equation*}
    But the image $\pi(Z(x)^{\rig})\subset N^{\rig}$ will factor through the rigid space associated to a closed subscheme of $N$, which is a contradiction. See \cite[Proposition 4.3]{Howard2019} for the detailed argument (one can alternatively establish properness by using the $p$-adic uniformization of the integral model of the unitary Shimura variety).

The same argument works here for $\CY$-cycles by replacing $X$ with $X^\vee$ and $L_X$ with $L_{X^\vee}$. In particular, we have that the closed formal subscheme $\CY(x)\subset\CN^{\nspl}_{\BX}$ is a Cartier divisor when $x\neq 0$.
\end{proof}

Next, we conclude the linear invariance properties of special cycles on $\CN_\BX^{\spl}$.

Let $S$ be a regular formal scheme. Recall from \cite[Appendix B]{WZhang21} that $K'_0(S)$ is the free abelian group generated by symbols $[F]$, where $F$ runs over all isomorphism classes of coherent $\mathcal{O}_S$-modules, modulo the relations $[F_1]+[F_3]=[F_2]$ for every short exact sequence
$$
0\lr F_1\lr F_2\lr F_3 \lr 0.
$$
This group is endowed with a ring structure, where multiplication is defined via the derived tensor product.

Let $S\rightarrow  \CN^{\nspl}_{\BX}$ be a regular formal scheme over $\CN^{\nspl}_{\BX}$, for instance, we can take $S=\CN^{\spl}_{\BX}\hookrightarrow \CN^{\nspl}_{\BX}$ to be the splitting RZ space, this recovers the setting in Proposition \ref{prop:introduction linear invariance}.

When $x\neq 0$, by Proposition \ref{prop:special-cartier}, two special cycles $\CZ(x)_S$ and $\CY(x)_S$ define classes $[\CZ(x)_S]$ resp. $[\CY(x)_S]$ in the Grothendieck group $K_0'(S)$.
We define $[\CZ(0)_S]$, resp. $[\CY(0)_S]$, as the classes in $K_0'(S)$ of the following complex of $\CO_S$-modules
\begin{equation*}
	(\cdots \rightarrow 0\lr\omega_{\CZ}^{-1}\overset{0}{\lr}\CO_{S}\lr 0),
	\quad\text{resp.}\quad
	(\cdots \lr 0\lr\omega_{\CY}^{-1}\overset{0}{\lr}\CO_{S}\lr 0),
\end{equation*}
cf. \cite{Howard2019}.

\begin{corollary}\label{cor:linear-invar}
	Let $x_1,\cdots,x_r\in \BV$ and $y_1,\cdots y_r\in\BV$ generate the same $O_F$-submodule. When $S$ is regular, 
\begin{equation*}
	[\CZ(x_1)_S]\cdot [\CZ(x_2)_S]\ldots \cdot [\CZ(x_r)_S]=[\CZ(y_1)_S]\cdot [\CZ(y_2)_S]\ldots \cdot [\CZ(y_r)_S]\in K'_0(S).
\end{equation*}
Similarly, we have
\begin{equation*}
	[\CY(x_1)_S]\cdot [\CY(x_2)_S]\ldots \cdot [\CY(x_r)_S]=[\CY(y_1)_S]\cdot [\CY(y_2)_S]\ldots \cdot [\CY(y_r)_S]\in K'_0(S).
\end{equation*}
\end{corollary}
\begin{proof}
This is a purely homological algebra argument. By comparing with \cite[\S 5]{Howard2019}, we see that all we need to know are the following facts, which all hold in our situation:

\noindent(1) $\CZ(x)_S\subset S$ is a Cartier divisor when $x\neq 0$. Hence locally $\CZ(x)$ can be cut by a single equation.

\noindent(2) We can locally choose sections of a line bundle as obstructions to the first-order infinitesimal deformation, so that we can apply \cite[Lemma 5.2]{Howard2019}, which is a purely commutative algebra statement.

\noindent(3) $S$ is regular, this is used in \cite[(5.7)]{Howard2019}.
\end{proof}
\begin{remark}
The importance of including additional filtrations  
$\CF^0_i\subset\CF_i$ for both $i=k$ and $i=n-k$ in the naive splitting model becomes evident now.
This is because the proof of Proposition \ref{prop:special-cartier} relies on the line bundle of modular forms for both $\CZ$-cycles and $\CY$-cycles. Namely, $\omega_{\CZ}$ is derived from the $\CF_{n-k}^0$-filtration and $\omega_{\CY}$ is derived from the $\CF_k^0$-filtration, as outlined in Definition \ref{def:line-bundle-modular-form}.
\end{remark}

\bibliographystyle{alpha}
\bibliography{reference}

@inproceedings{Kramer2003,
  title={Local models for Ramified unitary groups},
  author={Kr{\"a}mer, Nicole},
  booktitle={Abhandlungen aus dem Mathematischen Seminar der Universit{\"a}t Hamburg},
  volume={73},
  pages={67--80},
  year={2003},
  organization={Springer}
}

@article{Pappas2000,
  title={{On the arithmetic moduli schemes of PEL Shimura varieties}},
  author={Pappas, George},
  journal={Journal of Algebraic Geometry},
  volume={9},
  number={3},
  pages={577},
  year={2000},
  publisher={Providence, RI: University Press, c1992-}
}

@article{PR2005,
author = {George Pappas and Michael Rapoport},
title = {{Local models in the ramified case, II: Splitting models}},
volume = {127},
journal = {Duke Mathematical Journal},
number = {2},
publisher = {Duke University Press},
pages = {193 -- 250},
year = {2005},
doi = {10.1215/S0012-7094-04-12721-6},
URL = {https://doi.org/10.1215/S0012-7094-04-12721-6}
}

@article{PR2009, 
author={Pappas, George and Rapoport, Michael},
title={Local models in the ramified case. {III} {U}nitary groups}, 
volume={8}, 
DOI={10.1017/S1474748009000139}, 
number={3}, 
journal={Journal of the Institute of Mathematics of Jussieu}, 
publisher={Cambridge University Press},  
year={2009}, 
pages={507–564}
}

@phdthesis{Yu2019,
  title={On Moduli Description of Local Models For Ramified Unitary Groups and Resolution of Singularity},
  author={Yu, Si},
  year={2019},
  school={Johns Hopkins University}
}

@article {LUO2024,
    AUTHOR = {Luo, Yu},
     TITLE = {On the moduli description of ramified unitary local models of
              signature {$(n-1,1)$}},
   JOURNAL = {Math. Ann.},
  FJOURNAL = {Mathematische Annalen},
    VOLUME = {392},
      YEAR = {2025},
    NUMBER = {4},
     PAGES = {4661--4738},
      ISSN = {0025-5831,1432-1807},
   MRCLASS = {14G35 (11G18)},
  MRNUMBER = {4958488},
MRREVIEWER = {Matthew\ Dawes},
       DOI = {10.1007/s00208-025-03194-7},
       URL = {https://doi-org.libproxy.mit.edu/10.1007/s00208-025-03194-7},
}

@article{Smithling2015, 
title={On the Moduli Description of Local Models for Ramified Unitary Groups}, 
volume={2015}, 
DOI={https://doi.org/10.1093/imrn/rnv095}, 
number={24}, 
journal={ International Mathematics Research Notices}, 
publisher={Oxford Academic}, 
author={Smithling, Brian}, 
year={2015}, 
pages={13493–13532}
}

@misc{stackproject,
    shorthand    = {Stacks},
    author       = {The {Stacks Project Authors}},
    title        = {\textit{Stacks Project}},
    howpublished = {\url{https://stacks.math.columbia.edu}},
  }

@article{Shi2022,
  title={{Special cycles on unitary Shimura curves at ramified primes}},
  author={Shi, Yousheng},
  journal={{M}anuscripta mathematica},
  pages={1--70},
  year={2022},
  publisher={Springer}
}

@article {bijakowski2024geometry,
    AUTHOR = {Bijakowski, St\'ephane and Hernandez, Valentin},
     TITLE = {On the geometry of the {P}appas-{R}apoport models in the
              ({AR}) case},
   JOURNAL = {Pacific J. Math.},
  FJOURNAL = {Pacific Journal of Mathematics},
    VOLUME = {334},
      YEAR = {2025},
    NUMBER = {1},
     PAGES = {107--142},
      ISSN = {0030-8730,1945-5844},
   MRCLASS = {14G35 (11G18 14D24)},
  MRNUMBER = {4857516},
MRREVIEWER = {Rolf\ Berndt},
       DOI = {10.2140/pjm.2025.334.107},
       URL = {https://doi-org.libproxy.mit.edu/10.2140/pjm.2025.334.107},
}

@article {BHKRY,
    AUTHOR = {Bruinier, Jan H. and Howard, Benjamin and Kudla, Stephen S.
              and Rapoport, Michael and Yang, Tonghai},
     TITLE = {Modularity of generating series of divisors on unitary
              {S}himura varieties},
   JOURNAL = {Ast\'{e}risque},
  FJOURNAL = {Ast\'{e}risque},
    VOLUME = {421},
      YEAR = {2020},
     PAGES = {7--125},
      ISSN = {0303-1179,2492-5926},
      ISBN = {978-2-85629-927-2},
   MRCLASS = {14G35 (11F27 11F55 11G18 14G40)},
  MRNUMBER = {4183376},
MRREVIEWER = {Haowu\ Wang},
       DOI = {10.24033/ast},
       URL = {https://doi.org/10.24033/ast},
}

@article {HLSY,
    AUTHOR = {He, Qiao and Li, Chao and Shi, Yousheng and Yang, Tonghai},
     TITLE = {A proof of the {K}udla-{R}apoport conjecture for {K}r\"{a}mer
              models},
   JOURNAL = {Invent. Math.},
  FJOURNAL = {Inventiones Mathematicae},
    VOLUME = {234},
      YEAR = {2023},
    NUMBER = {2},
     PAGES = {721--817},
      ISSN = {0020-9910,1432-1297},
   MRCLASS = {11G18 (11G15 14G35)},
  MRNUMBER = {4651010},
       DOI = {10.1007/s00222-023-01209-1},
       URL = {https://doi.org/10.1007/s00222-023-01209-1},
}

@article {Howard2019,
    AUTHOR = {Howard, Benjamin},
     TITLE = {Linear invariance of intersections on unitary
              {R}apoport-{Z}ink spaces},
   JOURNAL = {Forum Math.},
  FJOURNAL = {Forum Mathematicum},
    VOLUME = {31},
      YEAR = {2019},
    NUMBER = {5},
     PAGES = {1265--1281},
      ISSN = {0933-7741,1435-5337},
   MRCLASS = {11G18 (14G35)},
  MRNUMBER = {4000587},
MRREVIEWER = {Rolf\ Berndt},
       DOI = {10.1515/forum-2019-0023},
       URL = {https://doi.org/10.1515/forum-2019-0023},
}

@article {Bijakowski-Hernandez-2023,
    AUTHOR = {Bijakowski, St\'{e}phane and Hernandez, Valentin},
     TITLE = {On the geometry of the {P}appas-{R}apoport models for {PEL}
              {S}himura varieties},
   JOURNAL = {J. Inst. Math. Jussieu},
  FJOURNAL = {Journal of the Institute of Mathematics of Jussieu. JIMJ.
              Journal de l'Institut de Math\'{e}matiques de Jussieu},
    VOLUME = {22},
      YEAR = {2023},
    NUMBER = {5},
     PAGES = {2403--2445},
      ISSN = {1474-7480,1475-3030},
   MRCLASS = {11G18 (14G35 14L05)},
  MRNUMBER = {4624966},
       DOI = {10.1017/S1474748022000019},
       URL = {https://doi.org/10.1017/S1474748022000019},
}

@article {shen2023fzips,
    AUTHOR = {Shen, Xu and Zheng, Yuqiang},
     TITLE = {{$F$}-zips with additional structure on splitting models of
              {S}himura varieties},
   JOURNAL = {Forum Math. Sigma},
  FJOURNAL = {Forum of Mathematics. Sigma},
    VOLUME = {13},
      YEAR = {2025},
     PAGES = {Paper No. e142, 65},
      ISSN = {2050-5094},
   MRCLASS = {11G18 (14G35)},
  MRNUMBER = {4958720},
MRREVIEWER = {Rolf\ Berndt},
       DOI = {10.1017/fms.2025.10091},
       URL = {https://doi-org.libproxy.mit.edu/10.1017/fms.2025.10091},
}

@phdthesis{zhao-thesis,
  title={Affine {Grassmannians and splitting models for triality groups}},
  author={Zhihao Zhao},
  year={2021},
  school={Michigan State University}
}

@article {zachos2023semistable,
    AUTHOR = {Zachos, Ioannis and Zhao, Zhihao},
     TITLE = {Semi-stable and splitting models for unitary {S}himura
              varieties over ramified places. {I}.},
   JOURNAL = {Forum Math. Sigma},
  FJOURNAL = {Forum of Mathematics. Sigma},
    VOLUME = {13},
      YEAR = {2025},
     PAGES = {Paper No. e119, 35},
      ISSN = {2050-5094},
   MRCLASS = {14G35 (11G18)},
  MRNUMBER = {4941593},
       DOI = {10.1017/fms.2025.10079},
       URL = {https://doi-org.libproxy.mit.edu/10.1017/fms.2025.10079},
}

@article {Reduzzi-Xiao,
    AUTHOR = {Reduzzi, Davide A. and Xiao, Liang},
     TITLE = {Partial {H}asse invariants on splitting models of {H}ilbert
              modular varieties},
   JOURNAL = {Ann. Sci. \'{E}c. Norm. Sup\'{e}r. (4)},
  FJOURNAL = {Annales Scientifiques de l'\'{E}cole Normale Sup\'{e}rieure.
              Quatri\`eme S\'{e}rie},
    VOLUME = {50},
      YEAR = {2017},
    NUMBER = {3},
     PAGES = {579--607},
      ISSN = {0012-9593,1873-2151},
   MRCLASS = {11F41 (14G35)},
  MRNUMBER = {3665551},
MRREVIEWER = {Shuichiro\ Takeda},
       DOI = {10.24033/asens.2328},
       URL = {https://doi.org/10.24033/asens.2328},
}

@article {Sasaki,
    AUTHOR = {Sasaki, Shu},
     TITLE = {Integral models of {H}ilbert modular varieties in the ramified
              case, deformations of modular {G}alois representations, and
              weight one forms},
   JOURNAL = {Invent. Math.},
  FJOURNAL = {Inventiones Mathematicae},
    VOLUME = {215},
      YEAR = {2019},
    NUMBER = {1},
     PAGES = {171--264},
      ISSN = {0020-9910,1432-1297},
   MRCLASS = {11F80 (11F33 11F41 14G22 14G35)},
  MRNUMBER = {3904451},
MRREVIEWER = {Claus\ Mazanti\ Sorensen},
       DOI = {10.1007/s00222-018-0825-x},
       URL = {https://doi.org/10.1007/s00222-018-0825-x},
}

@article {Diamond-Kassaei,
    AUTHOR = {Diamond, Fred and Kassaei, Payman L.},
     TITLE = {The cone of minimal weights for {M}od {$p$} {H}ilbert modular
              forms},
   JOURNAL = {Int. Math. Res. Not. IMRN},
  FJOURNAL = {International Mathematics Research Notices. IMRN},
      YEAR = {2023},
    VOLUME = {14},
     PAGES = {12148--12171},
      ISSN = {1073-7928,1687-0247},
   MRCLASS = {11F41 (11F33 14G35)},
  MRNUMBER = {4615228},
       DOI = {10.1093/imrn/rnac121},
       URL = {https://doi.org/10.1093/imrn/rnac121},
}

@article{zachos2024semistable,
    author = {Zachos, Ioannis and Zhao, Zhihao},
    title = {Semi-stable and splitting Models for unitary {S}himura varieties Over Ramified Places. {II}.},
    journal = {International Mathematics Research Notices},
    volume = {2025},
    number = {11},
    pages = {rnaf145},
    year = {2025},
    issn = {1073-7928},
    doi = {10.1093/imrn/rnaf145},
    url = {https://doi.org/10.1093/imrn/rnaf145},
}

@article {LL2,
    AUTHOR = {Li, Chao and Liu, Yifeng},
     TITLE = {Chow groups and {$L$}-derivatives of automorphic motives for
              unitary groups, {II}.},
   JOURNAL = {Forum Math. Pi},
  FJOURNAL = {Forum of Mathematics. Pi},
    VOLUME = {10},
      YEAR = {2022},
     PAGES = {Paper No. e5, 71},
      ISSN = {2050-5086},
   MRCLASS = {11G18 (11G40 11G50 14C15 14G35)},
  MRNUMBER = {4390300},
MRREVIEWER = {Stefano\ Vigni},
       DOI = {10.1017/fmp.2022.2},
       URL = {https://doi-org.libproxy.mit.edu/10.1017/fmp.2022.2},
}

@book {RZ,
    AUTHOR = {Rapoport, M. and Zink, Th.},
     TITLE = {Period spaces for {$p$}-divisible groups},
    SERIES = {Annals of Mathematics Studies},
    VOLUME = {141},
 PUBLISHER = {Princeton University Press, Princeton, NJ},
      YEAR = {1996},
     PAGES = {xxii+324},
      ISBN = {0-691-02782-X; 0-691-02781-1},
   MRCLASS = {14G20 (11G18 14F30 14L05 14M15 20G05 20G25)},
  MRNUMBER = {1393439},
MRREVIEWER = {Robert\ E.\ Kottwitz},
       DOI = {10.1515/9781400882601},
       URL = {https://doi-org.libproxy.mit.edu/10.1515/9781400882601},
}

@article {RSZ18,
    AUTHOR = {Rapoport, Michael and Smithling, Brian and Zhang, Wei},
     TITLE = {Regular formal moduli spaces and arithmetic transfer
              conjectures},
   JOURNAL = {Math. Ann.},
  FJOURNAL = {Mathematische Annalen},
    VOLUME = {370},
      YEAR = {2018},
    NUMBER = {3-4},
     PAGES = {1079--1175},
      ISSN = {0025-5831,1432-1807},
   MRCLASS = {11G18 (14G17 22E55)},
  MRNUMBER = {3770164},
MRREVIEWER = {Rolf\ Berndt},
       DOI = {10.1007/s00208-017-1526-2},
       URL = {https://doi.org/10.1007/s00208-017-1526-2},
}

@article {Mihatsch_2022,
    AUTHOR = {Mihatsch, Andreas},
     TITLE = {Relative unitary {RZ}-spaces and the arithmetic fundamental
              lemma},
   JOURNAL = {J. Inst. Math. Jussieu},
  FJOURNAL = {Journal of the Institute of Mathematics of Jussieu. JIMJ.
              Journal de l'Institut de Math\'ematiques de Jussieu},
    VOLUME = {21},
      YEAR = {2022},
    NUMBER = {1},
     PAGES = {241--301},
      ISSN = {1474-7480,1475-3030},
   MRCLASS = {11G18 (14G35)},
  MRNUMBER = {4366338},
MRREVIEWER = {Wen-Wei\ Li},
       DOI = {10.1017/S1474748020000079},
       URL = {https://doi-org.libproxy.mit.edu/10.1017/S1474748020000079},
}

@article{Arzdorf,
author = {Kai Arzdorf},
title = {{On local models with special parahoric level structure}},
volume = {58},
journal = {Michigan Mathematical Journal},
number = {3},
publisher = {University of Michigan, Department of Mathematics},
pages = {683 -- 710},
year = {2009},
doi = {10.1307/mmj/1260475695},
URL = {https://doi.org/10.1307/mmj/1260475695}
}

@article {kudla2023padicuniformizationunitaryshimura,
    AUTHOR = {Kudla, Stephen and Rapoport, Michael and Zink, Thomas},
     TITLE = {On the {$p$}-adic uniformization of unitary {S}himura curves},
   JOURNAL = {M\'em. Soc. Math. Fr. (N.S.)},
  FJOURNAL = {M\'emoires de la Soci\'et\'e{} Math\'ematique de France.
              Nouvelle S\'erie},
    NUMBER = {183},
      YEAR = {2024},
     PAGES = {vi+212},
      ISSN = {0249-633X,2275-3230},
      ISBN = {978-2-37905-204-0},
   MRCLASS = {14G35},
  MRNUMBER = {4815435},
}

@article{LMZ,
      title={Unitary {S}himura varieties at ramified primes and arithmetic transfer}, 
      author={Yu Luo and Andreas Mihatsch and Zhiyu Zhang},
      year={2025},
      eprint={2504.17484},
      archivePrefix={arXiv},
      primaryClass={math.AG},
      url={https://arxiv.org/abs/2504.17484}, 
}

@article {Ahsendorf-Cheng-Zink,
    AUTHOR = {Ahsendorf, Tobias and Cheng, Chuangxun and Zink, Thomas},
     TITLE = {{$\mathcal{O}$}-displays and {$\pi$}-divisible formal {$\mathcal{O}$}-modules},
   JOURNAL = {J. Algebra},
  FJOURNAL = {Journal of Algebra},
    VOLUME = {457},
      YEAR = {2016},
     PAGES = {129--193},
      ISSN = {0021-8693,1090-266X},
   MRCLASS = {14L05 (14F30 14L15)},
  MRNUMBER = {3490080},
MRREVIEWER = {Rui\ Miguel\ Saramago},
       DOI = {10.1016/j.jalgebra.2016.03.002},
       URL = {https://doi.org/10.1016/j.jalgebra.2016.03.002},
}

@article{BZZ25,
      title={On the geometry of splitting models}, 
      author={S. Bijakowski and I. Zachos and Z. Zhao},
      year={2025},
      eprint={2501.05950},
      archivePrefix={arXiv},
      primaryClass={math.NT},
      url={https://arxiv.org/abs/2501.05950}, 
}

@article {Gora-thesis,
    AUTHOR = {Gora, Felix},
     TITLE = {Local models, {M}ustafin varieties and semi-stable
              resolutions},
   JOURNAL = {M\"unster J. Math.},
  FJOURNAL = {M\"unster Journal of Mathematics},
    VOLUME = {16},
      YEAR = {2023},
    NUMBER = {2},
     PAGES = {265--300},
      ISSN = {1867-5778,1867-5786},
   MRCLASS = {14M15 (14G35)},
  MRNUMBER = {4824824},
}

@article {Gortz,
    AUTHOR = {G\"ortz, Ulrich},
     TITLE = {On the flatness of models of certain {S}himura varieties of
              {PEL}-type},
   JOURNAL = {Math. Ann.},
  FJOURNAL = {Mathematische Annalen},
    VOLUME = {321},
      YEAR = {2001},
    NUMBER = {3},
     PAGES = {689--727},
      ISSN = {0025-5831,1432-1807},
   MRCLASS = {14G35 (11G18)},
  MRNUMBER = {1871975},
MRREVIEWER = {Conjeeveram\ S.\ Rajan},
       DOI = {10.1007/s002080100250},
       URL = {https://doi.org/10.1007/s002080100250},
}

@article{HLS-basic,
      title={The basic locus of ramified unitary {S}himura varieties of signature $(n-1,1)$ at maximal vertex level}, 
      author={Qiao He and Yu Luo and Yousheng Shi},
      year={2026},
      eprint={2502.06218},
      archivePrefix={arXiv},
      primaryClass={math.AG},
      url={https://arxiv.org/abs/2502.06218}, 
}

@book {Hartshorne,
    AUTHOR = {Hartshorne, Robin},
     TITLE = {Algebraic geometry},
    SERIES = {Graduate Texts in Mathematics},
    VOLUME = {No. 52},
 PUBLISHER = {Springer-Verlag, New York-Heidelberg},
      YEAR = {1977},
     PAGES = {xvi+496},
      ISBN = {0-387-90244-9},
   MRCLASS = {14-01},
  MRNUMBER = {463157},
MRREVIEWER = {Robert\ Speiser},
}

@article {WZhang21,
    AUTHOR = {Zhang, W.},
     TITLE = {Weil representation and arithmetic fundamental lemma},
   JOURNAL = {Ann. of Math. (2)},
  FJOURNAL = {Annals of Mathematics. Second Series},
    VOLUME = {193},
      YEAR = {2021},
    NUMBER = {3},
     PAGES = {863--978},
      ISSN = {0003-486X,1939-8980},
   MRCLASS = {11F27 (11F67 11G40 14C25 14G35)},
  MRNUMBER = {4250392},
MRREVIEWER = {Rolf\ Berndt},
       DOI = {10.4007/annals.2021.193.3.5},
       URL = {https://doi.org/10.4007/annals.2021.193.3.5},
}

@article{LRZ25,
      title={More regular formal moduli spaces and arithmetic transfer conjectures: the ramified quadratic case}, 
      author={Yu Luo and Michael Rapoport and Wei Zhang},
      year={2025},
      eprint={2507.01395},
      archivePrefix={arXiv},
      primaryClass={math.NT},
      url={https://arxiv.org/abs/2507.01395}, 
}

@article {Richarz-master,
    AUTHOR = {Richarz, Timo},
     TITLE = {Schubert varieties in twisted affine flag varieties and local
              models},
   JOURNAL = {J. Algebra},
  FJOURNAL = {Journal of Algebra},
    VOLUME = {375},
      YEAR = {2013},
     PAGES = {121--147},
      ISSN = {0021-8693,1090-266X},
   MRCLASS = {14M15},
  MRNUMBER = {2998951},
MRREVIEWER = {Leonardo\ Constantin\ Mihalcea},
       DOI = {10.1016/j.jalgebra.2012.11.013},
       URL = {https://doi.org/10.1016/j.jalgebra.2012.11.013},
}

@book {Fargues08,
    AUTHOR = {Fargues, Laurent and Genestier, Alain and Lafforgue, Vincent},
     TITLE = {L'isomorphisme entre les tours de {L}ubin-{T}ate et de
              {D}rinfeld},
    SERIES = {Progress in Mathematics},
    VOLUME = {262},
 PUBLISHER = {Birkh\"auser Verlag, Basel},
      YEAR = {2008},
     PAGES = {xxii+406},
      ISBN = {978-3-7643-8455-5},
   MRCLASS = {14L05 (11R39 14-06 14G22 14G35)},
  MRNUMBER = {2441311},
MRREVIEWER = {Eva\ Viehmann},
       DOI = {10.1007/978-3-7643-8456-2},
       URL = {https://doi-org.ezproxy.library.wisc.edu/10.1007/978-3-7643-8456-2},
}

@article{Zachos--Zhao_regular-basic,
      title={The basic locus of regular ramified unitary Rapoport-Zink spaces at vertex-stabilizer level}, 
      author={Ioannis Zachos and Zhihao Zhao},
      year={2025},
      eprint={2511.05576},
      archivePrefix={arXiv},
      primaryClass={math.NT},
      url={https://arxiv.org/abs/2511.05576}, 
}

@article{LSR,
      title={Automorphic vector bundles on {S}himura varieties}, 
      author={Yu Luo and Yousheng Shi and Michael Rapoport},
      year={2026},
      note={In prepration}, 
}

@article {HSY,
    AUTHOR = {He, Qiao and Shi, Yousheng and Yang, Tonghai},
     TITLE = {Kudla-{R}apoport conjecture for {K}r\"amer models},
   JOURNAL = {Compos. Math.},
  FJOURNAL = {Compositio Mathematica},
    VOLUME = {159},
      YEAR = {2023},
    NUMBER = {8},
     PAGES = {1673--1740},
      ISSN = {0010-437X,1570-5846},
   MRCLASS = {11G18 (14G35 14G40)},
  MRNUMBER = {4609752},
MRREVIEWER = {Jean\ Kieffer},
       DOI = {10.1112/s0010437x23007273},
       URL = {https://doi-org.libproxy.mit.edu/10.1112/s0010437x23007273},
}
\end{document}